\documentclass[reqno]{amsart}
\usepackage{amssymb,latexsym}
\usepackage{amsmath,color}
\usepackage{amsthm}
\usepackage{epsfig}
\usepackage{graphicx}
\usepackage{hyperref}
\usepackage{titletoc}
\usepackage{palatino,mathpazo}
\numberwithin{equation}{section}
\newtheorem{theorem}{Theorem}[section]
\newtheorem{proposition}[theorem]{Proposition}
\newtheorem{lemma}[theorem]{Lemma}

\theoremstyle{definition}

\newtheorem{remark}{Remark}[section]

\def\XXint#1#2#3{{\setbox0=\hbox{$#1{#2#3}{\int}$}
		\vcenter{\hbox{$#2#3$}}\kern-.5\wd0}}

\def\la{\lambda}

\def\e{\varepsilon}

\def\R{\mathbb{R}^4}

\def\R{\mathbb{R}}

\def\XXint#1#2#3{{\setbox0=\hbox{$#1{#2#3}{\int}$ }
		\vcenter{\hbox{$#2#3$ }}\kern-.6\wd0}}

\def\hk{\hat{K}}

\begin{document}

\title[KP-I lump solution to travelling waves of GP equation]{From KP-I lump solution to travelling waves of Gross-Pitaevskii equation}
	
\author[Y. Liu]{Yong Liu}
\address{Yong Liu, Department of Mathematics,  University of Science and Technology of China, Hefei, China}
\email{yliumath@ustc.edu.cn}

\author[Z.P. Wang]{Zhengping Wang}
\address{Zhengping Wang, Department of Mathematics, Wuhan University of Technology, Wuhan, Hubei, China}
\email{zpwang@whut.edu.cn}
	
\author[J.C. Wei]{Jun-cheng Wei}	
\address{Juncheng Wei, Department of Mathematics, University of British Columbia, Vancouver, B.C., Canada, V6T 1Z2}
\email{jcwei@math.ubc.ca}
	
\author[W. Yang]{Wen Yang}
\address{\noindent Wen ~Yang,~Wuhan Institute of Physics and Mathematics, Innovation Academy for Precision Measurement Science and Technology, Chinese Academy of Sciences, Wuhan 430071, P. R. China.}
\email{wyang@wipm.ac.cn}

\begin{abstract}
Let $q(x,y)$ be an nondegenerate lump solution to KP-I (Kadomtsev-Petviashvili-I) equation
$$\partial_x^4q-2\sqrt{2}\partial_x^2q-3\sqrt{2}\partial_x((\partial_xq)
^2)-2\partial_y^2q=0.
$$
We prove the existence of a traveling wave solution $ u_{\e} (x-ct, y)$ to
GP (Gross-Pitaevskii) equation
$$
i\partial_{t}\Psi+\Delta\Psi+(1-|\Psi|^{2})\Psi=0,\ \ \ \mbox{in} \ {\mathbb R}^2 $$
in the transonic limit
$$ c=\sqrt{2}-\epsilon^2 $$
with
$$ u_\epsilon =1 + i \epsilon q(x,y) + {\mathcal O} (\epsilon^2). $$
This proves the existence of finite energy solutions in the so-called   Jones-Roberts program  in the transonic range $ c \in (\sqrt{2}-\epsilon^2, \sqrt{2})$. The main ingredients in our proof are detailed point-wise estimates of the Green function associated to a family of fourth order hypoelliptic operators
$$\partial_x^4-(2\sqrt{2}-\e^2)\partial_x^2-2\partial_y^2+\e^2\partial_x^2\partial_y^2+\e^4\partial_y^4.$$
\end{abstract}
\maketitle

\section{Introduction and statement of the main result}
The GP (Gross-Pitaevskii) equation investigated in this paper has the form
\begin{equation}
	i\partial_{t}\Psi+\Delta\Psi+(1-|\Psi|^{2})\Psi=0,\ \ \ \mbox{in} \ {\mathbb R}^2 \label{GP}%
\end{equation}
where $\Psi$ is a complex-valued function in $ {\mathbb R}^2$. It is formally Hamiltonian and can
be written as
\[
\partial_{t}\Psi=-i\nabla E,
\]
where $E$ is the Hamiltonian energy given by%
\[
E\left(  \Psi\right)  :=\int_{ {\mathbb R}^2} \left[  \frac{1}{2}\left\vert \nabla
\Psi\right\vert ^{2}+\frac{1}{4}\left(  \left\vert \Psi\right\vert
^{2}-1\right)  ^{2}\right]  ,
\]
and $\nabla E:=-\Delta\Psi-(1-|\Psi|^{2})\Psi$ represents the $L^{2}$-gradient
of $E.$ In this paper, we are interested in the travelling wave solutions of   \eqref{GP}. More precisely, let us consider solutions of  \eqref{GP} of the form $\Phi\left(  x-ct,y\right)  .$ Then
$\Phi\left(  x,y\right)  $ satisfies the following elliptic equation
\begin{equation}
	ic\partial_{x}\Phi=\Delta\Phi+\left(  1-\left\vert \Phi\right\vert
	^{2}\right)  \Phi \ \ \ \mbox{in} \ {\mathbb R}^2. \label{GPe}%
\end{equation}
Since $\Phi$ is complex valued, actually \eqref{GPe} is a
system consisting of two elliptic equations.

Travelling wave solutions are expected to play central role in the long time
dynamics of the GP equation. One of the main issue in the study of \eqref{GP} is the existence and classification of solutions to the
equation \eqref{GPe}. We would like to mention that some types of classification results and universal bound of certain classes of solutions, including travelling wave solutions, to the GP equation has been obtained in \cite{Farina}. Reversing the time if necessary, we may
assume that the travelling speed $c$ is nonnegative. Let us focus on the
equation \eqref{GPe} imposed in the entire space
$\mathbb{R}^{n}.$ Without any assumption on the asymptotic behavior at
infinity, the structure of the solution set of \eqref{GPe}
could be complicated. Observe that the energy $E$ is conserved the GP flow. We
therefore focus on those finite energy solutions of \eqref{GPe}. It follows from results of \cite{Gra1,Gra3} that any finite energy solution
will satisfy, up to a multiplicative constant,
\[
\Phi\left(  z\right)  \rightarrow1,\text{ as }\left\vert z\right\vert
\rightarrow+\infty.
\]
More precise asymptotic behavior of finite energy solutions is also available.
Indeed, by results of \cite{Gra4}, for $c\in\left(
0,\sqrt{2}\right)  ,$ there holds%
\[
\left\vert z\right\vert \left(  \Phi\left(  z\right)  -1\right)  -if\left(\frac{z}{|z|}\right)  \rightarrow0,\text{ as }z\rightarrow
+\infty,
\]
where $f$ is a smooth function defined on $S^{n-1},$ the $n-1$-dimensional
unit sphere. Particularly, when dimension $n=2,$ there exists constants
$\alpha,\beta,$ such that
\begin{equation}
	f\left( \frac{z}{|z|}\right)  =\frac{\left(  \alpha x+\beta
		y\right)  \sqrt{x^{2}+y^{2}}}{_{x^{2}+\left(  1-\frac{c^{2}}{2}\right)  y^{2}%
	}}. \label{asy}%
\end{equation}

From the formula  \eqref{asy} we notice that the
existence or nonexistece of nontrivial travelling wave solutions heavily depends on the
speed $c.$ For $c>\sqrt{2},$ it is proved by Gravejat \cite{Gravejat} that any
finite energy solution to  \eqref{GPe}  has to be a constant.

In the case of $c=\sqrt{2},$ it is now known \cite{Gra2} that there does not
exist nonconstant finite energy solutions when $n=2.$ The problem is still
open for $n\geq3.$

In the subsonic regime, that is, when $c\in\left(  0,\sqrt{2}\right)  ,$
Jones et.al \cite{J,J1} studied the equation from the physical
point of view and obtained solutions with formal and numerical calculation.
Later on, many works have been done to get the solutions in a more rigorous mathematical way. The methods applied in these works can be divided into two
categories. The first one is variational, and the second one is
Lyapunov-Schmidt reduction.

Variationally, there are several different strategies to tackle the problem.
Mountain pass arguments were used in \cite{Ber3,Chiron} to prove the existence
of solutions with small travelling speed. On the other hand, minimizing the
energy functional with fixed momentum yields existence in dimension 2 for any
constrained value of the momentum and existence in dimension 3 when the
momentum $p>p_{0}$ for some threshold value $p_{0},$ see \cite{Ber1,Ber2}.
In spite of all the above results, the question remains as for whether for all
$c\in\left(  0,\sqrt{2}\right)  $, there is a solution. Minimizing the energy
under a Pohozaev constraint, Maris \cite{Mar} successfully obtained solutions in
the full speed interval $\left(  0,\sqrt{2}\right)  $ for dimension $n>2.$
Unfortunately, this argument breaks down in 2D, thus leaving the existence
problem open in this dimension. Recently, Bellazzini-Ruiz \cite{Ruiz} proved
the existence of almost all subsonic speed in 2D, using variational argument and Struwe's monotonicity trick. They also recovered the results of Maris in 3D.

Up to now, the variational arguments mentioned above have not been able to
tackle the higher energy solutions. For these type of solutions, when the
speed $c$ is close to zero, in \cite{Liu2,Liu-Wei}, the second and third authors applied
Lyapunov-Schmidt reduction method to show the existence by gluing suitable
copies of the vortex solutions of the Ginzburg-Landau equation. In dimension
two, the position of the vortices is determined by the Adler-Moser
polynomials; in higher dimensions, the position is determined by a family
polynomials, which we called generalized Adler-Moser polynomials. We also
refer to \cite{ChironP} for the discussion of two vortices case.

\medskip

When the speed $c$ tends  to the transonic limit $\sqrt{2}$,  it is now known that  a suitable rescaled travelling waves
will converge to solutions of the KP-I equation
\begin{equation}
\label{KPIeqn}
\partial_x^4q-2\sqrt{2}\partial_x^2q-3\sqrt{2}\partial_x((\partial_xq)
^2)-2\partial_y^2q=0,
\end{equation}
 which is an important
integrable system. (See Section 2 below for the derivation.) We refer to Bethuel, Gravejet and Saut \cite{Ber0} for convergence result in dimension two, and Chiron and Maris \cite{Chiron2} for convergence results  in dimension three or higher. (See also \cite{Chiron2010,Chiron3,Chiron20142} for detailed discussion and related results.)   Numerical evidence of these results can be found in
\cite{Chiron3,J}. An important open question is whether the converse is true, i.e., given a solution to the KP-I equation (\ref{KPIeqn}), whether or not a travelling wave solution to (\ref{GPe}) close to the KP-I solution  exists.  For the KP-I equation, lump solutions have been obtained by variational methods in \cite{Saut,Saut2,Liuyue}. We also point out that the existence of least energy solutions for $c$
close to $\sqrt{2}$ can be obtained by variational arguments. However, since
we lack a clear classification of solutions to the KP-I equation, the
precise form of these \textquotedblleft transonic" solutions are not known. On
the other hand, although the Lyapunov-Schmidt reduction type perturbation
arguments have the disadvantage that it can not be used for the general case $c\in\left(  0,\sqrt{2}\right)$, it is possible to handle the case
of $c$ close to $\sqrt{2}.$ This will be carried out in this paper and our main result is the following

\begin{theorem}
\label{app}
Let $ q (x,y)$ be an nondegenerate lump solution to the following KP-I equation (\ref{KPIeqn}). For any $\varepsilon>0$ sufficiently small, there exists a
solution $\Phi_{\varepsilon}$ to the equation \eqref{GPe}
with travelling speed $c=\sqrt{2}-\varepsilon^{2}$ and has the following asymptotic behavior
\begin{equation*}
\Phi_\e=1+i\e q (x,y)+{\mathcal O}(\e^2).
\end{equation*}
\end{theorem}

\begin{remark} {\em Lump solutions to KP-I are a kind of rational function solutions, localized and decaying in all directions.  A lump solution to KP-I is called nondegenerate if the corresponding linearized operator admits only translational kernels. For the standard lump solution to (\ref{KPIeqn})
\begin{equation}
\label{lump1}
q_0  (x,y)==-\frac{2\sqrt{2}x}{x^2+\sqrt{2}y^2+\frac{3}{2\sqrt{2}}}
 \end{equation}
it was proved in \cite{Liu-Wei-2} that $q_0$ is nondegenerate and has Morse index one. (We note that the Morse index of the lump solution (\ref{lump1}) has been shown numerically to be one (\cite{Chiron3}).)  There are many higher Morse index lump solutions to KP-I and their nondegeneracy are expected (but rigorous proofs are needed). See \cite{Ablowitz, Ma, Manakov1, Manakov2} and the references therein.}
\end{remark}


In the following let us sketch the main idea and explain the new ingredients of our proof. First of all, we assume our solution has the form $f+ig$, where $f,~g$ are real valued functions with $f-1=O\left(  \varepsilon^{2}\right)  $ and $g=O\left(
\varepsilon\right).$ Plugging the function into \eqref{GP}, we shall see that the original problem is connected with the KP-I equation in the second leading order. Particularly, $f$ should essentially be determined by $g$, which satisfies a perturbed KP-I equation. However, we can not just simply build up a genuine solution by the nondegeneracy of the lump solution and the standard Lyapunov-Schmidt reduction method. The crucial point is that the perturbation terms would create some difficulties in obtaining the suitable decay property and studying the nonlinear problem. Therefore, we have to consider the full linearized problem including the higher order derivatives on $y$ direction, but with small coefficients
\begin{equation}
\label{4thanisotropic}
\partial_x^4-(2\sqrt{2}-\e^2)\partial_x^2-2\partial_y^2+\e^2\partial_x^2\partial_y^2+\e^4\partial_y^4.
\end{equation}
The $L^p$ boundedness of this anisotropic 4-th order elliptic operator has been derived and used in the convergence results in \cite{Ber0, Chiron2}.
However, in order to capture the decay property of the perturbations, we need to obtain {\em point-wise} behavior of the Green's function associated with  the anisotropic 4-th order operator (\ref{4thanisotropic}).  Precisely, since the maximum principle can not be used to show the decay property for the fourth order anisotropic problem, we have to study the Green kernel and its pointwise derivatives by describing their asymptotic behaviors around the singular point and decay property at infinity instead. (It is also interesting to mention that for the kernel of the linear operator $\partial_x^4-(2\sqrt{2}-\e^2)\partial_x^2-2\partial_y^2$, several important properties including the decay property and the asymptotic behavior around the singularity on the kernel $K_0$ have been obtained  by Gravejat in \cite{Gra4}.)

In most situations, the Green function would decay faster if we take more derivatives on it. However, the situation becomes quite different here, this is due to the fact that the Fourier integration of the Fourier inverse transform of the Green function  has singularity in our problem. As a consequence, the higher derivatives of the Green function can decay at most $r^{-\frac32}$ at infinity. Simultaneously, as the original problem is an anisotropic elliptic operator, we would get some huge coefficients on the derivatives of the Green function with respect to the $y$ direction, see Theorem \ref{tha.1} in section 3. All the above difficulties make the
whole problem quite sensitive, it would break down the Lyapunov-Schmidt reduction process easily in the nonlinear problem if we do not pursue the accurate asymptotic behaviors on its various derivatives. To solve this issue, we have to study the linear problem in a suitable space, see the definition of space in \eqref{3.def-*}. The definition of this space turns out to be quite complicated, but it seems to us that most of the terms are necessary. Its norm is consisted of three parts, the first part $\|\cdot\|_a$ is used for studying the $L^2$ theory of the linearized problem, the second part (contain various mixed derivatives) is included to deal with the terms appeared in the error and nonlinear higher order perturbations. Precisely, consider the term $\partial_x\phi$ for example, we shall see that it appears in the error terms, such as $\partial_y(\partial_xg_1\partial_yg_1)$, see \eqref{2.dec-p2} in section 2. When we deal with this term, we need at least one of $\partial_x\phi$ and $\partial_y\phi$ has decay up to $r^{-\frac32}$ at infinity, while concerning the linear problem in weighted Sobolev space, we require $\partial_x\phi$ to be small, but sacrificing a little bit in its decay rate. Therefore we have to include both situations in the definition of $\|\cdot\|_*$. The third part, the anti-derivative terms are used for making the perturbations of $f$ have enough decay to run the fixed point argument. After establishing the a-priori estimate for the full linearized problem, we study the existence of the nonlinear problem by contraction principle. Then the original result is established.  We remark that the arguments presented in this paper is quite robust for treating other models with the finite Morse index solutions to (\ref{KPIeqn}).

The paper is organized as follows. In Section \ref{GPequation} we derive the
equations satisfied by $f$ and $g.$ In Section 3, we study the Green function associated to the linear operator. In section 5, we analyzed the perturbed linearized KP-I operator around the lump solution. In Section 5, we solve the nonlinear
problem using contraction mapping principle and finish the proof of Theorem \ref{app}.
\medskip

\begin{center}
{\bf Notation}
\end{center}

Throughout the paper, the letter C will stand for positive constants which are allowed to vary among different formulas and even within the same lines. The letter $\varepsilon>0$ will represent a small parameter.

\section{KP-I as a consistent condition for GP\label{GPequation}}

We are seeking for travelling wave solutions $\Phi$ in the transonic regime,
that is, with speed close to $\sqrt{2}.$ In this section, we compute the
equation satisfied by the real and imaginary parts of $\Phi.$ Formal
computation has already been done in \cite{J}. We point out that in
\cite{Ber0,Chiron3}, the transonic limit has been investigated in great
details. The approach taken there is writing the solution $\Phi$ in the form
$\rho e^{i\phi}$ and analyzing the equations satisfied by $\rho$ and $\phi.$
Here we study the real and imaginary parts directly, because as we shall see
later on, in this form, the real part can be determined from the imaginary
part through an ODE, which is technically relatively easier.

Let us write $\Phi$ as $f+ig,$ where $f=\operatorname{Re}\Phi$ and
$g=\operatorname{Im}\Phi$. Then the travelling wave equation  \eqref{GPe}  is
equivalent to the following system
\[
\left\{
\begin{array}
	[c]{c}%
	c\partial_{\tilde{x}}g=-\Delta f+\left(  f^{2}+g^{2}-1\right)  f,\\
	-c\partial_{\tilde{x}}f=-\Delta g+\left(  f^{2}+g^{2}-1\right)  g.
\end{array}
\right.
\]
Introducing new variables by
\[
x=\varepsilon \tilde{x},y=\varepsilon^2\tilde{y},
\]
we obtain
\begin{equation}
	\left\{
	\begin{array}
		[c]{c}%
		c\varepsilon\partial_{x}g=-\varepsilon^{4}\partial_{y}^{2}f-\varepsilon
		^{2}\partial_{x}^{2}f+\left(  f^{2}+g^{2}-1\right)  f,\\
		\\
		-c\varepsilon\partial_{x}f=-\varepsilon^{4}\partial_{y}^{2}g-\varepsilon
		^{2}\partial_{x}^{2}g+\left(  f^{2}+g^{2}-1\right)  g.
	\end{array}
	\right.  \label{fg}%
\end{equation}
We assume that the travelling speed $c$ has the form $c=c_{0}-\varepsilon^{2},$
where $c_{0}$ is a constant independent of $\varepsilon.$

We look for $f$ and $g$ in the form
\[
f=1+\varepsilon^{2}f_{1}+\varepsilon^{4}f_{2},\text{ }g=\varepsilon g_{1},
\]
where $f_{1},f_{2},g_{1}$ are unknown functions to be determined. Substituting
this into the system $\left(  \ref{fg}\right)  $ we get the following two
equations to be solved:%
\begin{equation}%
	\begin{aligned}
		\left(  c_{0}-\varepsilon^{2}\right)  \partial_{x}g_{1}=&
		\left[  2\left(  f_{1}+\varepsilon^{2}f_{2}\right)  +g_{1}^{2}+\left(
		\varepsilon f_{1}+\varepsilon^{3}f_{2}\right)  ^{2}\right]  \left(
		1+\varepsilon^{2}f_{1}+\varepsilon^{4}f_{2}\right) \\
		&-\varepsilon
		^{4}\left(  \partial_{y}^{2}f_{1}+\varepsilon^{2}\partial_{y}^{2}f_{2}\right)
		-\left(  \varepsilon^{2}\partial_{x}^{2}f_{1}+\varepsilon^{4}\partial_{x}%
		^{2}f_{2}\right),
	\end{aligned}
	\label{equa1}%
\end{equation}
and%
\begin{equation}%
	\begin{aligned}
		-\left(  c_{0}-\varepsilon^{2}\right)  \left(  \partial_{x}f_{1}%
		+\varepsilon^{2}\partial_{x}f_{2}\right)  =&\left[  2\left(  f_{1}+\varepsilon^{2}f_{2}\right)  +g_{1}^{2}+\left(
		\varepsilon f_{1}+\varepsilon^{3}f_{2}\right)  ^{2}\right]  g_{1}\\
		&-\varepsilon^{2}\partial_{y}%
		^{2}g_{1}-\partial_{x}^{2}g_{1}.
	\end{aligned}
	\label{equa2}%
\end{equation}
To solve these two equations, we will expand them into powers of $\varepsilon
$.

\subsection{The $O\left(  1\right)  $ terms }

Comparing the $O\left(  1\right)  $ terms in   \eqref{equa1} ,
we get
\begin{equation}
	c_{0}\partial_{x}g_{1}=2f_{1}+g_{1}^{2}.\label{f1g1-1}%
\end{equation}
On the other hand, comparing the $O\left(  1\right)  $ term in
\eqref{equa2} leads to
\begin{equation}
	-c_{0}\partial_{x}f_{1}=-\partial_{x}^{2}g_{1}+g_{1}\left(  2f_{1}+g_{1}%
	^{2}\right)  .\label{f1g1-2}%
\end{equation}
The consistency of the two equations   \eqref{f1g1-1}  and
 \eqref{f1g1-2} requires that
\[
-\frac{c_{0}}{2}\partial_{x}\left(  c_0\partial_{x}g_{1}-g_{1}^{2}\right)  =-\partial_{x}^{2}g_{1}+c_{0}g_1\partial_{x}g_{1}.
\]
Hence it is natural to set $c_{0}=\sqrt{2}.$ Equation   \eqref{f1g1-1}  then has the form
\begin{equation}
	f_{1}=\frac{\sqrt{2}}{2}\partial_{x}g_{1}-\frac{g_{1}^{2}}{2}.\label{f1g1-3}%
\end{equation}
This means that $f_{1}$ is essentially determined by $g_{1}.$

\subsection{The $o\left(  1\right)  $ terms and the consistent condition}

With the $O\left(  1\right)  $ terms being understood, we now divide the
$o\left(  1\right)  $ terms in the equation \eqref{equa1} by
$\varepsilon^{2}.$ This yields
\begin{align*}
	-\partial_{x}g_{1}= &  -\varepsilon^{2}\left( \partial_{y}^{2}f_{1}%
	+\varepsilon^{2}\partial_{y}^{2}f_{2}\right)  -\left(  \partial_{x}^{2}%
	f_{1}+\varepsilon^{2}\partial_{x}^{2}f_{2}\right)  +2f_{2}+\left(
	f_{1}+\varepsilon^{2}f_{2}\right)  ^{2}\nonumber\\
	&  +\left(  f_{1}+\varepsilon^{2}f_{2}\right)  \left\{  2f_{1}+g_{1}%
	^{2}+2\varepsilon^{2}f_{2}+\left(  \varepsilon f_{1}+\varepsilon^{3}%
	f_{2}\right)  ^{2}\right\}  .\label{f2}%
\end{align*}
The above equation can be written into the following more compact form:
\begin{equation}
	-\partial_{x}g_{1}=-\partial_{x}^{2}f_{1}+f_{1}\left(  2f_{1}+g_{1}%
	^{2}\right)  +2f_{2}+f_{1}^{2}+\Pi_{1},\label{Per1}%
\end{equation}
where the perturbation term $\Pi_{1}$ is defined by
\begin{equation}
\label{Pi-1}
\begin{aligned}
\Pi_1=-\e^2\partial_y^2f_1-\e^4\partial_y^2f_2-\e^2\partial_x^2f_2+6\e^2f_1f_2+\e^2(f_1+\e^2f_2)^3+3\e^4f_2^2+\e^2f_2g_1^2.	
\end{aligned}
\end{equation}
For convenience, we write the last four terms as
\begin{equation}
\label{Pi-2}
\Pi_2=6\e^2f_1f_2+\e^2(f_1+\e^2f_2)^3+3\e^4f_2^2+\e^2f_2g_1^2.
\end{equation}

To proceed, we define the function
\begin{equation}
	\Xi:=2f_{2}+f_{1}^{2}+\sqrt{2}\left(  \partial_{x}\left(  f_{1}g_{1}\right)
	+g_{1}^{2}\partial_{x}g_{1}\right)  .\label{Xi}%
\end{equation}
This definition is inspired by equations (A13) and (A14) of \cite{J}. Note
that $\Xi$ involves $f_{1},g_{1}$ and $f_{2},$ we can represent $\Xi$ in the following way

\begin{lemma}
	\label{first1}Suppose  \eqref{f1g1-3} and
	\eqref{Per1}  hold. Then we have%
	\[
	\Xi=-\partial_{x}g_{1}+\frac{1}{\sqrt{2}}\partial_{x}^{3}g_{1}-\left(
	\partial_{x}g_{1}\right)  ^{2}-\Pi_{1}.
	\]
	
\end{lemma}

\begin{proof}
	Eliminating $f_{2}$ in  \eqref{Xi} by using equation
	\eqref{Per1} , we get%
	\[
	\Xi=-\partial_{x}g_{1}+\partial_{x}^{2}f_{1}-f_{1}\left(  2f_{1}+g_{1}%
	^{2}\right)  +\sqrt{2}\left(  \partial_{x}\left(  f_{1}g_{1}\right)
	+g_{1}^{2}\partial_{x}g_{1}\right) -\Pi_{1}.
	\]
	Inserting  \eqref{f1g1-1} into the right hand side, we
	compute%
	\begin{align*}
		\Xi  =& -\partial_{x}g_{1}+\partial_{x}^{2}\left(  \frac{\partial_{x}g_{1}%
		}{\sqrt{2}}-\frac{g_{1}^{2}}{2}\right)  -\sqrt{2}\left(  \frac{\partial
			_{x}g_{1}}{\sqrt{2}}-\frac{g_{1}^{2}}{2}\right)  \partial_{x}g_{1}\\
		&  +\sqrt{2}\left( \partial_{x}\left(  \frac{g_{1}\partial_{x}g_{1}}{\sqrt
			{2}}-\frac{g_{1}^{3}}{2}\right)  +g_{1}^{2}\partial_{x}g_{1}\right)  -\Pi_{1}\\
		 =& -\partial_{x}g_{1}+\frac{1}{\sqrt{2}}\partial_{x}^{3}g_{1}-\frac{1}%
		{2}\partial_{x}^{2}\left(  g_{1}^{2}\right)  -\left(  \partial_{x}%
		g_{1}\right)  ^{2}+\frac{1}{\sqrt{2}}g_{1}^{2}\partial_{x}g_{1}\\
		&  +\partial_{x}\left(  g_{1}\partial_{x}g_{1}\right)  -\frac{1}{\sqrt{2}%
		}\partial_{x}\left(  g_{1}^{3}\right)  +\sqrt{2}g_{1}^{2}\partial_{x}%
		g_{1}-\Pi_{1}.
	\end{align*}
	Hence
	\begin{equation}
		\Xi=-\partial_{x}g_{1}+\frac{1}{\sqrt{2}}\partial_{x}^{3}g_{1}-\left(
		\partial_{x}g_{1}\right)  ^{2}-\Pi_{1}.\label{Fi}%
	\end{equation}
	This finishes the proof.
\end{proof}

Next let us divide the $o\left(  1\right)  $ terms in  \eqref{equa2} by $\varepsilon^{2}.$ This leads to
\begin{equation}
\label{Per2}
-(\sqrt{2}-\e^2)\partial_xf_2=-\partial_xf_1-\partial_y^2g_1+g_1(2f_2+f_1^2)+\Pi_3,
\end{equation}
where the perturbation term $\Pi_3$ is given by
\begin{equation}
\label{Pi-3}
\Pi_3=2\e^2f_1f_2g_1+\e^4g_1f_2^2.
\end{equation}
We shall use this information to compute $\partial_{x}\Xi.$ We have
the following

\begin{lemma}
Suppose \eqref{f1g1-3}, \eqref{Per1} , \eqref{Per2} hold. Then
\begin{equation*}
\begin{aligned}
\partial_x\Xi=~&\partial_x^2g_1\left(\frac{\sqrt{2}}{\sqrt{2}-\e^2}+4\partial_xg_1\right)+\frac{2}{\sqrt{2}-\e^2}\left(\partial_y^2g_1+g_1\Pi_1-\Pi_3\right)\\
&+\frac{\e^2}{2-\sqrt{2}\e^2}\partial_xg_1\partial_x(g_1^2)-\frac{\e^2}{2\sqrt{2}-2\e^2}g_1^2\partial_x(g_1^2)-\frac{\sqrt{2}\e^2}{\sqrt{2}-\e^2}g_1\partial_x^2f_1.	
\end{aligned}
\end{equation*}	
\end{lemma}

\begin{proof}
First of all, using  \eqref{Per2} to eliminate $\partial_{x}f_{2},$ we have%
\begin{align*}
\partial_x\Xi
=~&\partial_x\left(2f_2+f_1^2+\sqrt{2}(\partial_x(f_1g_1)+g_1^2\partial_xg_1)\right)\\
=~&\frac{2}{\sqrt{2}-\e^2}\left(\partial_xf_1+\partial_y^2g_1-g_1(2f_2+f_1^2)-\Pi_3\right)\\&
+\partial_x\left(f_1^2+\sqrt{2}(\partial_x(f_1g_1)+g_1^2\partial_xg_1)\right).
\end{align*}
Using equation   \eqref{Per1},  we can eliminate $f_{2}$
appeared on the right hand side. It follows that%
\begin{align*}
\partial_{x}\Xi   =~&\frac{2}{\sqrt{2}-\e^2}\left(\partial_{x}f_{1}
+\partial_{y}^{2}g_1\right)%
-\frac{2}{\sqrt{2}-\e^2}g_{1}\left(  -\partial_{x}g_{1}+\partial_{x}^{2}f_{1}%
-f_{1}\left(  2f_{1}+g_{1}^{2}\right)\right)  \\
&+\frac{2}{\sqrt{2}-\e^2}(g_1\Pi_1-\Pi_3)+\partial_{x}\left(  f_{1}^{2}+\sqrt{2}\left(\partial_{x}\left(  f_{1}g_{1}\right)  +g_{1}^{2}\partial_{x}g_{1}\right)  \right)  \\
=~&\frac{2}{\sqrt{2}-\e^2}\left(\partial_xf_1+\partial_y^2g_1+g_1\partial_xg_1-g_1\partial_x^2f_1+g_1f_1(2f_1+g_1^2)\right)\\
&+\frac{2}{\sqrt{2}-\e^2}(g_1\Pi_1-\Pi_3)+\partial_x(f_1^2)+\sqrt{2}\partial_x(g_1^2\partial_xg_1)\\
&+\sqrt{2}\left(g_1\partial_x^2f_1+2\partial_xg_1\partial_xf_1+\partial_x^2g_1f_1\right)\\
=~&\frac{2}{\sqrt{2}-\e^2}\left(\partial_xf_1+\partial_y^2g_1+g_1\partial_xg_1+g_1f_1(2f_1+g_1^2)\right)+\frac{2}{\sqrt{2}-\e^2}(g_1\Pi_1-\Pi_3)\\
&+\partial_x(f_1^2)
+\sqrt{2}(f_1\partial_x^2g_1+2\partial_xf_1\partial_xg_1)+\sqrt{2}\partial_x(g_1^2\partial_xg_1)-\frac{\sqrt{2}\e^2}{\sqrt{2}-\e^2}g_1\partial_x^2f_1.
\end{align*}
	To further simplify the expression, we use   \eqref{f1g1-3}  to
	compute
	\begin{align*}
		\partial_{x}\Xi  =~&\frac{2}{\sqrt{2}-\e^2}\left(\partial_xf_1+\partial_y^2g_1+g_1\partial_xg_1\right)+\sqrt{2}f_1\partial_x^2g_1+2\sqrt{2}\partial_xf_1\partial_xg_1+\partial_x(f_1^2)\\
		&+\frac{2\sqrt{2}}{\sqrt{2}-\e^2}f_1g_1\partial_xg_1
		+\frac{2}{\sqrt{2}-\e^2}(g_1\Pi_1-\Pi_3)+\sqrt{2}\partial_x(g_1^2\partial_xg_1)-\frac{\sqrt{2}\e^2}{\sqrt{2}-\e^2}g_1\partial_x^2f_1\\
		=~&\sqrt{2}\partial_xf_1\left(\frac{\sqrt{2}}{\sqrt{2}-\e^2}+2\partial_xg_1+\sqrt{2}f_1\right)+\frac{2}{\sqrt{2}-\e^2}\partial_y^2g_1+\frac{1}{\sqrt{2}-\e^2}\partial_x(g_1^2)+\sqrt{2}\partial_x(g_1^2\partial_xg_1)\\
		&+\sqrt{2}f_1\left(\partial_x^2g_1+\frac{2}{\sqrt{2}-\e^2}g_1\partial_xg_1\right)+\frac{2}{\sqrt{2}-\e^2}(g_1\Pi_1-\Pi_3)
		-\frac{\sqrt{2}\e^2}{\sqrt{2}-\e^2}g_1\partial_x^2f_1.
	\end{align*}
	Hence
	\begin{align*}
		\partial_{x}\Xi =~&\left(  \partial_{x}^{2}g_{1}-\frac{\sqrt{2}}{2}%
		\partial_{x}\left(  g_{1}^{2}\right)  \right)  \left(  \frac{\sqrt{2}}{\sqrt{2}-\e^2}+2\partial_{x}%
		g_{1}+\left(  \partial_{x}g_{1}-\frac{g_{1}^{2}}{\sqrt{2}}\right)  \right)\\
		&+\frac{2}{\sqrt{2}-\e^2}\left(\partial_{y}^{2}g_{1}+g_{1}\partial_{x}g_{1}\right)+\frac{2}{\sqrt{2}-\e^2}
		\left(g_1\Pi_1-\Pi_3\right)+\sqrt{2}\partial_{x}\left(  g_{1}^{2}\partial
		_{x}g_{1}\right)\\
		&  +\left(  \partial_{x}g_{1}-\frac{g_{1}^{2}}{\sqrt{2}}\right)  \left(
		\partial_{x}^{2}g_{1}+\frac{2}{\sqrt{2}-\e^2}g_{1}\partial_{x}g_{1}\right)-\frac{\sqrt{2}\e^2}{\sqrt{2}-\e^2}g_1\partial_x^2f_1\\
		=~&\partial_{x}^{2}g_{1}\left(\frac{\sqrt{2}}{\sqrt{2}-\e^2}+3\partial_xg_1-\frac{g_1^2}{\sqrt{2}}+\partial_xg_1-\frac{g_1^2}{\sqrt{2}}+\sqrt{2}g_1^2\right)+\frac{2}{\sqrt{2}-\e^2}\partial_y^2g_1\\
		&-\frac{\sqrt{2}}{2}\partial_x(g_1^2)\left(\frac{\sqrt{2}}{\sqrt{2}-\e^2}+3\partial_xg_1-\frac{g_1^2}{\sqrt{2}}-\frac{\sqrt{2}}{\sqrt{2}-\e^2}-\frac{\sqrt{2}}{\sqrt{2}-\e^2}(\partial_xg_1-\frac{g_1^2}{\sqrt{2}})-2\partial_xg_1\right)\\
	    &+\frac{2}{\sqrt{2}-\e^2}(g_1\Pi_1-\Pi_3)-\frac{\sqrt{2}\e^2}{\sqrt{2}-\e^2}g_1\partial_x^2f_1.
	\end{align*}
	Direct computation shows that
	\begin{equation}
	\label{dFi}
	\begin{aligned}
	\partial_{x}\Xi=~&\partial_x^2g_1\left(\frac{\sqrt{2}}{\sqrt{2}-\e^2}+4\partial_xg_1\right)+\frac{2}{\sqrt{2}-\e^2}\partial_y^2g_1+\frac{2}{\sqrt{2}-\e^2}(g_1\Pi_1-\Pi_3)\\
	&+\frac{\e^2}{2-\sqrt{2}\e^2}\partial_xg_1\partial_x(g_1^2)
	-\frac{\e^2}{2\sqrt{2}-2\e^2}g_1^2\partial_x(g_1^2)-\frac{\sqrt{2}\e^2}{\sqrt{2}-\e^2}g_1\partial_x^2f_1.
    \end{aligned}
    \end{equation}
	The proof is thus completed.
\end{proof}

Observe that we can also use Lemma \ref{first1} to compute $\partial_{x}\Xi.$
Consistency for  \eqref{Fi} and  \eqref{dFi}
requires that $g_{1}$ satisfies the following equation:%
\begin{align*}
&\partial_x\left(-\partial_xg_1+\frac{1}{\sqrt{2}}\partial_x^3g_1-(\partial_xg_1)^2-\Pi_1\right)\\
&=\partial_x^2g_1\left(\frac{\sqrt{2}}{\sqrt{2}-\e^2}+4\partial_xg_1\right)+\frac{2}{\sqrt{2}-\e^2}\partial_y^2g_1+\frac{2}{\sqrt{2}-\e^2}(g_1\Pi_1-\Pi_3)\\
&\quad+\frac{\e^2}{2-\sqrt{2}\e^2}\partial_xg_1\partial_x(g_1^2)
-\frac{\e^2}{2\sqrt{2}-2\e^2}g_1^2\partial_x(g_1^2)-\frac{\sqrt{2}\e^2}{\sqrt{2}-\e^2}g_1\partial_x^2f_1.
\end{align*}
It can be regarded as a perturbed KP-I equation:
\begin{equation}
\label{pKPI}
\begin{aligned}
&\frac{1}{\sqrt{2}}\partial_x^4g_1-\frac{2\sqrt{2}-\e^2}{\sqrt{2}-\e^2}\partial_x^2g_1-3\partial_x\left((\partial_xg_1)^2\right)-\frac{2}{\sqrt{2}-\e^2}\partial_y^2g_1\\
&=\partial_x\Pi_1+\frac{2}{\sqrt{2}-\e^2}(g_1\Pi_1-\Pi_3)+\frac{\e^2}{2-\sqrt{2}\e^2}\partial_xg_1\partial_x(g_1^2)\\
&\quad-\frac{\e^2}{2\sqrt{2}-2\e^2}g_1^2\partial_x(g_1^2)
-\frac{\sqrt{2}\e^2}{\sqrt{2}-\e^2}g_1\partial_x^2f_1.
\end{aligned}
\end{equation}
Substituting \eqref{Pi-1} and \eqref{Pi-2} into $\partial_x\Pi_1+\frac{2}{\sqrt{2}-\e^2}g_1\Pi_1$ we have
\begin{equation}
\label{2xp-p-1}
\begin{aligned}
\partial_x\Pi_1+\frac{2}{\sqrt{2}-\e^2}g_1\Pi_1
=~&-\e^2\partial_x\partial_y^2f_1-\e^4\partial_x\partial_y^2f_2-\e^2\partial_x^3f_2+\partial_x\Pi_2\\
&-\frac{2}{\sqrt{2}-\e^2}g_1\left(\e^2\partial_y^2f_1+\e^4\partial_y^2f_2+\e^2\partial_x^2f_2-\Pi_2\right).
\end{aligned}	
\end{equation}
Using \eqref{Per2}, we can further write \eqref{2xp-p-1} as
\begin{equation}
\label{2xp-p}
\begin{aligned}
&\partial_x\Pi_1+\frac{2}{\sqrt{2}-\e^2}g_1\Pi_1\\
&=-\e^2\partial_x\partial_y^2\left(\frac{\sqrt{2}}{2}\partial_xg_1-\frac{g_1^2}{2}\right)-\frac{2\e^2}{\sqrt{2}-\e^2}g_1\partial_y^2\left(\frac{\sqrt{2}}{2}\partial_xg_1-\frac{g_1^2}{2}\right)+\partial_x\Pi_2+\frac{2}{\sqrt{2}-\e^2}g_1\Pi_2\\	
&\quad-\frac{\e^4}{\sqrt{2}-\e^2}\partial_y^2\left(\partial_xf_1+\partial_y^2g_1-g_1f_1^2-\Pi_3\right)+\frac{2\e^4}{\sqrt{2}-\e^2}\left(\partial_y^2g_1f_2+2\partial_yg_1\partial_yf_2\right)\\
&\quad-\frac{\e^2}{\sqrt{2}-\e^2}\partial_x^2\left(\partial_xf_1+\partial_y^2g_1-g_1f_1^2-\Pi_3\right)+\frac{2\e^2}{\sqrt{2}-\e^2}\left(\partial_x^2g_1f_2+2\partial_xg_1\partial_xf_2\right)\\
&=-\frac{2}{\sqrt{2}-\e^2}\e^2\partial_x^2\partial_y^2g_1-\frac{\e^2}{2-\sqrt{2}\e^2}\partial_x^4g_1-\frac{\e^4}{\sqrt{2}-\e^2}\partial_y^4g_1+\Pi_4,		
\end{aligned}	
\end{equation}
where
\begin{equation}
\label{2.def-p}
\begin{aligned}
\Pi_4=~&\frac{\e^2}{2}\partial_x\partial_y^2(g_1^2)-\frac{\sqrt{2}\e^2}{\sqrt{2}-\e^2}g_1\partial_x\partial_y^2g_1+\frac{\e^2}{\sqrt{2}-\e^2}g_1\partial_y^2g_1^2+\frac{\e^4}{\sqrt{2}-\e^2}\partial_y^2(g_1\partial_xg_1)\\
&+\frac{\e^2}{\sqrt{2}-\e^2}\partial_x^2(g_1\partial_xg_1)+\frac{\e^4}{\sqrt{2}-\e^2}\partial_y^2(g_1f_1^2+\Pi_3)+\frac{\e^2}{\sqrt{2}-\e^2}\partial_x^2(g_1f_1^2+\Pi_3)\\
&+\frac{2\e^4}{\sqrt{2}-\e^2}\left(\partial_y^2g_1f_2+2\partial_yg_1\partial_yf_2\right)+\frac{2\e^2}{\sqrt{2}-\e^2}\left(\partial_x^2g_1f_2+2\partial_xg_1\partial_xf_2\right)\\
&+\partial_x\Pi_2+\frac{2}{\sqrt{2}-\e^2}g_1\Pi_2.	
\end{aligned}	
\end{equation}
Based on \eqref{2xp-p} and \eqref{2.def-p} we can
write \eqref{pKPI} as
\begin{equation}
\label{pKPI-1}
\begin{aligned}
&\partial_x^4g_1-(2\sqrt{2}-\e^2)\partial_x^2g_1-3(\sqrt{2}-\e^2)\partial_x((\partial_xg_1)^2)-2\partial_y^2g_1+2\e^2\partial_x^2\partial_y^2g_1+\e^4\partial_y^4g_1\\
&=(\sqrt{2}-\e^2)\Pi_4-2\Pi_3+\frac{1}{\sqrt{2}}\e^2\partial_xg_1\partial_x(g_1^2)-\frac{1}{2}\e^2g_1^2\partial_x(g_1^2)-\sqrt{2}\e^2g_1\partial_x^2f_1.
\end{aligned}
\end{equation}
We denote the right hand side of \eqref{pKPI-1} by $P$ and it can be written as
\begin{equation}
\label{2.def.p}
\begin{aligned}
P=~&\left(\frac{\sqrt{2}-\e^2}{2}\right)\e^2\partial_x\partial_y^2(g_1^2)-\sqrt{2}\e^2g_1\partial_x\partial_y^2g_1+\e^2g_1\partial_y^2(g_1^2)+\e^4\partial_y^2(g_1\partial_xg_1)\\
&+\e^2\partial_x^2(g_1\partial_xg_1)+\e^4\partial_y^2\left(g_1f_1^2+2\e^2f_1f_2g_1+\e^4g_1f_2^2\right)\\
&+\e^2\partial_x^2\left(g_1f_1^2+2\e^2f_1f_2g_1+\e^4g_1f_2^2\right)+2\e^4(\partial_y^2g_1f_2+2\partial_yg_1\partial_yf_2)\\
&+2\e^2(\partial_x^2g_1f_2+2\partial_xg_1\partial_xf_2)-2(2\e^2f_1f_2g_1+\e^4g_1f_2^2)
\\&+(\sqrt{2}-\e^2)\partial_x\left(6\e^2f_1f_2+\e^2(f_1+\e^2f_2)^3+3\e^4f_2^2+\e^2f_2g_1^2\right)\\
&+2g_1\left(6\e^2f_1f_2+\e^2(f_1+\e^2f_2)^3+3\e^4f_2^2+\e^2f_2g_1^2\right)\\
&+\frac{1}{\sqrt{2}}\e^2\partial_xg_1\partial_xg_1^2-\frac12\e^2g_1^2\partial_x(g_1^2)-\sqrt{2}\e^2g_1\partial_x^2f_1.
\end{aligned}
\end{equation}
Among the terms in \eqref{2.def.p}, we shall analyze the following four terms $g_1\partial_x\partial_y^2g_1$, $g_1\partial_x^2f_1$, $\partial_y^2g_1f_2+2\partial_yg_1\partial_yf_2$ and $\partial_x^2g_1f_2+2\partial_xg_1\partial_xf_2$. For the first one we could rewrite it as
\begin{equation}
\label{2.ap-1}
\begin{aligned}
g_1\partial_x\partial_y^2g_1=\partial_x(g_1\partial_y^2g_1)-\partial_y(\partial_xg_1\partial_yg_1)+\frac12\partial_x((\partial_yg_1)^2).
\end{aligned}	
\end{equation}
Using \eqref{f1g1-3}, we could write the second one as
\begin{equation}
\label{2.ap-2}
\begin{aligned}
g_1\partial_x^2f_1=\partial_x(g_1\partial_xf_1)
-\frac{\sqrt{2}}{4}\partial_x((\partial_xg_1)^2)+\frac12\partial_xg_1\partial_x(g_1^2).	
\end{aligned}	
\end{equation}
For the third term $\partial_y^2g_1f_2+2\partial_yg_1\partial_yf_2$, based on \eqref{f1g1-3} and \eqref{Per2} we obtain that
\begin{equation}
\label{2.ap-3}
\begin{aligned}
(\partial_y^2g_1f_2+2\partial_yg_1\partial_yf_2)=~&2\partial_y(\partial_yg_1f_2)-\partial_y^2g_1f_2\\
=~&2\partial_y(\partial_yg_1f_2)+\left(\partial_xf_1-(\sqrt{2}-\e^2)\partial_xf_2-g_1(2f_2+f_1^2)\right)f_2\\
&-(2\e^2f_1f_2g_1+\e^4g_1f_2^2)f_2\\
=~&2\partial_y(\partial_yg_1f_2)-\frac{\sqrt{2}-\e^2}{2}\partial_x(f_2^2)-g_1(2f_2+f_1^2)f_2\\
&-(2\e^2f_1f_2g_1+\e^4g_1f_2^2)f_2+\partial_x(f_1f_2)\\
&+\frac{1}{\sqrt{2}-\e^2}\left(-\partial_xf_1-\partial_y^2g_1+g_1(2f_2+f_1^2)\right)f_1\\
&+\frac{1}{\sqrt{2}-\e^2}(2\e^2f_1f_2g_1+\e^4g_1f_2^2)f_1.
\end{aligned}
\end{equation}
By \eqref{f1g1-3} we also notice that
\begin{equation*}
\begin{aligned}
\partial_y^2g_1f_1=~&\partial_y(\partial_yg_1f_1)-\partial_yg_1\left(\frac{1}{\sqrt{2}}\partial_x\partial_yg_1-\frac12\partial_y(g_1^2)\right)\\
=~&\partial_y(\partial_yg_1f_1)-\frac{1}{2\sqrt{2}}\partial_x((\partial_yg_1)^2)+\frac{1}{2}\partial_yg_1\partial_y(g_1^2).
\end{aligned}
\end{equation*}
Then we can write \eqref{2.ap-3} as
\begin{equation}
\label{2.ap-4}
\begin{aligned}
(\partial_y^2g_1f_2+2\partial_yg_1\partial_yf_2)
=~&2\partial_y(\partial_yg_1f_2)-\frac{\sqrt{2}-\e^2}{2}\partial_x(f_2^2)+\partial_x(f_1f_2)-\frac{1}{2\sqrt{2}-2\e^2}\partial_xf_1^2\\
&-\frac{1}{\sqrt{2}-\e^2}\partial_y(\partial_yg_1f_1)
+\frac{1}{4-2\sqrt{2}\e^2}\partial_x((\partial_yg_1)^2)\\
&-\frac{1}{2\sqrt{2}-2\e^2}\partial_yg_1\partial_y(g_1^2)-g_1(2f_2+f_1^2)f_2\\
&-(2\e^2f_1f_2g_1+\e^4g_1f_2^2)f_2+\frac{1}{\sqrt{2}-\e^2}g_1(2f_2+f_1^2)f_1\\
&+\frac{1}{\sqrt{2}-\e^2}(2\e^2f_1f_2g_1+\e^4g_1f_2^2)f_1.
\end{aligned}	
\end{equation}
Using \eqref{f1g1-3} and \eqref{Per2} again we rewrite $\partial_x^2g_1f_2+2\partial_xg_1\partial_xf_2$ as
\begin{equation}
\label{2.ap-5}
\begin{aligned}
(\partial_x^2g_1f_2+2\partial_xg_1\partial_xf_2)=~&
\partial_x(\partial_xg_1f_2)+\frac{1}{\sqrt{2}-\e^2}
\left(\partial_xf_1+\partial_y^2g_1-g_1(2f_2+f_1^2)\right)\partial_xg_1\\	
	&-\frac{1}{\sqrt{2}-\e^2}\partial_xg_1(2\e^2f_1f_2g_1+\e^4g_1f_2^2)\\
	=~&\partial_x(\partial_xg_1f_2)+\frac{1}{4-2\sqrt{2}\e^2}\partial_x((\partial_xg_1)^2)
	+\frac{1}{\sqrt{2}-\e^2}\partial_y(\partial_yg_1\partial_xg_1)\\
	&-\frac{1}{2\sqrt{2}-2\e^2}\partial_x((\partial_yg_1)^2)
	-\frac{1}{2(\sqrt{2}-\e^2)}\partial_x(g_1^2)\partial_xg_1\\
	&-\frac{1}{\sqrt{2}-\e^2}\left(g_1(2f_2+f_1^2)\partial_xg_1+\partial_xg_1(2\e^2f_1f_2g_1+\e^4g_1f_2^2)\right).
\end{aligned}	
\end{equation}
Substituting \eqref{2.ap-1}, \eqref{2.ap-2}, \eqref{2.ap-4} and \eqref{2.ap-5} into \eqref{2.def.p} we have
\begin{equation}
\label{2.dec-p}
P=P_1+P_2+P_3,	
\end{equation}
where
\begin{equation}
\label{2.dec-p1}
\begin{aligned}
P_1=~&\e^2\partial_x^2(g_1\partial_xg_1)+\e^2\partial_x^2(g_1f_1^2+2\e^2f_1f_2g_1+\e^4g_1f_2^2)-(\sqrt{2}-\e^2)\e^4\partial_x(f_2^2)\\
&-\frac{\e^4}{\sqrt{2}-\e^2}\partial_xf_1^2+2\e^2\partial_x(\partial_xg_1f_2)+\frac{2\e^2-\frac{\sqrt{2}}{2}\e^4}{2-\sqrt{2}\e^2}\partial_x(\partial_xg_1)^2\\
&+(\sqrt{2}-\e^2)\partial_x\left(6\e^2f_1f_2+\e^2(f_1+\e^2f_2)^3+3\e^4f_2^2+\e^2f_2g_1^2\right)\\
&+2\e^4\partial_x(f_1f_2)-\sqrt{2}\e^2\partial_x(g_1\partial_xf_1),
\end{aligned}	
\end{equation}
\begin{equation}
\label{2.dec-p2}
\begin{aligned}
P_2=~&\sqrt{2}\e^2\partial_y(\partial_xg_1\partial_yg_1)+\e^4\partial_y^2\left(g_1f_1^2+2\e^2f_1f_2g_1+\e^4g_1f_2^2\right)\\
&+4\e^4\partial_y(\partial_yg_1f_2)-\frac{2\e^4}{\sqrt{2}-\e^2}\partial_y(\partial_yg_1f_1)+\frac{2\e^2}{\sqrt{2}-\e^2}\partial_y(\partial_yg_1\partial_xg_1),
\end{aligned}	
\end{equation}
and
\begin{equation}
\label{2.dec-p3}
\begin{aligned}
P_3=~&\e^2g_1\partial_y^2(g_1^2)-\frac{\e^4}{\sqrt{2}-\e^2}\partial_yg_1\partial_y(g_1^2)-2\e^4g_1(2f_2+f_1^2)f_2\\
&-2\e^4(2\e^2f_1f_2g_1+\e^4g_1f_2^2)f_2+\frac{2\e^4}{\sqrt{2}-\e^2}g_1(2f_2+f_1^2)f_1\\
&+\frac{2\e^4}{\sqrt{2}-\e^2}(2\e^2f_1f_2g_1+\e^4g_1f_2^2)f_1-\frac{\e^2}{\sqrt{2}-\e^2}\partial_x(g_1^2)\partial_xg_1\\
&-\frac{2\e^2}{\sqrt{2}-\e^2}\left(g_1(2f_2+f_1^2)\partial_xg_1+\partial_xg_1(2\e^2f_1f_2g_1+\e^4g_1f_2^2)\right)\\
&+2g_1\left(6\e^2f_1f_2+\e^2(f_1+\e^2f_2)^3+3\e^4f_2^2+\e^2f_2g_1^2\right)\\
&-2(2\e^2f_1f_2g_1+\e^4g_1f_2^2)-\frac12\e^2g_1^2\partial_x(g_1^2).
\end{aligned}	
\end{equation}
With \eqref{2.dec-p}, and \eqref{2.dec-p1}-\eqref{2.dec-p3} we write \eqref{pKPI-1} as
\begin{equation}
\label{2.equation}
\begin{aligned}
&\partial_x^4g_1-(2\sqrt{2}-\e^2)\partial_x^2g_1-3(\sqrt{2}-\e^2)\partial_x((\partial_xg_1)^2)-2\partial_y^2g_1+2\e^2\partial_x^2\partial_y^2g_1+\e^4\partial_y^4g_1\\
&=P_1+P_2+P_3.
\end{aligned}	
\end{equation}
Observe that $P_1$ and $P_2$ can be regarded as
two functions $\partial_xH_1$ and $\partial_yH_2$ respectively for some proper functions $H_1$ and $H_2$, and we shall show that $P_3$ decays faster than $r^{-3}$ at infinity. The advantage of writing \eqref{pKPI-1} as \eqref{2.equation} is that we could derive a better decay for the perturbation function.

Concerning the left hand side of \eqref{2.equation}, we shall look for solution $g_1=q+\phi$, where $q$ is a lump solution of a rescaled KP-I
\begin{equation}
\label{lump}
\partial_x^4q-(2\sqrt{2}-\e^2)\partial_x^2q-3\sqrt{2}\left(\frac{2\sqrt{2}-\e^2}{2\sqrt{2}}\right)^{5/2}\partial_x((\partial_xq)^2)-2\partial_y^2q=0.	\end{equation}

(Note that the constant in the third term in (\ref{lump}) is slightly different from the constant in (\ref{2.equation}) because we also collect the first term in $P_1$, for the convenience  of proof.)

Actually, \eqref{lump} is closely related to one of KP-I equation after suitable rescaling.  In the rest of the proof we take the standard lump solution
\begin{equation}
\label{2.lump}
q_\e(x,y)=-\left(\frac{2\sqrt{2}}{2\sqrt{2}-\e^2}\right)^2\dfrac{\sqrt{8-2\sqrt{2}\e^2}x}{\frac{2\sqrt{2}-\e^2}{2\sqrt{2}}x^2+\frac{(2\sqrt{2}-\e^2)^2}{4\sqrt{2}}y^2+\frac{3}{2\sqrt{2}}},	\end{equation}
whose nondegeneracy is known (\cite{Liu-Wei2}). The proof works exactly the same for other nondegenerate lump solutions. The associated unique negative eigenvalue and eigenfunction  with respect to $q_\e$ are well approximated by the ones of the limit equation (sending $\e$ to $0$)
\begin{equation}
\label{lump-0}
\partial_x^4q-2\sqrt{2}\partial_x^2q-3\sqrt{2}\partial_x((\partial_xq)
^2)-2\partial_y^2q=0,
\end{equation}
where the corresponding solution is
\begin{equation*}
\label{lump-limit}
q_0(x,y)=-\frac{2\sqrt{2}x}{x^2+\sqrt{2}y^2+\frac{3}{2\sqrt{2}}}.
\end{equation*}
As we can see that the linearized problem of \eqref{2.equation} is well approximated by the one of KP-I equation. However, we can not just ignore these two terms. Indeed, we have seen that $\partial_y^2f_2$ appears in $P$ and it is related to $\partial_y^4g_1$ by \eqref{Per2}, so we need to get some control on $\partial_y^4g_1$ in showing the existence of $f_2$.  In next section we shall study the Green kernel of the linear operator associated to \eqref{2.equation}.

\section{The Green Function}
In this section, we shall study the Green kernel $G(x,y)$ of the differential operator (introduced in section 2),
\begin{equation}
\label{3.green-kernel}
\partial_x^4-(2\sqrt{2}-\e^2)\partial_x^2-2\partial_y^2+\e^2\partial_x^2\partial_y^2+\e^4\partial_y^4.
\end{equation}
To simplify our discussion, we set all the constant coefficients to be $1$ (after suitable scaling), i.e.,
$$\partial_x^4-\partial_x^2-\partial_y^2+\e^2\partial_x^2\partial_y^2+\e^4\partial_y^4.$$
It is not difficult to see that the associated Green kernel (denoted by $K(x,y)$) has similar asymptotic behavior at singularity and decay properties at infinity as $G(x,y)$.

The Fourier transform of $K(x,y)$ can be written as
\begin{equation*}
\label{a.1}
\hk(\xi_1,\xi_2)=\frac{1}{\xi_1^4+\xi_1^2+\xi_2^2+\e^2\xi_1^2\xi_2^2+\e^4\xi_2^4}.
\end{equation*}
The main purpose of this section is devoted to studying the decay property and asymptotic behavior of the derivatives of $K(x,y)$. Based on the elementary computation from Fourier analysis, we know that
\begin{equation*}
\widehat{\partial_x^m\partial_y^nK}(\xi_1,\xi_2)=(-i)^{m+n}\xi_1^m\xi_2^n\hk(\xi_1,\xi_2).
\end{equation*}
In order to study the asymptotic behavior of $\partial_x^m\partial_y^nK(x,y)$, it is enough to compute the following integration when $r=\sqrt{x^2+y^2}$ tends to $\infty$ and $0$,
\begin{equation}
\label{a.2}
K_{m,n}(x,y):=\int_{\R}\int_{\R}\dfrac{\xi_1^m\xi_2^ne^{ix\xi_1+iy\xi_2}}{\xi_1^4+\xi_1^2+\xi_2^2+\e^2\xi_1^2\xi_2^2+\e^4\xi_2^4}
d\xi_1d\xi_2.
\end{equation}

In our problem,  the behavior of $K_{m,n}(x,y)$ for the following cases are needed
$$(m,n)\in\left\{(1,0),(2,0),(3,0),(0,1),(0,2),(0,3),(1,1),(1,2)\right\}.$$
Now we shall state the corresponding properties for $K_{m,n}(x,y)$ in the following two lemmas. The first result is as follows

\begin{lemma}
\label{lea.1}
Let $K_{m,0}(x,y)$ be defined in \eqref{a.2} with $m=1,2,3$. Then for $r$ small we have
\begin{equation*}
r|K_{m,0}(x,y)|\leq
\begin{cases}
C,\quad &\mathrm{for}\quad m=1,\\
C\left(\min\left\{\log\frac{1+\sqrt{|y|}}{\sqrt{|y|}},\log\frac{1}{\e}\right\}+e^{-\frac{C}{\e^2}y}\right),\quad &\mathrm{for}\quad m=2,\\
C\left(\min\left\{\frac{1}{\e},\frac{1}{\sqrt{|y|}}\right\}+\frac{1}{\e}\max\left\{1,\log\frac{\e^2}{|y|}\right\}
e^{-\frac{C}{\e^2}|y|}\right),\quad &\mathrm{for}\quad m=3.
\end{cases}
\end{equation*}
While for $r$ large we have
\begin{equation*}
r|K_{m,0}(x,y)|\leq
\begin{cases}
C,\quad &\mathrm{for}\quad m=1,\\
\frac{C}{r^{\frac12}},\quad &\mathrm{for}\quad m=2,\\
\frac{C}{\e^{\frac12}r^{\frac12}}\quad &\mathrm{for}\quad m=3.
\end{cases}
\end{equation*}
\end{lemma}

\begin{proof}
By definition we have
\begin{equation}
\label{a.3}
K_{m,0}(x,y)=\int_{\R}\int_{\R}\dfrac{\xi_1^me^{ix\xi_1+iy\xi_2}}{\xi_1^4+\xi_1^2+\xi_2^2+\e^2\xi_1^2\xi_2^2+\e^4\xi_2^4}
d\xi_1d\xi_2.
\end{equation}
Concerning the denominator, we write it as
\begin{equation*}
\xi_1^4+\xi_1^2+\xi_2^2+\e^2\xi_1^2\xi_2^2+\e^4\xi_2^4
=\e^4(\xi_2^2+a(\xi_1))(\xi_2^2+b(\xi_1)),
\end{equation*}
where $a(\xi_1),~b(\xi_1)$ are functions of $\xi_1$ defined as:
\begin{equation*}
\begin{aligned}
a(\xi_1)=\frac{1+\e^2\xi_1^2-D(\xi_1)}{2\e^4},\quad b(\xi_1)=\frac{1+\e^2\xi_1^2+D(\xi_1)}{2\e^4},
\end{aligned}
\end{equation*}
and
\begin{equation*}
D(x):=\sqrt{(1+\e^2x)^2-4\e^4(x^2+x^4)}=\sqrt{1+2\e^2x^2-4\e^4x^2-3\e^4x^4}.
\end{equation*}
The function $D(x)$ can be decomposed by
\begin{equation*}
D(x)=\sqrt{-3\e^4(x^2-c_\e^2)(x^2+d_\e^2)},
\end{equation*}
where
\begin{equation}
\label{5.cd-asy}
\begin{aligned}
c_\e^2=&\frac{1-2\e^2+2\sqrt{1-\e^2+\e^4}}{3\e^2}=\frac{1}{\e^2}-1+O(\e^2)~\mathrm{as}~\e\to0,\\
d_\e^2=&\frac{-1+2\e^2+2\sqrt{1-\e^2+\e^4}}{3\e^2}=\frac{1}{3\e^2}+\frac{1}{3}+O(\e^2)~\mathrm{as}~\e\to0.
\end{aligned}
\end{equation}
Then we get that $D(x)$ admits two roots
\begin{equation*}
c_\e= \frac{1}{\e}-\frac{\e}{2}+O(\e^2)\quad \mbox{and}\quad
-c_\e=-\frac{1}{\e}+\frac{\e}{2}+O(\e^2)\quad \mbox{as}\quad \e\to0.	
\end{equation*}
Around the root points $c_\e$ and $-c_\e$, we have
\begin{equation}
\label{5.dasy}
D(x)=\begin{cases}
\sqrt{6\e^4c_\e(c_\e^2+d_\e^2)}\left((c_\e-x)^{\frac12}+O((x-c_\e)^{\frac32})\right)\quad &\mathrm{as}\quad x\to c_\e,\\
\sqrt{6\e^4c_\e(c_\e^2+d_\e^2)}\left((x+c_\e)^{\frac12}+O((x+c_\e)^{\frac32})\right)\quad &\mathrm{as}\quad x\to -c_\e.	
\end{cases}	
\end{equation}
It is known that $D(\xi)$ is positive when $|\xi|<c_\e$ and negative for $|\xi|>c_\e.$ We write the integral \eqref{a.3} as
\begin{equation}
\label{5-10}
K_{m,0}(x,y)=\int_{\R}\int_{\R}\frac{\xi_1^m}{D(\xi_1)}\left(\frac{e^{ix\xi_1+iy\xi_2}}{\xi_2^2+a(\xi_1)}-\frac{e^{ix\xi_1+iy\xi_2}}{\xi_2^2+b(\xi_1)}\right)d\xi_2d\xi_1.	
\end{equation}
We shall first study $K_{m,0}(x,y)$ for the case $y\neq0$. Without loss of generality, we may assume that $y>0$. From the above discussion on $D(x)$ we see that both $a(\xi_1)$ and $b(\xi_1)$ turns to be imaginary number if $|\xi_1|>c_\e$. In our calculations, for convenience we always assume that the real parts of $\sqrt{a(\xi_1)}$ and $\sqrt{b(\xi_1)}$ are positive. Then using the Residue Theorem we could write \eqref{5-10} as
\begin{equation}
\label{5.res}
K_{m,0}(x,y)=\pi\int_{\R}\frac{\xi^m}{D(\xi)}
\left(\frac{e^{ix\xi-\sqrt{a(\xi)}y}}{\sqrt{a(\xi)}}-\frac{e^{ix\xi-\sqrt{b(\xi)}y}}{\sqrt{b(\xi)}}\right)d\xi.
\end{equation}
To estimate the above integral \eqref{5.res}, we write
\begin{equation}
\label{5.dec}
\begin{aligned}	
&\frac{e^{ix\xi-\sqrt{a(\xi)}y}}{\sqrt{a(\xi)}}	
-\frac{e^{ix\xi-\sqrt{b(\xi)}y}}{\sqrt{b(\xi)}}\\
&=\frac{e^{ix\xi-\sqrt{a(\xi)}y}}{\sqrt{a(\xi)}}
-\frac{e^{ix\xi-\sqrt{a(\xi)}y}}{\sqrt{b(\xi)}}	
+\frac{e^{ix\xi-\sqrt{a(\xi)}y}}{\sqrt{b(\xi)}}
-\frac{e^{ix\xi-\sqrt{b(\xi)}y}}{\sqrt{b(\xi)}}\\
&=e^{ix\xi-\sqrt{a(\xi)}y}\left(\frac{b(\xi)-a(\xi)}{\sqrt{a(\xi)b(\xi)}(\sqrt{a(\xi)}+\sqrt{b(\xi)})}+\frac{1-e^{(\sqrt{a(\xi)}-\sqrt{b(\xi)})y}}{\sqrt{b(\xi)}}\right).
\end{aligned}
\end{equation}
Substituting \eqref{5.dec} into \eqref{5.res} and we decompose $K_{m,0}(x,y)$ into two parts
\begin{equation*}
\begin{aligned}
K_{m,0}(x,y)&=\int_{\R}\frac{\pi\xi^m}{D(\xi)}	
\frac{e^{ix\xi-\sqrt{a(\xi)}y}(b(\xi)-a(\xi))}{\sqrt{a(\xi)b(\xi)}(\sqrt{a(\xi)}+\sqrt{b(\xi)})}d\xi\\	
&\quad+\int_{\R}\frac{\pi\xi^m}{D(\xi)}\frac{e^{ix\xi-\sqrt{a(\xi)}y}}{\sqrt{b(\xi)}}\left(1-e^{(\sqrt{a(\xi)}-\sqrt{b(\xi)})y}\right)d\xi\\
&=:\pi(I_1+I_2).	
\end{aligned}
\end{equation*}
Concerning the integrand in $I_1$, we could write it as
\begin{equation*}
\begin{aligned}
&\frac{\xi^m}{D(\xi)}	
\frac{e^{ix\xi-\sqrt{a(\xi)}y}(b(\xi)-a(\xi))}{\sqrt{a(\xi)b(\xi)}(\sqrt{a(\xi)}+\sqrt{b(\xi)})}
=\frac{\xi^m e^{ix\xi-\sqrt{a(\xi)}y}}{\e^4\sqrt{b(\xi)}(\sqrt{a(\xi)b(\xi)}+a(\xi))}\\
&=e^{ix\xi-\sqrt{a(\xi)}y}
\frac{\xi^m\sqrt{1+\e^2\xi^2-D(\xi)}}{\sqrt{2}\e^2\xi^2(1+\xi^2)+\frac{\sqrt{2}}{2}
|\xi|\sqrt{1+\xi^2}
(1+\e^2\xi^2-D(\xi))}.	
\end{aligned}	
\end{equation*}
We set
\begin{equation}
\label{5.m1}
M_m(\xi):=\frac{\xi^m\sqrt{1+\e^2\xi^2-D(\xi)}}{\sqrt{2}\e^2\xi^2(1+\xi^2)+\frac{\sqrt{2}}{2}
|\xi|\sqrt{1+\xi^2}
(1+\e^2\xi^2-D(\xi))}.
\end{equation}
It is not difficult to check that
\begin{equation}
\label{5.mm-asy}
M_m(\xi)=sgn(\xi)\xi^{m-1}+O(\xi^{m})\quad \mathrm{as}\quad|\xi|\to0.	
\end{equation}
Therefore we can see that $M_m(\xi)$ are all bounded at $0$ and $M_2(\xi),~M_3(\xi)=0$ at $\xi=0$.

Next we introduce a cut-off function $\chi_\e(\xi)$, defined by
\begin{equation}
\label{5.def-cut}
\chi_\e(\xi)=\begin{cases}
1,\quad &\mbox{if}\quad |\xi|\in\left[\frac34c_\e,\frac32c_\e\right],\\
\\
0,\quad &\mbox{if}\quad |\xi|\in[0,+\infty)\setminus\left[\frac12c_\e,2c_\e\right].	
\end{cases}	
\end{equation}
We decompose $I_1$ as
\begin{equation}
\label{5.dec-i1}
\begin{aligned}
I_1=~&\int_{\R}(1-\chi_\e(\xi))e^{ix\xi-\sqrt{a(\xi)}y}
M_m(\xi)d\xi+	\int_{\R}\chi_\e(\xi)e^{ix\xi-\sqrt{a(\xi)}y}
M_m(\xi)d\xi\\
=~&I_{11}+I_{12}.
\end{aligned}	
\end{equation}
From \eqref{5.mm-asy} we can see that the integrand function of $I_{11}$ is discontinuous at $\xi=0$ for $m=1$, continuous for $m=2$ and differentiable for $m=3.$ Using integration by parts, we have
\begin{equation}
\label{5.i11}
\begin{aligned}
I_{11}=~&\lim_{\xi\to0^-}\frac{e^{ix\xi-\sqrt{a(\xi)}y}}{ix-\frac{a'(\xi)}{2\sqrt{a(\xi)}}y}M_m(\xi)-\lim_{\xi\to0^+}\frac{e^{ix\xi-\sqrt{a(\xi)}y}}{ix-\frac{a'(\xi)}{2\sqrt{a(\xi)}}y}M_m(\xi)\\
&-\int_{\R}e^{ix\xi-\sqrt{a(\xi)}y}\partial_\xi\left(
	\frac{(1-\chi_\e(\xi))M_m(\xi)}{ix-\frac{a'(\xi)}{2\sqrt{a(\xi)}}y}\right)d\xi\\
	=&~I_{111}+I_{112}+I_{113}.
\end{aligned}	
\end{equation}
For the right hand side of \eqref{5.i11}, we notice that the boundary terms vanish for $m=2,3$ and equals to $\frac{2ix}{x^2+y^2}$ for $m=1$. Then we have
\begin{equation}
\label{5.i11-bou}
\begin{aligned}
|I_{111}|+|I_{112}|\begin{cases}
\leq\frac{C}{r},\quad &\mbox{for}\quad m=1,\\
\\
=0, &\mbox{for}\quad m=2,3.
\end{cases}
\end{aligned}
\end{equation}
Concerning the integrand term of $I_{113}$, we have
\begin{equation*}
\begin{aligned}
\partial_\xi\left(\frac{(1-\chi_\e(\xi))M_m(\xi)}{ix-\frac{a'(\xi)}{2\sqrt{a(\xi)}}y}\right)
=~&\frac{\frac{2a''(\xi)a(\xi)-(a'(\xi))^2}{4(a(\xi))^{\frac32}}}{\left(ix-\frac{a'(\xi)}{2\sqrt{a(\xi)}}y\right)^2}y(1-\chi_\e(\xi))M_m(\xi)
-\frac{\chi'_\e(\xi)M_m(\xi)}{ix-\frac{a'(\xi)}{2\sqrt{a(\xi)}}y}\\
&+\frac{1-\chi_\e(\xi)}{ix-\frac{a'(\xi)}{2\sqrt{a(\xi)}}y}M_m'(\xi)\\=~&I_{1131}+I_{1132}+I_{1133}.	
\end{aligned}
\end{equation*}
By direct computation we could derive the following estimate
\begin{equation}
\label{5.est-i11}
\begin{aligned}
|I_{1131}|+|I_{1132}|+|I_{1133}|
\leq C\dfrac{|\xi|^{m-1}\min\left\{\e|\xi|,\e^2\xi^2\right\}}{\e^2(1+\xi^4)r}.
\end{aligned}
\end{equation}
By \eqref{5.est-i11} we get
\begin{equation}
\label{5.esti113}
\begin{aligned}
|I_{113}|\leq~&\frac{C}{r}\int_{-2c_\e}^{2c_\e}
\frac{e^{-\sqrt{a(\xi)}y}}{(1+|\xi|)^{3-m}}d\xi
+\frac{C}{r}\int_{\R\setminus[-2c_\e,2c_\e]}
\frac{1}{\e|\xi|^{4-m}}
e^{-\frac{C|\xi|}{\e}y}d\xi\\
\leq~&\begin{cases}
\frac{C}{r},\quad &\mbox{if}\quad m=1,\\
C\left(\min\left\{\log\frac{1+\sqrt{y}}{\sqrt{y}},\log\frac{1}{\e}\right\}+e^{-\frac{C}{\e^2}y}\right)\frac{1}{r}, \quad&\mbox{if}\quad m=2,\\
C\left(\min\left\{\frac{1}{\sqrt{y}},\frac{1}{\e}\right\}+\frac{1}{\e}\max\left\{1,\log\frac{\e^2}{y}\right\}e^{-\frac{C}{\e^2}y}\right)\frac{1}{r},\quad &\mbox{if}\quad m=3,
\end{cases}
\end{aligned}	
\end{equation}
where we used that
\begin{equation*}
\int_0^{2c_\e}\frac{1}{1+\xi}e^{-C\xi^2y}d\xi
\leq C\int_0^{2c_\e\sqrt{y}}\frac{1}{\xi+\sqrt{y}}e^{-C\xi^2}d\xi\leq C\min\left\{\log\frac{1+\sqrt{y}}{\sqrt{y}},\log\frac{1}{\e}\right\},	
\end{equation*}
\begin{equation*}
\int_0^{2c_\e}e^{-C\xi^2y}d\xi
\leq \frac{C}{\sqrt{y}}\int_0^{Cc_\e\sqrt{y}}e^{-\xi^2}d\xi\leq C\min\left\{\frac{1}{\sqrt{y}},\frac{1}{\e}\right\},	
\end{equation*}
and
\begin{align*}
\int_{2c_\e}^\infty\frac{1}{\xi}e^{-C\frac{\xi}{\e}y}
d\xi
=~&Ce^{-\frac{C}{\e^2}y}\int_{\frac{2Cc_\e y}{\e}}^\infty\frac{1}{\xi}e^{-\xi}d\xi\\
\leq~&Ce^{-\frac{C}{\e^2}y}\left(\int_{\frac{2Cc_\e y}{\e}}^{\max\{1,\frac{2Cc_\e y}{\e}\}}\frac{1}{\xi}d\xi+\int_{\max\{1,\frac{2Cc_\e y}{\e}\}}^\infty e^{-\xi}d\xi\right)\\
\leq~& C\max\left\{1,\log\frac{\e^2}{y}\right\}e^{-\frac{C}{\e^2}y}.
\end{align*}

The estimations \eqref{5.esti113} and \eqref{5.i11-bou} give the asymptotic behavior for $I_{11}$ when $r$ is small, i.e.,
\begin{equation}
\label{5.i11-asy}
\begin{aligned}
|I_{11}|\leq\begin{cases}
\frac{C}{r},\quad&\mbox{if}\quad m=1,\\
C\left(\min\left\{\log\frac{1+\sqrt{y}}{\sqrt{y}},\log\frac{1}{\e}\right\}+e^{-\frac{C}{\e^2}y}\right)\frac{1}{r}, \quad&\mbox{if}\quad m=2,\\
C\left(\min\left\{\frac{1}{\sqrt{y}},\frac{1}{\e}\right\}+\frac{1}{\e}\left\{1,\log\frac{\e^2}{y}\right\}e^{-\frac{C}{\e^2}y}\right)\frac{1}{r},\quad &\mbox{if}\quad m=3.	
\end{cases}
\end{aligned}
\end{equation}
While for $r\to\infty$, we could use integration by parts once more after \eqref{5.i11} for $m=2$ and two more times for $m=3$, due to \eqref{5.mm-asy}. Then at least we can get that
\begin{equation}
\label{5.i11-asy-1}
|I_{11}|\leq \frac{C}{r^\frac32},\quad \mbox{as}\quad r\to\infty\quad \mbox{for}~m\geq2.
\end{equation}

Concerning the second term $I_{12}$ on the right handside of \eqref{5.dec-i1}, when $|x|\leq Cy$ is small, we have
\begin{equation*}
I_{12}=-\int_{\R}e^{ix\xi-\sqrt{a(\xi)}y}\partial_\xi\left(\frac{1}{ix-\frac{a'(\xi)}{2\sqrt{a(\xi)}}y}\chi_\e(\xi)M_m(\xi)\right)	d\xi.
\end{equation*}
As \eqref{5.est-i11} we get that
\begin{equation}
\label{5.est-i12}
\left|\partial_\xi\left(\frac{1}{ix-\frac{a'(\xi)}{2\sqrt{a(\xi)}}y}\chi_\e(\xi)M_m(\xi)\right)	\right|\leq \frac{C}{r}\e^{4-m}\left(1+\frac{1}{\e^{\frac12}\sqrt{|\xi-c_\e|}}\right)e^{-\frac{C}{\e^2}y}.	
\end{equation}
Therefore,
\begin{equation}
\label{5.i12-1}
|I_{12}|\leq C\frac{\e^{3-m}}{r}e^{-\frac{C}{\e^2}y}\leq C\frac{\e^{4-m}}{r^{\frac32}}.
\end{equation}
While if $|x|\geq Cy$, then we have
\begin{equation*}
\int_{\R}\chi_\e(\xi)e^{ix\xi-\sqrt{a(\xi)}y}M_m(\xi)d\xi=-\frac{1}{ix}\int_{\R}e^{ix\xi}\partial_\xi\left(\chi_\e(\xi)e^{-\sqrt{a(\xi)}y}M_m(\xi)\right)d\xi.
\end{equation*}
By direct computation, we have
\begin{align*}
\partial_\xi
\left(\chi_\e(\xi)e^{-\sqrt{a(\xi)}y}M_m(\xi)\right)
=~&\chi_\e'(\xi)e^{-\sqrt{a(\xi)}y}M_m(\xi)+\chi_\e(\xi)e^{-\sqrt{a(\xi)}y}M_m'(\xi)\\
&-\chi_\e(\xi)\frac{a'(\xi)}{2\sqrt{a(\xi)}}ye^{-\sqrt{a(\xi)}y}M_m(\xi)\\
=~&I_{121}+I_{122}+I_{123}.
\end{align*}
From the expression of $M_m(\xi)$, we could obtain the following estimation
\begin{equation*}
|I_{121}|+|I_{122}|+|I_{123}|\leq \frac{C\e^{3-m}}{r}\left(1+\frac{1}{\e^{\frac12}\sqrt{|\xi-c_\e|}}\right)e^{-\frac{C}{\e^2}y}.
\end{equation*}
Then
\begin{equation}
\label{5.i12-est}
\begin{aligned}
|I_{12}|&\leq C\left(\int_{-2c_\e}^{-\frac12c_\e}
+\int_{\frac12c_\e}^{2c_\e}\right)\frac{\e^{3-m}}{r}\left(1+\frac{1}{\e^{\frac12}\sqrt{|\xi-c_\e|}}\right)e^{-\frac{C}{\e^2}y}d\xi\\
&\leq \frac{C}{r}\e^{2-m}e^{-\frac{C}{\e^2}y}.
\end{aligned}
\end{equation}
For $r$ large, without loss of generality we only give the explanation for the situation $|x|\geq Cy,$ while the case $|x|\leq Cy$ can be handled similarly. The crucial point is to obtain the most singular part in $\partial_\xi(\chi_\e(\xi)e^{-\sqrt{a(\xi)}y}M_m(\xi))$. Using \eqref{5.dasy}, we could write
\begin{equation}
\label{5.i12-asy}
I_{121}+I_{122}+I_{123}=
\begin{cases}
l_{\e}\frac{1}{\sqrt{c_\e-\xi}}(1+c_1\sqrt{c_\e-\xi}+O(c_\e-\xi)),~\mbox{around}~c_\e,\\	
l_{\e}\frac{1}{\sqrt{c_\e+\xi}}(1+c_1\sqrt{c_\e+\xi}+O(c_\e+\xi)),~\mbox{around}~-c_\e,
\end{cases}
\end{equation}
where
\begin{equation*}
\begin{aligned}
l_\e=&+\frac{c_\e^{m+1}(1-3\e^2c_\e^2-2\e^2)y}
{2\sqrt{3\e^4c_\e(c_\e^2+d_\e^2)}(\sqrt{2}\e^2c_\e^2(1+c_\e^2)+\frac{\sqrt{2}}{2}(1+\e^2c_\e^2)c_\e\sqrt{1+c_\e^2})}e^{-\frac{\sqrt{1+\e^2c_\e^2}}{\sqrt{2}\e^2}y}\\
&-\frac{c_\e^{m+1}(1-2\e^2-3\e^2c_\e^2)}
{\sqrt{6c_\e(1+\e^2c_\e^2)(c_\e^2+d_\e^2)}
(\sqrt{2}\e^2c_\e^2(1+c_\e^2)+\frac{\sqrt{2}}{2}c_\e(1+\e^2c_\e^2)\sqrt{1+c_\e^2})}e^{-\frac{\sqrt{1+\e^2c_\e^2}}{\sqrt{2}\e^2}y}\\
&+\frac{c_\e^{m+2}\sqrt{(1+c_\e^2)(1+\e^2c_\e^2)}(1-2\e^2-3\e^2c_\e^2)}{\sqrt{3c_\e(c_\e^2+d_\e^2)}
(\sqrt{2}\e^2c_\e^2(1+c_\e^2)+\frac{\sqrt{2}}{2}c_\e(1+\e^2c_\e^2)\sqrt{1+c_\e^2})^2}e^{-\frac{\sqrt{1+\e^2c_\e^2}}{\sqrt{2}\e^2}y}.
\end{aligned}
\end{equation*}
Around the zero point $c_\e$ and $-c_\e$, except the most singular terms $\frac{l_\e}{\sqrt{c_\e-\xi}}$ and $\frac{l_\e}{\sqrt{c_\e+\xi}}$, the left terms are continuous at $c_\e$ and one can get better decay after using integration by parts once more. For the most singular term we have
\begin{equation}
\label{5.i12-sin}
\int_{\R}e^{ix\xi}\left(l_\e\frac{1}{\sqrt{c_\e-\xi}}
+l_\e\frac{1}{\sqrt{c_\e+\xi}}\right)d\xi
=\frac{l_\e}{\sqrt{x}}\left(e^{ixc_\e}\int_{\R}\frac{e^{it}}{\sqrt{-t}}dt+e^{-ixc_\e}\int_{\R}\frac{e^{it}}{\sqrt{t}}dt\right).
\end{equation}
Using the asymptotic expansion \eqref{5.cd-asy} we can check that
\begin{equation}
\label{5.asy-le}
l_\e=O\left(\e^{\frac12-m}ye^{-\frac{C}{\e^2}y}+\e^{\frac52-m}e^{-\frac{C}{\e^2}y}\right).
\end{equation}
Using \eqref{5.i12-asy}, \eqref{5.i12-sin} and \eqref{5.asy-le}, we derive that
\begin{equation}
\label{5.i12}
|I_{12}|\leq
\begin{cases}
\frac{C}{r},\quad &\mbox{if}\quad m=1,\\
\frac{C}{r^{\frac32}},\quad &\mbox{if}\quad m=2,\\
\frac{C}{\e^{\frac12}r^{\frac32}},\quad  &\mbox{if}\quad m=3.	
\end{cases}
\end{equation}
Using \eqref{5.i11-asy}, \eqref{5.i11-asy-1}, \eqref{5.i12-1}, \eqref{5.i12-est} and \eqref{5.i12} we get that for $r$ small that
\begin{equation}
\label{5.i1-small}
|I_1|\leq
\begin{cases}
\frac{C}{r},\quad&\mbox{if}\quad m=1,\\
C\left(\min\left\{\log\frac{1+\sqrt{y}}{\sqrt{y}},\log\frac{1}{\e}\right\}+e^{-\frac{C}{\e^2}y}\right)\frac{1}{r}, \quad&\mbox{if}\quad m=2,\\
C\left(\min\left\{\frac{1}{\sqrt{y}},\frac{1}{\e}\right\}+\frac{1}{\e}\max\left\{1,\log\frac{\e^2}{y}\right\}e^{-\frac{C}{\e^2}y}\right)\frac{1}{r},\quad &\mbox{if}\quad m=3.	
\end{cases}
\end{equation}
While for $r$ large that
\begin{equation}
\label{5.i1-big}
|I_1|\leq
\begin{cases}
\frac{C}{r},\quad&\mbox{if}\quad m=1,\\
\frac{C}{r^{\frac32}}, \quad&\mbox{if}\quad m=2,\\
\frac{C}{\e^{\frac12}r^{\frac32}},\quad &\mbox{if}\quad m=3.	
\end{cases}
\end{equation}

Next we consider $I_2$. We set
$$z:=\left(\sqrt{a(\xi)}-\sqrt{b(\xi)}\right)y=-\frac{D(\xi)}{\e^4(\sqrt{a(\xi)}+\sqrt{b(\xi)})}y.$$
As $I_1$ we write
\begin{equation*}
I_2=\int_{\R}\frac{y\xi^me^{ix\xi-\sqrt{a(\xi)}y}}{\e^4(\sqrt{a(\xi)b(\xi)}+b(\xi))}\frac{e^z-1}{z}\left((1-\chi_\e(\xi))+\chi_\e(\xi)\right)d\xi
=\pi(I_{21}+I_{22}).
\end{equation*}
Concerning $I_{21}$, using integration by parts we get
\begin{equation}
\label{5.i21}
I_{21}=-\int_{\R}e^{ix\xi-\sqrt{a(\xi)}y}\partial_{\xi}\left(\frac{1}{ix-\frac{a'(\xi)}{2\sqrt{a(\xi)}}y}
\frac{y\xi^m(1-\chi_\e(\xi))}{\e^4(b(\xi)+\sqrt{a(\xi)b(\xi)})}\eta(z)\right)d\xi,
\end{equation}
where $\eta(z)=\frac{e^z-1}{z}$. As \eqref{5.est-i11} and \eqref{5.est-i12} we could obtain the following estimation,
\begin{equation*}
\begin{aligned}
&\left|\partial_{\xi}
\left(\frac{1}{ix\xi-\frac{a'(\xi)}{2\sqrt{a(\xi)}}y}
\frac{y\xi^m(1-\chi_\e(\xi))}{\e^4(b(\xi)+\sqrt{a(\xi)b(\xi)})}\eta(z)\right)\right|\\
&\leq\begin{cases}
C\dfrac{\e(1+\xi)^{m}+(1+\xi)^{m-1}}{r}y,
\quad&\mbox{if}\quad |\xi|\leq 2c_\e,\\
C\dfrac{1}{\e^2\xi^{3-m}r}y,\quad&\mbox{if}\quad |\xi|\geq 2c_\e.
\end{cases}
\end{aligned}
\end{equation*}
Then we have
\begin{equation}
\label{5.i21-est-1}
\begin{aligned}
|I_{21}|\leq~&\int_{-2c_\e}^{2c_\e}\frac{C}{r}\left(\e(1+\xi)^{m-2}+(1+\xi)^{m-3}\right)e^{-C\xi^2y}d\xi\\
&+\int_{\R\setminus[-2c_\e,2c_\e]}\frac{C}{\e\xi^{4-m}r}e^{-C\frac{|\xi|}{\e}y}d\xi\\
\leq~&\begin{cases}
\frac{C}{r},\quad&\mbox{if}\quad m=1,\\
C\left(\min\left\{\log\frac{1+\sqrt{y}}{\sqrt{y}},\log\frac{1}{\e}\right\}+e^{-\frac{C}{\e^2}y}\right)\frac{1}{r},\quad&\mbox{if}\quad m=2,\\
C\left(\min\left\{\frac{1}{\sqrt{y}},\frac{1}{\e}\right\}+\frac{1}{\e}\max\left\{1,\log\frac{\e^2}{y}\right\}e^{-\frac{C}{\e^2}y}\right)\frac{1}{r},\quad&\mbox{if}\quad m=3.
\end{cases}
\end{aligned}
\end{equation}
The above estimate \eqref{5.i21-est-1} gives the control on the asymptotic behavior for $I_{21}$ when $r$ is small. While for $r\to\infty$, we could use integration by parts once more after \eqref{5.i21} for $m=2$, and two more times for $m=3$, then there is no boundary terms due to that the integrand is $C^1$ continuous at $0$ for $m=2$ and $C^2$ continuous at $0$ for $m=3.$ After differentiation we could obtain that the integrand function has better decay and one can at least show that
\begin{equation}
\label{5.i21-est-2}
|I_{21}|\leq\frac{C}{r^\frac{3}{2}},\quad \mbox{as}\quad r\to\infty\quad \mbox{for}~m\geq2.	
\end{equation}

Concerning $I_{22},$ if $|y|\geq C|x|$ we have
\begin{equation*}
I_{22}=-\int_{\R}e^{ix\xi-\sqrt{a(\xi)}y}\partial_\xi
\left(\frac{1}{ix-\frac{a'(\xi)}{2\sqrt{a(\xi)}}y}
\frac{y\xi^m\eta(z)\chi_\e(\xi)}{\e^4(\sqrt{a(\xi)b(\xi)}+b(\xi))}
\right)d\xi.	
\end{equation*}
It is not difficult to check that
\begin{equation*}
\begin{aligned}
&\left|\partial_\xi
\left(\frac{1}{ix-\frac{a'(\xi)}{2\sqrt{a(\xi)}}y}
\frac{y\xi^m\eta(z)\chi_\e(\xi)}{\e^4(\sqrt{a(\xi)b(\xi)}+b(\xi))}
\right)\right|
&\leq \frac{C}{r}\left(\frac{\e^{-\frac12-m}y^2}{\sqrt{|c_\e-\xi|}}+\frac{\e^{\frac32-m}y}{\sqrt{|c_\e-\xi|}}\right)e^{-\frac{C}{\e^2}y}.
\end{aligned}	
\end{equation*}
As a consequence, we derive that
\begin{equation}
\label{5.i22-1}
\begin{aligned}
|I_{22}|\leq~&\frac{C}{r}\left(\int_{-2c_\e}^{-\frac12c_\e}+\int_{\frac12c_\e}^{2c_\e}\right)\frac{y\e^{\frac32-m}+y^2\e^{-\frac12-m}}{\sqrt{|c_\e-\xi|}}e^{-\frac{C}{\e^2}y}d\xi\\
=~&\frac{C}{r}(\e^{1-m}y+\e^{-1-m}y^2)e^{-\frac{C}{\e^2}y}\leq \frac{C}{r}\e^{3-m}e^{-\frac{C}{\e^2}y}.	
\end{aligned}
\end{equation}
While for $|x|\geq C|y|$ we use integration by parts to get that
\begin{equation}
\label{i.22-1}
I_{22}=-\int_{\R}\frac{1}{ix}e^{ix\xi}\partial_\xi
\left(\frac{y\xi^me^{-\sqrt{a(\xi)}y}}{\e^4(\sqrt{a(\xi)b(\xi)}+b(\xi))}\eta(z)\chi_\e(\xi)\right)d\xi.
\end{equation}
The support of the integrand is $[-2c_\e,-\frac12c_\e]\cup[\frac12c_\e,2c_\e]$. In the support of this function one can check
\begin{equation*}
\begin{aligned}
&\left|\partial_\xi\left(\frac{y\xi^me^{-\sqrt{a(\xi)}y}}{\e^4(\sqrt{a(\xi)b(\xi)}+b(\xi))}\eta(z)\chi_\e(\xi)\right)\right|\\
&\leq C\left(\e^{1-m}y+\e^{-1-m}y^2
+\frac{y}{\e^{-\frac12+m}\sqrt{|\xi-c_\e|}}+\frac{y^2}{\e^{\frac32+m}\sqrt{|\xi-c_\e|}}
\right)e^{-\frac{C}{\e^2}y}.
\end{aligned}
\end{equation*}
Then
\begin{equation}
\label{5.i22-3}
\begin{aligned}
|I_{22}|\leq~& \frac{C}{r}\left(\int_{-2c_\e}^{-\frac12c_\e}+\int_{\frac12c_\e}^{2c_\e}\right)\left(\frac{y}{\e^{-\frac12+m}\sqrt{|\xi-c_\e|}}+\frac{y^2}{\e^{\frac32+m}\sqrt{|\xi-c_\e|}}\right)e^{-\frac{C}{\e^2}y}
d\xi\\
\leq~&\frac{C}{r}\min\{
\e^{-2-m}y^2,\e^{-m}y\}e^{-\frac{C}{\e^2}y}
\leq\begin{cases}
\frac{C}{r},\quad &\mbox{if}\quad m=1,\\
\frac{C}{r},\quad   &\mbox{if}\quad m=2,\\ 	
\frac{C}{\e r}e^{-\frac{C}{\e^2}y}  ,\quad &\mbox{if}\quad m=3.
\end{cases}
\end{aligned}
\end{equation}
While for $r$ large, in order to get a better decay on the integration, we notice that the most singular term behave as $\frac{1}{\sqrt{c_\e-\xi}}$ and $\frac{1}{\sqrt{c_\e+\xi}}$. Following almost the same argument as we did from \eqref{5.i12-est} to \eqref{5.i12} we conclude that
\begin{equation}
\label{5.i22-4}
|I_{22}|\leq
\begin{cases}
\frac{C}{r},\quad &\mbox{if}\quad m=1,\\
\frac{C}{r^{\frac32}},\quad &\mbox{if}\quad m=2,\\
\frac{C}{\e^{\frac12}r^{\frac32}}e^{-\frac{C}{\e^2}y},\quad  &\mbox{if}\quad m=3.	
\end{cases}	
\end{equation}
Combined with \eqref{5.i21-est-1}, \eqref{5.i21-est-2}, \eqref{5.i22-3} we get for $r$ small that
\begin{equation}
\label{5.i2-small}
|I_2|\leq
\begin{cases}
\frac{C}{r},\quad&\mbox{if}\quad m=1,\\
C\left(\min\left\{\log\frac{1+\sqrt{y}}{\sqrt{y}},\log\frac{1}{\e}\right\}+e^{-\frac{C}{\e^2}y}\right)\frac{1}{r}, \quad&\mbox{if}\quad m=2,\\
C\left(\min\left\{\frac{1}{\sqrt{y}},\frac{1}{\e}\right\}+\frac{1}{\e}\max\left\{1,\log\frac{\e^2}{y}\right\}e^{-\frac{C}{\e^2}y}\right)\frac{1}{r},\quad &\mbox{if}\quad m=3.	
\end{cases}
\end{equation}
While for $r$ large it holds that
\begin{equation}
\label{5.i2-big}
|I_2|\leq
\begin{cases}
\frac{C}{r},\quad&\mbox{if}\quad m=1,\\
\frac{C}{r^{\frac32}}, \quad&\mbox{if}\quad m=2,\\
\frac{C}{\e^{\frac12}r^{\frac32}}e^{-\frac{C}{\e^2}y},\quad &\mbox{if}\quad m=3.	
\end{cases}
\end{equation}

It remains to consider the case where $y=0$. In this case, we derive that
\begin{equation*}
K_{m,0}(x,0)=\pi\int_{\R} e^{ix\xi}M_m(\xi)d\xi.
\end{equation*}
Since the integrand decays at infinity for $m=1$ and $m=2$ and the estimations for $K_{m,0}(x,0)$ ($m=1$ and $m=2$) are similar to the case $y\neq0$. While for $m=3$, one can easily see that as $|\xi|\to\infty$ we have
\begin{equation}
\label{5.m3asy}
M_3(\xi)=
\begin{cases}
\dfrac{\sqrt{1-\sqrt{3}i}}{\sqrt{2}+\frac{\sqrt{2}}{2}(1-\sqrt{3}i)}\frac{1}{\e}\left(1+O(\e^{-2}\xi^{-2})\right),\quad &\mbox{as}\quad \xi\to\infty,\\
\\
-\dfrac{\sqrt{1-\sqrt{3}i}}{\sqrt{2}+\frac{\sqrt{2}}{2}(1-\sqrt{3}i)}\frac{1}{\e}\left(1+O(\e^{-2}\xi^{-2})\right),\quad &\mbox{as}\quad \xi\to-\infty	.
\end{cases}
\end{equation}
Then we write
\begin{equation*}
\begin{aligned}
K_{3,0}=~&\pi\int_{\R} e^{ix\xi}\left(M_3(\xi)-sgn(\xi)\frac{\sqrt{1-\sqrt{3}i}}{\sqrt{2}+\frac{\sqrt{2}}{2}(1-\sqrt{3}i)}\frac{1}{\e}\right)d\xi\\
&+\pi\int_{\R}sgn(\xi)\frac{\sqrt{1-\sqrt{3}i}}{\sqrt{2}+\frac{\sqrt{2}}{2}(1-\sqrt{3}i)}\frac{1}{\e}e^{ix\xi}d\xi\\
=~&\pi\int_{\R} e^{ix\xi}\widetilde{M}_3(\xi)d\xi+\frac{i}{\e x}\frac{2\sqrt{1-\sqrt{3}i}}{\sqrt{2}+\frac{\sqrt{2}}{2}(1-\sqrt{3}i)},
\end{aligned}
\end{equation*}
where we used that
\begin{equation*}
\pi\int_{\R}e^{ix\xi}sgn(\xi)d\xi=-\frac{2}{ix}.	
\end{equation*}
and
\begin{equation*}
\widetilde{M}_3(\xi)=M_3(\xi)-sgn(\xi)\frac{\sqrt{1-\sqrt{3}i}}{\sqrt{2}+\frac{\sqrt{2}}{2}(1-\sqrt{3}i)}\frac{1}{\e}.
\end{equation*}
With \eqref{5.m3asy}, we see that $\widetilde{M}_3'(\xi)$ decays at infinity like $O(\e^{-2}\xi^{-3})$. Then following the arguments of deriving $I_1$, we could get
\begin{equation*}
r|K_{3,0}(x,0)|\leq\frac{C}{\e r}.
\end{equation*}
Then we finish the whole proof.
\end{proof}

In the next lemma, we shall investigate the behavior of $K_{m,n}(x,y)$ with $(m,n)=(0,1),(0,2),(0,3),(1,1),(1,2).$

\begin{lemma}
\label{lea.2}
Let $K_{m,n}(x,y)$ be defined in \eqref{a.2}. Then for $r$ small we have
\begin{equation*}
r|K_{m,n}(x,y)|\leq
\begin{cases}
C\min\left\{\log\frac{1+\sqrt{|y|}}{\sqrt{|y|}},\log\frac{1}{\e}\right\},\quad &\mathrm{for}\quad (m,n)=(0,1),\\
C\e^{-2}\left(\log\frac{1}{\e}+1\right)e^{-\frac{C}{\e^2}|y|}+\frac{C}{\sqrt{|y|}}\e^{-1},\quad &\mathrm{for}\quad (m,n)=(0,2),\\
C\e^{-4}\left(\log\frac{1}{\e}+\max\left\{1,\log\frac{\e^2}{|y|}\right\}\right)e^{-\frac{C}{\e^2}|y|}+\frac{C}{\sqrt{|y|}}\e^{-3},\quad &\mathrm{for}\quad (m,n)=(0,3),\\
C\left(\min\left\{\frac{1}{\sqrt{|y|}},\frac{1}{\e}\right\}+\frac{1}{\e}e^{-\frac{C}{\e^2}|y|}\right),\quad &\mathrm{for}\quad (m,n)=(1,1),\\
C\e^{-2}\left(\min\left\{\frac{1}{\sqrt{|y|}},\frac{1}{\e}\right\}+\frac{1}{\e}\max\left\{1,\log\frac{\e^2}{|y|}\right\}e^{-\frac{C}{\e^2}|y|}\right), &\mathrm{for}\quad (m,n)=(1,2).
\end{cases}
\end{equation*}
While for $r$ large we have
\begin{equation*}
r|K_{m,n}(x,y)|\leq
\begin{cases}
C\log\frac{1}{\e},\quad &\mathrm{for}\quad (m,n)=(0,1),\\
C\e^{-\frac32}r^{-\frac12},\quad &\mathrm{for}\quad (m,n)=(0,2),\\
C\e^{-\frac72}r^{-\frac12},\quad &\mathrm{for}\quad (m,n)=(0,3),\\
C\e^{-\frac12}r^{-\frac12},\quad &\mathrm{for}\quad (m,n)=(1,1),\\
C\e^{-\frac52}r^{-\frac12},\quad &\mathrm{for}\quad (m,n)=(1,2).
\end{cases}
\end{equation*}
\end{lemma}

\begin{proof}
The proof of this lemma goes almost the same as Lemma \ref{lea.1}. Here we only sketch the proof for the case $(m,n)=(1,2)$. First, we study the integration for $y>0$. Using Residue's Theorem we get
\begin{equation*}
\begin{aligned}
K_{1,2}(x,y)=~&\pi\int_{\R}\frac{\xi e^{ix\xi}}{D(\xi)}\left(\sqrt{b(\xi)}e^{-\sqrt{b(\xi)}y}-\sqrt{a(\xi)}e^{-\sqrt{a(\xi)}y}\right)d\xi\\
=~&\pi\int_{\R}\frac{\xi(\sqrt{b(\xi)}-\sqrt{a(\xi)})e^{ix\xi-\sqrt{b(\xi)}y}}{D(\xi)}d\xi\\
&+\pi\int_{\R}\frac{\xi\sqrt{a(\xi)}e^{ix\xi-\sqrt{a(\xi)}y}}{D(\xi)
}\left(e^{(\sqrt{a(\xi)}-\sqrt{b(\xi)})y}-1\right)d\xi\\
=~&\pi(I_3+I_4).
\end{aligned}
\end{equation*}
In the following, we shall estimate $I_3$ and $I_4$ respectively. Concerning $I_3$ first, we write
\begin{equation*}
\begin{aligned}
I_3=\int_{\R}\frac{\xi}{\e^4(\sqrt{b(\xi)}+\sqrt{b(\xi)})}e^{ix\xi-\sqrt{b(\xi)}y}\left((1-\chi_\e(x))+\chi_\e(x)\right)d\xi
=I_{31}+I_{32},
\end{aligned}
\end{equation*}
where $\chi_\e(x)$ is introduced in \eqref{5.def-cut}. Concerning $I_{31}$, using integration by parts we get
\begin{equation*}
I_{31}=-\int_{\R}e^{ix\xi-\sqrt{b(\xi)}y}\partial_\xi\left(\frac{1}{ix-\frac{b'(\xi)}{2\sqrt{b(\xi)}}y}\frac{(1-\chi_\e(\xi))\xi}{\e^4(\sqrt{a(\xi)}+\sqrt{b(\xi)})}\right)d\xi
\end{equation*}
By direct computation, we have
\begin{equation*}
\begin{aligned}
&\partial_\xi\left(\frac{1}{ix-\frac{b'(\xi)}{2\sqrt{b(\xi)}}y}\frac{(1-\chi_\e(\xi))\xi}{\e^4(\sqrt{a(\xi)}+\sqrt{b(\xi)})}\right)\\
&=\frac{\frac{2b''(\xi)b(\xi)-(b'(\xi))^2}{4(b(\xi))^{\frac32}}y}{\left(ix-\frac{b'(\xi)}{2\sqrt{b(\xi)}}y\right)^2}\frac{(1-\chi_\e(\xi))\xi}{\e^4(\sqrt{a(\xi)}+\sqrt{b(\xi)})}
+\frac{1}{ix-\frac{b'(\xi)}{2\sqrt{b(\xi)}}y}\frac{(1-\chi_\e(\xi))-\chi_\e'(\xi)\xi}{\e^4(\sqrt{a(\xi)}+\sqrt{b(\xi)})}\\
&\quad+\frac{1}{ix-\frac{b'(\xi)}{2\sqrt{b(\xi)}}y}\frac{(1-\chi_\e(\xi))\xi(\sqrt{a(\xi)}+\sqrt{b(\xi)})'}{\e^4(\sqrt{a(\xi)}+\sqrt{b(\xi)})^2}\\
&=I_{311}+I_{312}+I_{313}.
\end{aligned}
\end{equation*}
One can check that have
\begin{equation*}
|I_{311}|+|I_{312}|+|I_{313}|\leq
\begin{cases}
\frac{C}{\e^2r}  ~&\mathrm{for}~|\xi|\leq 2c_\e,\\
\\
\frac{C}{\e^3|\xi| r}  ~&\mathrm{for}~|\xi|\geq2c_\e.
\end{cases}
\end{equation*}
Then
\begin{equation}
\label{5.31-est1}
\begin{aligned}
|I_{31}|\leq~& \frac{C}{\e^2r}\int_{-2c_\e}^{2c_\e}e^{-\sqrt{b(\xi)}y}d\xi
+\frac{C}{\e^2r}\int_{\R\setminus[-2c_\e,2c_\e]}\frac{1}{\e|\xi| }e^{-\sqrt{b(\xi)}y}d\xi\\
\leq~&\frac{C}{\e^2r}\int_{-2c_\e}^{2c_\e}e^{-\frac{C}{\e^2}y}d\xi
+\frac{C}{\e^2r}\int_{\R\setminus[-2c_\e,2c_\e]}\frac{1}{\e|\xi| }e^{-C\frac{|\xi|}{\e}y}d\xi\\
\leq~&\frac{C}{\e^3r} \max\left\{1,\log\frac{\e^2}{y}\right\}e^{-\frac{C}{\e^2}y}.
\end{aligned}
\end{equation}
For $I_{31}$, the estimate \eqref{5.31-est1} gives the asymptotic behavior for $r$ small. While for $r$ large, we could apply the integration by parts to derive a better decay for $I_{31}$ as $r$ tends to $\infty$, and there is no boundary term arising due to the integrand function vanishes at $0$. At least one can show that $|I_{31}|\leq C\e^{-\frac52}r^{-\frac32}.$

For $I_{32}$, when $|x|\leq Cy$ we have
\begin{equation*}
I_{32}=-\int_{\R}e^{ix\xi-\sqrt{b(\xi)}}
\partial_\xi\left(\frac{1}{ix-\frac{b'(\xi)}{2\sqrt{b(\xi)}}y}\frac{\chi_\e(\xi)\xi}{\e^4(\sqrt{a(\xi)}+\sqrt{b(\xi)})}\right)d\xi.
\end{equation*}
For the term inside the bracket we could bound it by
\begin{equation*}
\left|\partial_\xi\left(\frac{1}{ix-\frac{b'(\xi)}{2\sqrt{b(\xi)}}y}\frac{\chi_\e(\xi)\xi}{\e^4(\sqrt{a(\xi)}+\sqrt{b(\xi)})}\right)\right|
\leq \frac{C}{r}\e^{-2}\left(1+\frac{1}{\e^{\frac12}\sqrt{|c_\e-\xi|}}\right)e^{-\frac{C}{\e^2}y}.	
\end{equation*}
Based on the above estimation we gain that
\begin{equation}
\label{5.31-est2}
|I_{32}|	\leq \frac{C}{\e^3r}e^{-\frac{C}{\e^2}y}.\end{equation}
While if $|x|\geq Cy$, then by integration by parts
\begin{equation}
\label{5.31-est3}
I_{32}=-\frac{1}{ix}\int_{\R}e^{ix\xi}\partial_\xi\left	
(e^{-\sqrt{b(\xi)}y}\frac{\chi_\e(\xi)\xi}{\e^4(\sqrt{a(\xi)}+\sqrt{b(\xi)})}\right)d\xi.
\end{equation}
By direct computation, we have
\begin{equation*}
\begin{aligned}
&\left|\partial_\xi\left	
(e^{-\sqrt{b(\xi)}y}\frac{\chi_\e(\xi)\xi}{\e^4(\sqrt{a(\xi)}+\sqrt{b(\xi)})}\right)\right|\\
&\leq \frac{C}{r}\left(\frac{1}{\e^2}+\frac{1}{\e^{\frac52}\sqrt{|c_\e-\xi|}}
+\frac{1}{\e^4}y+\frac{1}{\e^{\frac92}\sqrt{|c_\e-\xi|}}y\right)e^{-\frac{C}{\e^2}y}.
\end{aligned}	
\end{equation*}
As a consequence
\begin{equation}
\label{5.31-est5}
\begin{aligned}
|I_{32}|	\leq~ &\frac{C}{\e^2r}e^{-\frac{C}{\e^2}y}\left(\int_{-2c_\e}^{-\frac12c_\e}
+\int_{\frac12c_\e}^{2c_\e}\right)\left(
1+\frac{1}{\e^{\frac12}\sqrt{|c_\e-\xi|}}+\frac{1}{\e^2}y+\frac{1}{\e^{\frac52}\sqrt{|c_\e-\xi|}}y\right)d\xi\\
\leq ~& \frac{C}{\e^3r}e^{-\frac{C}{\e^2}y}.
\end{aligned}
\end{equation}
While for $r$ large, we could follow the computations performed from \eqref{5.i12-asy} to \eqref{5.i12} to derive that
\begin{equation}
\label{5.31-est6}
|I_{32}|\leq \frac{C}{\e^{\frac52}r^{\frac32}}e^{-\frac{C}{\e^2}y}.	
\end{equation}
Using \eqref{5.31-est1}-\eqref{5.31-est6} we get that
\begin{equation}
\label{5.i3}
|I_3|\leq\begin{cases}
\frac{C}{\e^3r}\max\left\{1,\log\frac{\e^2}{y}\right\}e^{-\frac{C}{\e^2}y},\quad &\mathrm{as}\quad r\to0,\\
\frac{C}{\e^{\frac52}r^{\frac32}}e^{-\frac{C}{\e^2}y},\quad &\mathrm{as}\quad r\to\infty.	
\end{cases}
\end{equation}

To study $I_4$, we write it as
\begin{equation*}
\label{5.i4-1}
\begin{aligned}
I_{4}=-\int_{\R}\frac{y\xi\sqrt{a(\xi)}e^{ix\xi-\sqrt{a(\xi)}y}}{\e^4(\sqrt{a(\xi)}+\sqrt{b(\xi)})}\eta(z)((1-\chi_\e(\xi))+\chi_\e(\xi))d\xi=-\pi(I_{41}+I_{42}).
\end{aligned}	
\end{equation*}
As we did for $I_2$ of last lemma and $I_3$ from above, we could see that
\begin{equation}
\label{5.i41}
|I_{41}|\leq\begin{cases}
C\left(\min\left\{\frac{1}{\sqrt{y}},\frac{1}{\e}\right\}+\frac{1}{\e}\max\left\{1,\log\frac{\e^2}{y}\right\}e^{-\frac{C}{\e^2}y}\right)\frac{1}{\e^2r},\quad&\mbox{as}\quad r\to0,\\
\frac{C}{\e^{\frac52}r^{\frac32}},&\mbox{as}\quad r\to\infty,	
\end{cases}	
\end{equation}
and
\begin{equation}
\label{5.i42}
|I_{42}|\leq\begin{cases}
C\left(\min\left\{\frac{1}{\sqrt{y}},\frac{1}{\e}\right\}+\frac{1}{\e}\max\left\{1,\log\frac{\e^2}{y}\right\}e^{-\frac{C}{\e^2}y}\right)\frac{1}{\e^2r},\quad&\mbox{as}\quad r\to0,\\
\frac{C}{\e^{\frac52}r^{\frac32}}e^{-\frac{C}{\e^2}y},&\mbox{as}\quad r\to\infty.	
\end{cases}
\end{equation}
Combining \eqref{5.i41} and \eqref{5.i42} we get
\begin{equation*}
|I_4|\leq\begin{cases}
C\left(\min\left\{\frac{1}{\sqrt{y}},\frac{1}{\e}\right\}+\frac{1}{\e}\max\left\{1,\log\frac{\e^2}{y}\right\}e^{-\frac{C}{\e^2}y}\right)\frac{1}{\e^2r},\quad &\mathrm{as}\quad r\to0,\\
\frac{C}{\e^{\frac52}r^{\frac32}},\quad &\mathrm{as}\quad r\to\infty.	
\end{cases}
\end{equation*}

It remains to consider the case where $y=0.$ In this case, the process is as lemma \ref{lea.1}, we have to analyze the following integral
\begin{equation*}
\int_{\R}\frac{\xi}{\e^4(\sqrt{a(\xi)}+\sqrt{b(\xi)})}e^{ix\xi}d\xi.
\end{equation*}
It is not difficult to check that as $|\xi|\to\infty$ we have
\begin{equation*}
\frac{\xi}{\e^4(\sqrt{a(\xi)}+\sqrt{b(\xi)})}	
=\begin{cases}
\frac{(1+O(\e^{-2}\xi^{-2}))}{\e^3\left(\sqrt{\frac{1-\sqrt{3}i}{2}}+\sqrt{\frac{1+\sqrt{3}i}{2}}\right)},~&\mbox{as}~\xi\to\infty,\\
\\
-\frac{(1+O(\e^{-2}\xi^{-2}))}{\e^3\left(\sqrt{\frac{1-\sqrt{3}i}{2}}+\sqrt{\frac{1+\sqrt{3}i}{2}}\right)},	~&\mbox{as}~\xi\to-\infty.
\end{cases}
\end{equation*}
Then we write
\begin{equation*}
\begin{aligned}
K_{1,2}(x,0)=~&\pi\int_{\R}e^{ix\xi}\left(\frac{\xi}{\e^4(\sqrt{a(\xi)}+\sqrt{b(\xi)})}-\frac{sgn(\xi)}{\e^3\left(\sqrt{\frac{1-\sqrt{3}i}{2}}+\sqrt{\frac{1+\sqrt{3}i}{2}}\right)}\right)d\xi\\
=~&\pi\int_{\R}e^{ix\xi}\frac{\xi}{\e^4(\sqrt{a(\xi)}+\sqrt{b(\xi)})}d\xi-\frac1x\frac{2 i}{\e^3\left(\sqrt{\frac{1-\sqrt{3}i}{2}}+\sqrt{\frac{1+\sqrt{3}i}{2}}\right)}.
\end{aligned}	
\end{equation*}
Then we can proceed the arguments as $I_3$ to treat the first one, and the integration by parts process works due to that the integrand function decays enough after taking differentiation. Then we get that
\begin{equation*}
\e^3|rK_{1,2}(x,0)|\leq \begin{cases}
 C,\quad &\mbox{as}\quad r\to0,\\
 \frac{C\e^\frac12}{r^\frac12},\quad	&\mbox{as}\quad r\to\infty.
 \end{cases}	
\end{equation*}	
\end{proof}

With Lemmas \ref{lea.1} and \ref{lea.2} we state the main result of this section
\begin{theorem}
\label{tha.1}
Let $G(x,y)$ be the Green kernel of the linear operator \eqref{3.green-kernel}. Then for $r$ small, we have
\begin{equation*}
\begin{aligned}
r|\partial_x^m\partial_y^nG(x,y)|
\leq\begin{cases}
C,\quad &\mathrm{for}\quad (m,n)=(1,0)\\
C\left(\min\left\{\log\frac{1+\sqrt{|y|}}{\sqrt{|y|}},\log\frac{1}{\e}\right\}+e^{-\frac{C}{\e^2}|y|}\right),\quad &\mathrm{for}\quad (m,n)=(2,0),\\
C\left(\min\left\{\frac{1}{\sqrt{|y|}},\frac{1}{\e}\right\}+\frac{1}{\e}\max\left\{1,\log\frac{\e^2}{|y|}\right\}
e^{-\frac{C}{\e^2}|y|}\right),\quad &\mathrm{for}\quad (m,n)=(3,0),\\
C\min\left\{\log\frac{1+\sqrt{|y|}}{\sqrt{|y|}},\log\frac{1}{\e}\right\},\quad &\mathrm{for}\quad (m,n)=(0,1),\\
C\e^{-2}\left(\log\frac{1}{\e}+1\right)e^{-\frac{C}{\e^2}|y|}+\frac{C}{\sqrt{|y|}}\e^{-1},\quad &\mathrm{for}\quad (m,n)=(0,2),\\
C\e^{-4}\left(\log\frac{1}{\e}+\max\left\{1,\log\frac{\e^2}{|y|}\right\}\right)e^{-\frac{C}{\e^2}|y|}+\frac{C}{\sqrt{|y|}}\e^{-3},\quad &\mathrm{for}\quad (m,n)=(0,3),\\
C\left(\min\left\{\frac{1}{\sqrt{|y|}},\frac{1}{\e}\right\}+\frac{1}{\e}e^{-\frac{C}{\e^2}|y|}\right),\quad &\mathrm{for}\quad (m,n)=(1,1),\\
C\e^{-2}\left(\min\left\{\frac{1}{\sqrt{|y|}},\frac{1}{\e}\right\}+\frac{1}{\e}\max\left\{1,\log\frac{\e^2}{|y|}\right\}e^{-\frac{C}{\e^2}|y|}\right), &\mathrm{for}\quad (m,n)=(1,2).
\end{cases}	
\end{aligned}	
\end{equation*}
While for $r$ large, we have
\begin{equation*}
\label{a.green2}
r|\partial_x^m\partial_y^nG(x,y)|
\leq
\begin{cases}
C,\quad &\mathrm{for}\quad (m,n)=(1,0),\\
\frac{C}{r^{\frac12}},\quad &\mathrm{for}\quad (m,n)=(2,0),\\
\frac{C}{\e^{\frac12}r^{\frac12}}\quad &\mathrm{for}\quad (m,n)=(3,0),\\
C\log\frac{1}{\e},\quad &\mathrm{for}\quad (m,n)=(0,1),\\
C\e^{-\frac32}r^{-\frac12},\quad &\mathrm{for}\quad (m,n)=(0,2),\\
C\e^{-\frac72}r^{-\frac12},\quad &\mathrm{for}\quad (m,n)=(0,3),\\
C\e^{-\frac12}r^{-\frac12},\quad &\mathrm{for}\quad (m,n)=(1,1),\\
C\e^{-\frac52}r^{-\frac12},\quad &\mathrm{for}\quad (m,n)=(1,2).	
\end{cases}
\end{equation*}
In addition, we have the following estimation on the integration $\int_{B_r(0)}|\partial_x^m\partial_y^nG(x,y)|dxdy$, denoted by $\mathcal{I}_{m,n}$
\begin{equation*}
\mathcal{I}_{m,n}\leq\begin{cases}
Cr,\quad &\mathrm{for}~(m,n)=(1,0),\\
Cr^\frac12,\quad &\mathrm{for}~(m,n)=(2,0),\\
C\left(r^{\frac12}+\e^{\frac12}\left(\log(1+r)\right)^2r^{\frac14}\right),\quad &\mathrm{for}~ (m,n)=(3,0),\\
Cr\left(\log r\right)^2,\quad &\mathrm{for}~(m,n)=(0,1),\\
C\left(\e^{-1}r^{\frac12}+\e^{-\frac12}r^{\frac14}\log r\right),\quad &\mathrm{for} ~(m,n)=(0,2),\\
C\left(\e^{-3}r^{\frac12}+\e^{-\frac52}\log\frac{1}{\e}\left(\log(1+r)\right)^2r^{\frac14}\right),~ &\mathrm{for}~(m,n)=(0,3),\\
Cr^{\frac12},\quad &\mathrm{for}~(m,n)=(1,1),\\
C\left(\e^{-2}r^{\frac12}+\e^{-\frac32}\left(\log(1+r)\right)^2r^{\frac14}\right),\quad &\mathrm{for}~ (m,n)=(1,2).
\end{cases}
\end{equation*}
\end{theorem}

\begin{proof}
The estimation on the asymptotic behavior and decay estimate follows easily by Lemmas \ref{lea.1} and \ref{lea.2}. It remains to prove the estimation on the integral $\mathcal{I}_{m,n}$.

Since the proof for each case is almost similar, we shall only give the details for $(0,3)$ and $(1,2)$. For the former one, we have
\begin{equation}
\begin{aligned}
\mathcal{I}_{0,3}
\leq~&C\e^{-4}\int_0^r\int_0^r\frac{1}{\sqrt{x^2+y^2}}\left(\log\frac{1}{\e}+\log\frac{1}{|y|}\right)e^{-\frac{C}{\e^2}|y|}dxdy\\
&+C\e^{-3}\int_{B_r(0)}\frac{1}{\sqrt{x^2+y^2}\sqrt{|y|}}dxdy\\
\leq~&C\e^{-4}\int_0^r\left(\int_0^{\frac{r}{|y|}}\frac{1}{\sqrt{x^2+1}}dx\right)\left(\log\frac{1}{\e}+\log\frac{1}{|y|}\right)e^{-\frac{C}{\e^2}|y|}dy\\
&+C\e^{-3}\int_0^r\int_0^{2\pi}\frac{1}{\rho^{\frac32}}\frac{1}{\sqrt{|\sin\theta|}}d\theta d\rho\\
\leq~&C\e^{-4}\int_0^r\left(\log\frac{1}{\e}+\log\frac{1}{|y|}\right)\log\left(1+\frac{r}{|y|}\right)e^{-\frac{C}{\e^2}|y|}dy+C\e^{-3}r^{\frac12}\\
\leq~&\left(\e^{-3}r^{\frac12}+\e^{-\frac52}\log\frac{1}{\e}\left(\log(1+r)\right)^2r^{\frac14}\right).
\end{aligned}
\end{equation}
While for the latter one, we have
\begin{equation}
\begin{aligned}
\mathcal{I}_{1,2}
\leq&~\frac{C}{\e^2}\int_{B_r(0)}\frac{1}{\sqrt{x^2+y^2}}\left(
\min\left\{\frac{1}{\sqrt{|y|}},\frac{1}{\e}\right\}+\frac{1}{\e}\max\left\{1,\log\frac{\e^2}{|y|}\right\}e^{-\frac{C}{\e^2}|y|}\right)dxdy\\
\leq&~C\e^{-3}\int_0^r\left(\int_0^{\frac{r}{|y|}}\frac{1}{\sqrt{x^2+1}}dx\right)\left(1+\log\frac{1}{|y|}\right)e^{-\frac{C}{\e^2}|y|}dxdy\\
&+C\e^{-2}\int_0^r\int_0^{2\pi}\frac{1}{\rho^{\frac32}}\frac{1}{\sqrt{|\sin\theta|}}d\theta d\rho\\
\leq&~C\e^{-3+\frac32}\int_0^r\left(1+\log\frac{1}{|y|}\right)\log\left(1+\frac{r}{|y|}\right)\frac{1}{|y|^\frac34}dy+C\e^{-2}r^{\frac12}\\
\leq&~C\left(\e^{-2}r^{\frac12}+\e^{-\frac32}\left(\log(1+r)\right)^2r^{\frac14}\right).
\end{aligned}
\end{equation}
Hence we finish the proof.
\end{proof}

\section{The linearized KP-I operator}

In this section, we study the mapping properties of the modified linearized KP-I operator
\begin{equation}
\label{3.1}
\begin{aligned}
&\partial_x^4\phi-(2\sqrt{2}-\e^2)\partial_x^2\phi-6(\sqrt{2}-\e^2)\partial_x(\partial_xq\partial_x\phi)-2\partial_y^2\phi+2\e^2\partial_x^2\partial_y^2\phi+\e^4\partial_y^4\phi\\
&=\partial_xh_1+\partial_yh_2.
\end{aligned}
\end{equation}
The functions $\phi$, $h_1$ and $h_2$ are in suitable functional spaces defined as follows respectively:
\begin{equation}
\mathcal{H}_1:=\left\{f\mid f\in \mathcal{H}_{ox},~ \|f\|_a<+\infty\right\},
\end{equation}
\begin{equation}
\mathcal{H}_2:=\left\{f\mid f\in\mathcal{H}_e,~\|f\|^2_{b}:=\int_{\R^2}(|f|^2+|\partial_xf|^2)dxdy<+\infty
\right\},
\end{equation}
and
\begin{equation}
\mathcal{H}_3:=\left\{f\mid f\in\mathcal{H}_{oxy},~\|f\|^2_{c}:=\int_{\R^2}(|f|^2+|\partial_yf|^2)dxdy<+\infty
\right\},	
\end{equation}
where
\begin{equation*}
\begin{aligned}
\|f\|_a^2=\int_{\R^2}\left(|\partial_x^4\phi|^2+\e^4|\partial_x^2\partial_y^2\phi|^2+\e^8|\partial_y^4\phi|^2+|\nabla^2\phi|^2+|\nabla\phi|^2\right)dxdy,	
\end{aligned}
\end{equation*}
$\mathcal{H}_o$ and $\mathcal{H}_e$ are given as
\begin{equation*}
\mathcal{H}_{ox}=\left\{f\mid f(x,y)=-f(-x,y)=f(x,-y)\right\},
\end{equation*}
\begin{equation*}
\mathcal{H}_e=\left\{f\mid f(x,y)=f(-x,y)=f(x,-y)\right\},	
\end{equation*}
and
\begin{equation*}
\mathcal{H}_{oxy}=\left\{f\mid f(x,y)=-f(-x,y)=-f(x,-y)\right\}.	
\end{equation*}

The first result of this section is about the a-priori estimate
\begin{proposition}
\label{pr3.1}
Let $\phi\in \mathcal{H}_1$ be a solution of \eqref{3.1} with $h_1\in\mathcal{H}_2$ and $h_2\in\mathcal{H}_3$. Then there exist positive constants $\e_0$ and $C$ such that for all $\e\in(0,\e_0)$, it holds that
\begin{equation*}
\|\phi\|_a\leq C\left(\|h_1\|_{b}+\|h_2\|_c\right).
\end{equation*}
\end{proposition}

\begin{proof}
Consider the differential operator
$$L\phi=\partial_x^2\phi-(2\sqrt{2}-\e^2)\phi-6\sqrt{2}\left(\frac{2\sqrt{2}-\e^2}{2\sqrt{2}}\right)^{\frac52}(\partial_xq\phi)-2\partial_x^{-2}\partial_y^2\phi.$$
It is known from \cite[Theorem 2]{Liu-Wei-2} that it admits only negative eigenvalue $\la_{1}$, the associated eigenfunction is denoted by $\phi_0$, i.e.,
\begin{equation*}
\partial_x^2\phi_0-(2\sqrt{2}-\e^2)\phi_0-6\sqrt{2}\left(\frac{2\sqrt{2}-\e^2}{2\sqrt{2}}\right)^{\frac52}(\partial_xq\phi_0)-2\partial_x^{-2}\partial_y^2\phi_0+\la_1\phi_0=0.
\end{equation*}
We notice that as $\e\to0$, the negative eigenvalue $\la_{1}$ has a limit value which corresponds to the unique negative eigenvalue of the linearized equation of \eqref{lump-0}. It is known that $\int_{-\infty}^\infty\phi_0dx=0$ for any $y$ and we can define $$\phi_1:=\partial_x^{-1}\phi_0=\int_{-\infty}^x\phi_0dx.$$
Then we see that $\phi_1$ is a function odd in $x$, even in $y$,  and
\begin{equation*}
\partial_x^4\phi_1-(2\sqrt{2}-\e^2)\partial_x^2\phi_1-6\sqrt{2}\left(\frac{2\sqrt{2}-\e^2}{2\sqrt{2}}\right)^{\frac52}\partial_x(\partial_xq\partial_x\phi_1)-2\partial_y^2\phi_1=-\la_1\partial_x^2\phi_1.
\end{equation*}

We decompose $\phi$ into $\phi=c\phi_1+\phi_2,$ with $$c=\frac{\int_{\R^2}\partial_x\phi\partial_x\phi_1dx}{\int_{\R^2}\partial_x\phi_1\partial_x\phi_1dx},$$
and $\phi_2$ belonging to the complement of $\phi_1$ in the space $\mathcal{H}_1$ under the following product
\begin{equation*}
\label{3.pro}
(f,g)
=\int_{\R^2}\left(\partial_x^2f\partial_x^2g+(2\sqrt{2}-\e^2)\partial_xf\partial_xg
+6\sqrt{2}\left(\frac{2\sqrt{2}-\e^2}{2\sqrt{2}}\right)^{\frac52}\partial_xq\partial_xf\partial_xg+2\partial_yf\partial_yg\right)dxdy.
\end{equation*}
For any function $\psi$ in the complement space of $\phi_1$ in $\mathcal{H}_1$, we see that
\begin{equation*}
(\psi, \psi)\geq \la_2\|\partial_x\psi\|^2_{L^2(\R^2)},
\end{equation*}
where $\la_2$ refers to the smallest positive eigenvalue of $L$ and it is not difficult to see that $\la_2$ has a uniform bound for any $\e$. In addition, for such $\psi$, we can choose $\la_*\in(0,1)$ such that
\begin{equation}
\label{3.la*}
\la_2(1-\la_*)-24\la_*\geq 2\la_*.
\end{equation}
As a consequence
\begin{equation}
\label{3.complement}
\begin{aligned}
(\psi,\psi)\geq~&(1-\la_*)(\psi,\psi)-6\sqrt{2}\left(\frac{2\sqrt{2}-\e^2}{2\sqrt{2}}\right)^{\frac52}\la_*\max_{(x,y)\in\R^2}|\partial_xq|\int_{\R^2}(\partial_x\psi)^2dxdy\\
&+2\la_*\int_{\R^2}(\partial_y\psi)^2dxdy\\
\geq~&(\la_2(1-\la_*)-24\la_*)\int_{\R^2}(\partial_x\psi)^2dxdy+2\la_*\int_{\R}(\partial_y\psi)^2dxdy\\
=~&2\la_*\int_{\R^2}|\nabla\psi|^2dxdy,
\end{aligned}
\end{equation}
where we used $\max\limits_{(x,y)\in\R^2}6\sqrt{2}\left(\frac{2\sqrt{2}-\e^2}{2\sqrt{2}}\right)^{\frac52}|\partial_xq(x,y)|\leq24.$ We write \eqref{3.1} as
\begin{equation}
\label{3.3}
\begin{aligned}
&\partial_x^4\phi_2-(2\sqrt{2}-\e^2)\partial_x^2\phi_2-6(\sqrt{2}-\e^2)\partial_x(\partial_xq\partial_x\phi_2)-2\partial_y^2\phi_2+2\e^2\partial_x^2\partial_y^2\phi_2+\e^4\partial_y^4\phi_2\\
&=\partial_xh_1+\partial_yh_2+c\la_1\partial_x^2\phi_1+c\left(-6\sqrt{2}\left(\frac{2\sqrt{2}-\e^2}{2\sqrt{2}}\right)^{\frac52}+6\sqrt{2}-6\e^2
\right)\partial_x(\partial_xq\partial_x\phi_1)\\
&\quad-2c\e^2\partial_x^2\partial_y^2\phi_1-c\e^4\partial_y^4\phi_1.
\end{aligned}
\end{equation}
For convenience, we set
$$\Lambda_\e=\left(6\sqrt{2}\left(\frac{2\sqrt{2}-\e^2}{2\sqrt{2}}\right)^{\frac52}-6\sqrt{2}+6\e^2
\right).$$
It is not difficult to see that $\Lambda_\e=O(\e^2)$ and negative when $\e$ is small enough. Multiplying \eqref{3.3} by $c\phi_1$ and using integration by parts we gain
\begin{equation}
\label{3.4}
\begin{aligned}
&-c\int_{\R^2}\left(\Lambda_\e\partial_xq\partial_x\phi_1\partial_x\phi_2-2\e^2\partial_x\partial_y\phi_1\partial_x\partial_y\phi_2-\e^4\partial_y^2\phi_1\partial_y^2\phi_2\right)dxdy\\
&=-c\int_{\R^2}h_1\partial_x\phi_1dxdy-c\int_{\R^2}h_2\partial_y\phi_1dxdy-c^2\la_1\|\partial_x\phi_1\|_{L^2(\R^2)}^2-c^2\e^4\|\partial_y^2\phi_1\|_{L^2(\R^2)}^2\\
&\quad-2c^2\e^2\|\partial_x\partial_y\phi_1\|^2_{L^2(\R^2)}
+c^2\Lambda_\e\int_{\R^2}\partial_xq|\partial_x\phi_1|^2dxdy.
\end{aligned}
\end{equation}
From \eqref{3.4} and Young's inequality we get
\begin{equation}
\label{3.control-1}
\begin{aligned}
&c^2(-\la_1-\delta_\e)\|\partial_x\phi_1\|_{L^2(\R^2)}^2\\
&\leq-\frac{\la_*}{32\la_1}\left(\e^2\|\partial_x\phi_2\|_{L^2(\R^2)}^2+2\e^2\|\partial_x\partial_y\phi_2\|^2_{L^2(\R^2)}+\e^4\|\partial_y^2\phi_2\|_{L^2(\R^2)}^2\right)\\
&\quad-\frac{2}{\la_1}\|h_1\|^2_{L^2(\R^2)}
-\frac{2}{\la_1}\frac{\|\partial_y\phi_1\|^2_{L^2(\R^2)}}{\|\partial_x\phi_1\|^2_{L^2(\R^2)}}\|h_2\|^2_{L^2(\R^2)},
\end{aligned}
\end{equation}
where $\la_*$ is given in \eqref{3.la*}, $\delta_\e$ depends on $\e$ and can be arbitrarily small as $\e\to0$. In \eqref{3.control-1} we have also used $\delta(\e)\|\partial_x\phi_1\|_{L^2(\R^2)}$ to control the terms $\e^2\|\partial_x\partial_y\phi_1\|_{L^2(\R^2)}^2$ and $\e^4\|\partial_y^2\phi_1\|_{L^2(\R^2)}^2$, due to that $\phi_1$ is a concrete function. By the same reason, we see that the coefficient before $\|h_2\|_{L^2(\R^2)}$ is a specified constant.
\medskip

Multiplying \eqref{3.3} by $\phi_2$ and using \eqref{3.complement} we get
\begin{equation}
\label{3.5}
\begin{aligned}
&2\la_*\|\nabla\phi_2\|^2_{L^2(\R^2)}+2\e^2\|\partial_x\partial_y\phi_2\|^2_{L^2(\R^2)}+\e^4\|\partial_y^2\phi_2\|^2_{L^2(\R^2)}\\
&\leq\int_{\R^2}\left((\partial_x^2\phi_2)^2
+(2\sqrt{2}-\e^2)(\partial_x\phi_2)^2+6\sqrt{2}\left(\frac{2\sqrt{2}-\e^2}{2\sqrt{2}}\right)^{\frac52}\partial_xq(\partial_x\phi_2)^2+2(\partial_y\phi_2)^2\right)dxdy\\
&\quad+\int_{\R^2}\left(2\e^2(\partial_x\partial_y\phi_2)^2+\e^4(\partial_y^2\phi_2)^2\right)dxdy\\
&\leq \Lambda_\e\int_{\R}\partial_xq|\partial_x\phi_2|^2dxdy+\frac{1}{\la_*}\|h_1\|^2_{L^2(\R^2)}+\frac{1}{\la_*}\|h_2\|^2_{L^2(\R^2)}\\
&\quad+\frac{\la_*}{2}\|\nabla\phi_2\|^2_{L^2(\R^2)}+c^2\left(\frac{\la_1^2}{\la_*}+\tilde{\delta}_\e\right)\|\partial_x\phi_1\|^2_{L^2(\R^2)}
+\frac{1}{2}\e^2\|\partial_x\phi_2\|_{L^2(\R^2)}^2\\
&\quad +\e^2\|\partial_x\partial_y\phi_2\|_{L^2(\R^2)}^2
+\frac12\e^4\|\partial_y^2\phi_2\|_{L^2(\R^2)}^2,
\end{aligned}	
\end{equation}
where $\la_*>0$ is given in \eqref{3.la*},  and $\tilde\delta_\e$ can be arbitrary small as $\e\to0$. Taking $\e$ such that
$$\left(2-\Lambda_\e\max\limits_{(x,y)\in\R^2}|\partial_xq|\right)\e^2<\la_*,$$
we get
\begin{equation}
\label{3.control-2}
\begin{aligned}
&\la_*\|\nabla\phi_2\|^2_{L^2(\R^2)}+\e^2\|\partial_x\phi_2\|^2_{L^2(\R^2)}+2\e^2\|\partial_x\partial_y\phi_2\|^2_{L^2(\R^2)}+\e^4\|\partial_y^2\phi_2\|^2_{L^2(\R^2)}\\
&\leq \frac{2}{\la_*}\|h_1\|^2_{L^2(\R^2)}+\frac{2}{\la_*}\|h_2\|^2_{L^2(\R^2)}+2c^2\left(\frac{\la_1^2}{\la_*}+\tilde\delta_\e\right)\|\partial_x\phi_1\|^2_{L^2(\R^2)}.	
\end{aligned}
\end{equation}
Using \eqref{3.control-1} and \eqref{3.control-2} we derive that
\begin{equation}
\label{3.est}
\begin{aligned}
&\frac{c^2\la_1^2}{\la_*}\|\partial_x\phi_1\|^2_{L^2(\R^2)}
+\la_*\|\nabla\phi_2\|^2_{L^2(\R^2)}\\
&\leq\left(\frac{8}{\la_*}\frac{\|\partial_y\phi_1\|^2_{L^2(\R^2)}}{\|\partial_x\phi_1\|^2_{L^2(\R^2)}}+\frac{2}{\la_*}\right)\|h_2\|^2_{L^2(\R^2)}
+\frac{10}{\la_*}\|h_1\|^2_{L^2(\R^2)}.
\end{aligned}
\end{equation}
This implies that
\begin{equation}
\label{3.est-1}
\|\nabla\phi\|^2_{L^2(\R)}\leq C\left(\|h_1\|^2_{L^2(\R^2)}+\|h_2\|^2_{L^2(\R^2)}\right).
\end{equation}
On the other hand, for the specific function $\phi_1$,
\begin{equation}
\label{3.est-2}
\int_{\R^2}(\partial_x^2\phi_1)^2dxdy\leq C\int_{\R^2}(\partial_x\phi_1)^2dxdy.
\end{equation}
Together with \eqref{3.5} we have
\begin{equation}
\label{3.est-3}
\int_{\R^2}(\partial_x^2\phi_2)^2dxdy\leq C\left(\|h_1\|^2_{L^2(\R^2)}+\|h_2\|^2_{L^2(\R^2)}\right).
\end{equation}
As a consequence of \eqref{3.est-1}, \eqref{3.est-2} and \eqref{3.est-3} we get that
\begin{equation}
\label{3.l2}
\|\partial_x(q\partial_x\phi)\|_{L^2(\R^2)}\leq
C\left(\|h_1\|_{L^2(\R^2)}+\|h_2\|_{L^2(\R^2)}\right).
\end{equation}

Next, we write \eqref{3.1} as
\begin{equation}
\label{3.eq-2}
\begin{aligned}
&\partial_x^4\phi-(2\sqrt{2}-\e^2)\partial_x^2\phi-2\partial_y^2\phi+2\e^2\partial_x^2\partial_y^2\phi+\e^4\partial_y^4\phi\\
&=\partial_xh_1+\partial_yh_2+6(\sqrt{2}-\e^2)\partial_x(\partial_xq\partial_x\phi).
\end{aligned}	
\end{equation}
Denote the right handside by $h$, we get from \eqref{3.l2} that
$$\|h\|_{L^2(\R^2)}\leq C\int_{\R^2}(|\partial_xh_1|^2+|\partial_yh_2|^2+|h_1|^2+|h_2|^2)dxdy=C(\|h_1\|_{b}+\|h_2\|_c).$$
Multiplying \eqref{3.eq-2} by $\phi$, after integration by parts and using the H\"older inequality we have
\begin{equation*}
\int_{\R^2}\left(|\partial_x^2\phi|^2+|\nabla\phi|^2+\e^2|\partial_x\partial_y\phi|^2+\e^4|\partial_y^2\phi|^2\right)dxdy\leq
C\|h\|_{L^2(\R^2)}^2.
\end{equation*}
To prove the original conclusion, it is enough to show that each term in the definition of $\|\cdot\|_a$ is bounded by $\|h\|_{L^2(\R^2)}.$ We only give explanation for the term $\|\partial_x^4\phi\|_{L^2(\R^2)}^2$, while the other terms can be handled similarly. Taking Fourier transformation on both sides of \eqref{3.eq-2} we have
\begin{equation*}
\hat\phi(\xi_1,\xi_2)=\frac{1}{\xi_1^4+(2\sqrt{2}-\e^2)\xi_1^2+2\xi_2^2+2\e^2\xi_1^2\xi_2^2+\e^4\xi_2^4}\hat h(\xi_1,\xi_2).
\end{equation*}
It is known that
\begin{align*}
\|\partial_x^4\phi\|_{L^2(\R^2)}=~&\|\xi_1^4\hat\phi\|_{L^2(\R^2)}
=\left\|\frac{\xi_1^4}{\xi_1^4+2\sqrt{2}\xi_1^2+2\xi_2^2+2\e^2\xi_1^2\xi_2^2+\e^4\xi_2^4}\hat h(\xi_1,\xi_2)\right\|_{L^2(\R^2)}\\
\leq~&C\|\hat h\|_{L^2(\R^2)}=C\|h\|_{L^2(\R^2)}\leq C\left(\|h_1\|_{b}+\|h_2\|_{c}\right).
\end{align*}
Then we prove the conclusion.
\end{proof}

After establishing the $L^2$ theory for equation \eqref{3.1}, we consider the \eqref{3.1} in a suitable weighted Sobolev space, which helps us to study the nonlinear problem \eqref{2.equation}. Now, we introduce the weighted Sobolev space for $\phi$, $h_1$ and $h_2$:
\begin{equation}
\label{3.def-phi}
\phi\in \mathcal{F}_1:=\left\{f\mid f\in\mathcal{H}_1,~\|f\|_{*}
<+\infty\right\},
\end{equation}
\begin{equation}
\label{3.def-h1}
h_1\in\mathcal{F}_2:=\left\{f\mid f\in\mathcal{H}_2,~\|f\|_{**}<+\infty\right\},
\end{equation}
and
\begin{equation}
\label{3.def-h2}
h_2\in\mathcal{F}_3:=\left\{f\mid f\in\mathcal{H}_3,~\|f\|_{***}<+\infty\right\},
\end{equation}
where
\begin{equation}
\label{3.def-*}
\begin{aligned}
\|f\|_{*}=~&\|f\|_a+\|(1+r)^{1-\delta}f\|_{L^\infty(\R^2)}+
\frac{1}{\log\frac{1}{\e}}\|(1+r)f\|_{L^\infty(\R^2)}\\
&+\|(1+r)^{\frac32-\delta}\partial_xf\|_{L^\infty(\R^2)}+\e^\frac12\|(1+r)^{\frac32}\partial_xf\|_{L^\infty(\R^2)}\\
&+\|(1+r)^{\frac32}\partial_x^2f\|_{L^\infty(\R^2)}+\e^\frac12\|(1+r)^{\frac32}\partial_x^3f\|_{L^\infty(\R^2)}\\
&+\e^\frac12\|(1+r)^{\frac32}\partial_x^4f\|_{L^\infty(\R^2)}+\e^\frac12\|(1+r)^{\frac32-\delta}\partial_yf\|_{L^\infty(\R^2)}\\
&+\e^\frac32\|(1+r)^{\frac32}\partial_yf\|_{L^\infty(\R^2)}+\e^\frac32\|(1+r)^{\frac32}\partial_y^2f\|_{L^\infty(\R^2)}\\
&+\e^\frac72\|(1+r)^{\frac32}\partial_y^3f\|_{L^\infty(\R^2)}+\e^\frac{11}2\|(1+r)^{\frac32}\partial_y^4f\|_{L^\infty(\R^2)}\\
&+\e^\frac12\|(1+r)^{\frac32}\partial_x\partial_yf\|_{L^\infty(\R^2)}+\e^{\frac12}\|(1+r)^{\frac32}\partial_x^2\partial_yf\|_{L^\infty(\R^2)}\\
&+\e^{\frac32}\|(1+r)^{\frac32}\partial_x\partial_y^2f\|_{L^\infty(\R^2)}+\e^\frac52\|(1+r)^{\frac32}\partial_x^2\partial_y^2f\|_{L^\infty(\R^2)}\\
&+\e^\frac32\|(1+r)^{\frac32-\delta}\partial_x^{-1}\partial_y^2f\|_{L^\infty(\R^2)}+\e^\frac72\|(1+r)^{\frac32-\delta}\|\partial_x^{-1}\partial_y^3f\|_{L^\infty(\R^2)}\\
&+\e^\frac{11}{2}\|(1+r)^{\frac32-\delta}\partial_x^{-1}\partial_y^4f\|_{L^\infty(\R^2)},
\end{aligned}
\end{equation}
and
\begin{equation}
\label{3.def-**}
\begin{aligned}
\|f\|_{**}=~&\|f\|_{b}+\|(1+r)^{5/2-\delta}h\|_{L^\infty(\R^2)}	
+\|(1+r)^{5/2-\delta}\partial_xh\|_{L^\infty(\R^2)}\\
&+\|(1+r)^{5/2-\delta}\partial_x^2h\|_{L^\infty(\R^2)},
\end{aligned}
\end{equation}
\begin{equation}
\label{3.def-***}
\begin{aligned}
\|f\|_{***}=~&\|f\|_{c}+\|(1+r)^{3-\delta}h\|_{L^\infty(\R^2)}	
+\|(1+r)^{3-\delta}\partial_yh\|_{L^\infty(\R^2)}\\
&+\|(1+r)^{3-\delta}\partial_{xy}h\|_{L^\infty(\R^2)}.
\end{aligned}
\end{equation}
Here $\delta$ is slightly greater than $0$. We remark that for any function $f\in\mathcal{F}_3$, we could define $\partial_x^{-1}f=-\int_x^{\infty}f(s,y)ds$ and it is easy to check that
\begin{equation}
\label{3.x-1}
\|(1+r)^{2-\delta}\partial_x^{-1}f\|_{L^\infty(\R^2)}	
+\|(1+r)^{2-\delta}\partial_x^{-1}\partial_yf\|_{L^\infty(\R^2)}\leq C\|f\|_{***}.
\end{equation}
We shall use \eqref{3.x-1} in the following proposition frequently.

The second result of this section is as follows

\begin{proposition}
\label{pr3.2}
For each $h_1\in\mathcal{F}_2$ and $h_2\in\mathcal{F}_3$. Then there exist positive constants $\e_0$ such that for all $\e\in(0,\e_0)$ the linear equation \eqref{3.1} admits a solution $\phi\in\mathcal{F}_1$, with
\begin{equation}
\label{3.2-1}
\|\phi\|_{*}\leq C\left(\|h_1\|_{**}+\|h_2\|_{***}\right).	
\end{equation}
\end{proposition}

\begin{proof}
From Proposition \ref{pr3.1} and Lax-Milgram Theorem we can get a solution $\phi\in\mathcal{H}_1$ to equation \eqref{3.1} provided $h_1\in\mathcal{H}_2$ and $h_2\in\mathcal{H}_3$, the smooth property follows by the Sobolev inequality and Schauder estimate for the biharmonic equation, we refer the readers to \cite{Besov,Besov2,H1,Shao,Simon}. Then we shall focus on the decay estimate in the following. To prove \eqref{3.2-1}, we write equation \eqref{3.1} as
\begin{equation}
\label{3.2-eq}
\begin{aligned}
&\partial_x^4\phi-(2\sqrt{2}-\e^2)\partial_x^2\phi-2\partial_y^2\phi
+2\e^2\partial_x^2\partial_y^2\phi+\e^4\partial_y^4\phi\\	
&=\partial_x\left(6(\sqrt{2}-\e^2)\partial_xq\partial_x\phi+h_1\right)+\partial_yh_2.
\end{aligned}
\end{equation}
In addition, it holds that
\begin{equation*}
\|\phi\|_a\leq C\left(\|h_1\|_{b}+\|h_2\|_{c}\right).
\end{equation*}
As a consequence, we have
\begin{equation}
\label{3.gn-i}
\|\nabla\phi\|_{L^2(\R^2)}+\|\nabla^2\phi\|_{L^2(\R^2)}+\|\partial_y\partial_x^2\phi\|_{L^2(\R^2)}+\|\partial_x^3\phi\|_{L^2(\R^2)}
\leq C\left(\|h_1\|_{b}+\|h_2\|_{c}\right).
\end{equation}
Using Gagliardo-Nirenberg interpolation inequality (see \cite{Adams} for the proof of this inequality)
\begin{equation}
\label{3.gn}
\|D^ju\|_{L^{p_1}(\R^n)}\leq C\|D^mu\|_{L^r(\R^n)}^\alpha\|u\|_{L^{p_2}(\R^n)}^{1-\alpha},
\end{equation}
where $
\frac{1}{p_1}=\frac{j}{n}+\left(\frac1r-\frac{m}{n}\right)\alpha+\frac{1-\alpha}{p_2}.$
Using \eqref{3.gn} and \eqref{3.gn-i} we conclude that
\begin{equation}
\label{3.abc}
\|\partial_x\phi\|_{L^p(\R^2)}+\|\partial_x^2\phi\|_{L^p(\R^2)}\leq C\|\phi\|_a\leq C\left(\|h_1\|_{b}+\|h_2\|_{c}\right),~ \forall p\in[2,+\infty).
\end{equation}

Denote the Green kernel of the linear operator defined on the left hand side of \eqref{3.2-eq} by $G(x,y).$ In the following we shall derive the a-priori estimate \eqref{3.2-1} by analyzing $\|\phi\|_*$ term by term. The proof is quite long, we shall divide our discussion into the following five steps

Step 1. We start with the estimation on $\partial_x^m\phi(x,y)$ for $m=0,1,\cdots,4$. Using Green representation we have
\begin{equation}
\label{3.2-2}
\begin{aligned}
\phi(x,y)=~&-6(\sqrt{2}-\e^2)\int_{\R^2}\partial_xG(x-s,y-t)(\partial_xq\partial_x\phi)(s,t)dsdt\\
&-\int_{\R^2}\partial_xG(x-s,y-t)h_1(s,t)dsdt\\
&-\int_{\R^2}\partial_yG(x-s,y-t)h_2(s,t)dsdt.
\end{aligned}
\end{equation}
We shall consider the asymptotic behavior of $\phi(x,y)$ when $r=\sqrt{x^2+y^2}$ is large. For the first term on the right hand side of \eqref{3.2-2} we have
\begin{equation}
\label{3.2-2-1}
\begin{aligned}
&\left|\int_{\R^2}\partial_xG(x-s,y-t)(\partial_xq\partial_x\phi)(s,t)dsdt\right|\\
&\leq\int_{B_{r/2}(x,y)}\left|\partial_xG(x-s,y-t)(\partial_xq\partial_x\phi)(s,t)\right|dsdt\\
&\quad+\int_{\R^2\setminus B_{r/2}(x,y)}	\left|\partial_xG(x-s,y-t)(\partial_xq\partial_x\phi)(s,t)\right|dsdt\\
&\leq\frac{C}{r^2}\left(\int_{B_{r/2}(x,y)}|\partial_xG(x-s,y-t)|^\frac{6}{5}dsdt\right)^{\frac56}\left(\int_{B_{r/2}(x,y)}|\partial_x\phi(s,t)|^6dsdt\right)^{\frac16}\\
&\quad+\frac{C}{r}\left(\int_{\R^2\setminus B_{r/2}(x,y)}|\partial_xq(s,t)|^2dsdt\right)^{\frac12}
\left(\int_{\R^2\setminus B_{r/2}(x,y)}|\partial_x\phi(s,t)|^2dsdt\right)^{\frac12}\\
&\leq\frac{C}{r}\|\phi\|_a\leq\frac{C}{r}\left(\|h_1\|_{b}+\|h_2\|_c\right),
\end{aligned}
\end{equation}
where we used the fact that $|\partial_xG(x,y)|\leq\frac{C}{r}$ (see Theorem \ref{tha.1}) and \eqref{3.abc}. While for the second term on the right hand side of \eqref{3.2-2}, we have
\begin{equation}
\label{3.2-2-2}
\begin{aligned}	
&\left|\int_{\R^2}\partial_xG(x-s,y-t)h_1(s,t)dsdt\right|\\
&\leq C\|h_1\|_{**}\int_{\R^2}\frac{1}{\sqrt{(x-s)^2+(y-t)^2}}\frac{1}{(1+s^2+t^2)^{5/2-\delta}}dsdt\\
&\leq\frac{C}{r}\|h_1\|_{**}.
\end{aligned}
\end{equation}
Similarly, one can show that
\begin{equation}
\label{3.2-2-2a}
\begin{aligned}	
\left|\int_{\R^2}\partial_yG(x-s,y-t)h_2(s,t)dsdt\right|\leq\frac{C\log\frac{1}{\e}}{r}\|h_2\|_{***}.
\end{aligned}
\end{equation}
Using \eqref{3.2-2-1}, \eqref{3.2-2-2} and \eqref{3.2-2-2a} we get that
\begin{equation}
\label{3.2-3}
\|(1+r)\phi\|_{L^\infty(\R^2)}\leq C\log\frac{1}{\e}\left(\|h_1\|_{**}+\|h_2\|_{***}\right).
\end{equation}
If we do not pursue the good decay of $\phi$, we could replace \eqref{3.2-2-2a} by the following one
\begin{equation}
\label{3.2-3a}
\begin{aligned}
&\left|\int_{\R^2}G(x-s,y-t)\partial_yh_2(s,t)dsdt\right|\\
&\leq \left|\int_{\R^2}\partial_xG(x-s,y-t)\partial_{x}^{-1}\partial_yh_2(s,t)dsdt\right|\\
&\leq \frac{C}{r^{2-\delta}}\|h_2\|_{***}\int_{B_{r/2}(x,y)}|\partial_xG(x-s,y-t)|dsdt\\
&\quad +C\|h_2\|_{***}\int_{\R^2\setminus B_{r/2}(x,y)}\frac{1}{\sqrt{(x-s)^2+(y-t)^2}(1+\sqrt{s^2+t^2})^{2-\delta}}dsdt\\
&\leq\frac{C}{r^{1-\delta}}\|h_2\|_{***}.
\end{aligned}
\end{equation}
Then by \eqref{3.2-2-1}, \eqref{3.2-2-2} and \eqref{3.2-3a} we get
\begin{equation}
\label{3.2-3b}
\|(1+r)^{1-\delta}\phi\|_{L^\infty(\R^2)}\leq C\left(\|h_1\|_{**}+\|h_2\|_{***}\right).
\end{equation}
Next we study the asymptotic behavior of $\partial_x\phi$, we get that
\begin{equation}
\label{3.2-4}
\begin{aligned}
\partial_x\phi(x,y)=~&-6(\sqrt{2}-\e^2)\int_{\R^2}\partial^2_xG(x-s,y-t)(\partial_xq\partial_x\phi)(s,t)dsdt\\
&-\int_{\R^2}\partial_x^2G(x-s,y-t)h_1(s,t)dsdt\\
&-\int_{\R^2}\partial_x\partial_yG(x-s,y-t)h_2(s,t)dsdt.
\end{aligned}
\end{equation}
As \eqref{3.2-2-1}-\eqref{3.2-2-2a}, we could use \eqref{3.abc} and Theorem \ref{tha.1} derive that
\begin{equation}
\label{3.2-5}
\left|\int_{\R^2}\partial_x^2G(x-s,y-t)(\partial_xq\partial_x\phi)(s,t)dsdt\right|\leq Cr^{-\frac32}\|\phi\|_a,
\end{equation}
where we have used the following inequality
\begin{equation*}
\begin{aligned}
\int_{B_{r/2}(0)}|\partial_x^2G(s,t)|^{\frac65}dsdt	
\leq~& \int_{B_1(0)}|\partial_x^2G(s,t)|^{\frac65}dsdt
+\int_{B_{r/2}(0)\setminus B_1(0)}|\partial_x^2G(s,t)|^{\frac65}dsdt
\\
\leq~&C\int_{B_1(0)}\left(\log\frac{1}{\sqrt{|t|}}\frac{1}{\sqrt{s^2+t^2}}\right)^{\frac65}dsdt\\
&+C\int_{B_{r/2}(0)\setminus B_1(0)}\left(\frac{1}{r^{3/2}}\right)^\frac65dsdt\leq Cr^{\frac15}.
\end{aligned}
\end{equation*}
For the second and third terms on the right hand side of \eqref{3.2-4}, we have
\begin{equation}
\label{3.2-6}
\begin{aligned}
&\left|\int_{\R^2}\partial_x^2G(x-s,y-t)h_1(s,t)dsdt\right|\\
&\leq \frac{C}{(1+r)^{\frac52-\delta}}\|h_1\|_{**}\int_{B_{r/2}(x,y)}|\partial_x^2G(x-s,y-t)|dsdt\\
&\quad+\frac{C}{r^{\frac32}}\|h_1\|_{**}\int_{\R^2\setminus B_{r/2}(x,y)}\frac{1}{\left(1+\sqrt{s^2+t^2}\right)^{5/2-\delta}}dsdt\\
&\leq Cr^{-\frac32}\|h_1\|_{**},
\end{aligned}
\end{equation}
where we used Theorem \ref{tha.1}. Similar as \eqref{3.2-6} we get
\begin{equation}
\label{3.2-6a}
\left|\int_{\R^2}\partial_x\partial_yG(x-s,y-t)h_2(s,t)dsdt\right|\leq C\e^{-\frac12}r^{-\frac32}\|h_2\|_{***}.
\end{equation}
Using \eqref{3.2-5}-\eqref{3.2-6a} we have
\begin{equation}
\label{3.2-7}	
\|(1+r)^{\frac32}\partial_x\phi\|_{L^\infty(R^2)}\leq 	C\left(\|h_1\|_{**}+\e^{-\frac12}\|h_2\|_{***}\right).
\end{equation}
As \eqref{3.2-3b}, if we do not pursue the good decay of $\partial_x\phi$, we could replace \eqref{3.2-6a} by the following one
\begin{equation}
\label{3.2-7a}
\begin{aligned}
&\left|\int_{\R^2}\partial_xG(x-s,y-t)\partial_yh_2(s,t)dsdt\right|\\
&=\left|\int_{\R^2}\partial_x^2G(x-s,y-t)\partial_{x}^{-1}\partial_yh_2(s,t)dsdt\right|\\
&\leq \frac{C}{(1+r)^{2-\delta}}\|h_2\|_{**}\int_{B_{r/2}(x,y)}|\partial_x^2G(x-s,y-t)|dsdt\\
&\quad+\frac{C}{r^{3/2}}\|h_2\|_{**}\int_{B_{r/2}(0)}\frac{1}{(1+\sqrt{s^2+t^2})^{2-\delta}}dsdt\\
&\quad+C\|h_2\|_{**}\int_{\R^2\setminus \left(B_{r/2}(x,y)\cup B_{r/2}(0)\right)}\frac{\left(\sqrt{(x-s)^2+(y-s)^2}\right)^{-\frac32}}{\left(1+\sqrt{s^2+t^2}\right)^{2-\delta}}dsdt \\
&\leq Cr^{-\frac32+\delta}\|h_2\|_{**},
\end{aligned}	
\end{equation}
where we used \eqref{3.x-1}. Then combined with \eqref{3.2-5} and \eqref{3.2-6} we get
\begin{equation}
\label{3.2-7b}
\|(1+r)^{\frac32-\delta}\partial_x\phi(x,y)\|_{L^\infty(\R^2)}\leq C(\|h_1\|_{**}+\|h_2\|_{***}).
\end{equation}
To study the term $\partial_x^2\phi$, we notice that
\begin{equation}
\label{3.2-8}
\begin{aligned}
\partial_x^2\phi=~&(6(\sqrt{2}-\e^2)\int_{\R^2}\partial_x^2G(x-s,y-t)\partial_x(\partial_xq\partial_x\phi)(s,t)dsdt\\
&+\int_{\R^2}\partial_x^2G(x-s,y-t)\partial_xh_1(s,t)dsdt\\
&+\int_{\R^2}\partial_x^2G(x-s,y-t)\partial_yh_2(s,t)dsdt.
\end{aligned}
\end{equation}
For the first term on the right hand side of \eqref{3.2-8}, we notice that
\begin{equation}
\label{3.2-9}
\begin{aligned}
&\left|\int_{\R^2}\partial_x^2G(x-s,y-t)\partial_x(\partial_xq\partial_x\phi)(s,t)dsdt\right|\\
&\leq\frac{C}{r^3}\left(\int_{B_{r/2}(x,y)}|\partial_x^2G(x-s,y-t)|^\frac{6}{5}dsdt\right)^{\frac56}\left(\int_{B_{r/2}(x,y)}|\partial_x\phi|^6dsdt\right)^{\frac16}\\
&\quad+\frac{C}{r^2}\left(\int_{B_{r/2}(x,y)}|\partial_x^2G(x-s,y-t)|^\frac{6}{5}dsdt\right)^{\frac56}\left(\int_{B_{r/2}(x,y)}|\partial_x^2\phi|^6dsdt\right)^{\frac16}\\
&\quad+\frac{C}{r^{3/2}}\left(\int_{\R^2\setminus B_{r/2}(x,y)}|\partial_xq|^2dsdt\right)^{\frac12}
\left(\int_{\R^2\setminus B_{r/2}(x,y)}|\partial_x^2\phi|^2dsdt\right)^{\frac12}\\
&\quad+\frac{C}{r^{3/2}}\left(\int_{\R^2\setminus B_{r/2}(x,y)}|\partial_x^2q|^2dsdt\right)^{\frac12}
\left(\int_{\R^2\setminus B_{r/2}(x,y)}|\partial_x\phi|^2dsdt\right)^{\frac12}\leq\frac{C}{r^{\frac32}}\|\phi\|_a.
\end{aligned}
\end{equation}
While for the second and third terms on the right hand side of \eqref{3.2-8}, following the argument as we did in \eqref{3.2-6} we have
\begin{equation}
\label{3.2-10}
\begin{aligned}
\left|\int_{\R^2}\partial_x^2G(x-s,y-t)\partial_xh_1(s,t)dsdt\right|&\leq Cr^{-\frac32}\|h_1\|_{**},\\
\left|\int_{\R^2}\partial_x^2G(x-s,y-t)\partial_yh_2(s,t)dsdt\right|&\leq
Cr^{-\frac32}\|h_2\|_{***}.
\end{aligned}	
\end{equation}
Therefore, from \eqref{3.2-9} and \eqref{3.2-10} we get that
\begin{equation}
\label{3.2-11}	
\|(1+r)^{\frac32}\partial_x^2\phi\|_{L^\infty(R^2)}\leq 	C(\|h_1\|_{**}+\|h_2\|_{**}).
\end{equation}
For $\partial_x^3\phi(x,y)$ we have
\begin{equation}
\label{3.2-13}
\begin{aligned}
\partial_x^3\phi(x,y)=~&(6(\sqrt{2}-\e^2)\int_{\R^2}\partial_x^3G(x-s,y-t)\partial_x(\partial_xq\partial_x\phi)(s,t)dsdt\\
&+\int_{\R^2}\partial_x^3G(x-s,y-t)(\partial_xh_1(s,t)+\partial_yh_2(s,t))dsdt.
\end{aligned}
\end{equation}
For the first term on the right hand side of \eqref{3.2-13}, we have
\begin{equation}
\label{3.2-13a}
\begin{aligned}
&\left|\int_{\R^2}\partial_x^3G(x-s,y-t)\partial_x(\partial_xq\partial_x\phi)(s,t)dsdt\right|\\
&\leq\frac{C}{r^3}\left(\int_{B_{r/2}(x,y)}|\partial_x^3G(x-s,y-t)|^\frac{6}{5}dsdt\right)^{\frac56}\left(\int_{B_{r/2}(x,y)}|\partial_x\phi|^6dsdt\right)^{\frac16}\\
&\quad+\frac{C}{r^2}\left(\int_{B_{r/2}(x,y)}|\partial_x^3G(x-s,y-t)|^\frac{6}{5}dsdt\right)^{\frac56}\left(\int_{B_{r/2}(x,y)}|\partial_x^2\phi|^6dsdt\right)^{\frac16}\\
&\quad+\frac{C\e^{-\frac12}}{r^{3/2}}\left(\int_{\R^2\setminus B_{r/2}(x,y)}|\partial_xq|^2dsdt\right)^{\frac12}
\left(\int_{\R^2\setminus B_{r/2}(x,y)}|\partial_x^2\phi|^2dsdt\right)^{\frac12}\\
&\quad+\frac{C\e^{-\frac12}}{r^{3/2}}\left(\int_{\R^2\setminus B_{r/2}(x,y)}|\partial_x^2q|^2dsdt\right)^{\frac12}
\left(\int_{\R^2\setminus B_{r/2}(x,y)}|\partial_x\phi|^2dsdt\right)^{\frac12}\\
&\leq C\e^{-\frac12}r^{-\frac32}\|\phi\|_a+C\e^{-\frac12}r^{-\frac32}\left(\|h_1\|_{**}+\|h_2\|_{***}\right),
\end{aligned}
\end{equation}
where we used
\begin{equation*}
\begin{aligned}
\int_{B_{r/2}(0)}|\partial_x^3G(s,t)|^{\frac65}dsdt\leq~&
\int_{B_1(0)}|\partial_x^3G(s,t)|^{\frac65}dsdt
+C\e^{-\frac35}\int_{B_{r/2}(0)\setminus B_1(0)}\frac{1}{(\sqrt{s^2+t^2})^{\frac95}}dsdt\\
\leq~&C+C\e^{-\frac35}r^{\frac{1}{10}}.
\end{aligned}	
\end{equation*}
For the second term on the right hand side of  \eqref{3.2-13}, as \eqref{3.2-6} we have that
\begin{equation}
\label{3.2-13b}
\left|\int_{\R^2}\partial_x^3G(x-s,y-t)\partial_xh_1(s,t)dsdt\right|\leq C\e^{-\frac12}r^{-\frac32}\|h_1\|_{**},
\end{equation}
and
\begin{equation}
\label{3.2-13c}
\left|\int_{\R^2}\partial_x^3G(x-s,y-t)\partial_yh_2(s,t)dsdt\right|\leq C\e^{-\frac12}r^{-\frac32}\|h_2\|_{***}.
\end{equation}
Hence from \eqref{3.2-13a} to \eqref{3.2-13c} we get that
\begin{equation}
\label{3.2-13d}
\|(1+r)^{\frac32}\partial_x^3\phi\|_{L^\infty(R^2)}\leq 	C\e^{-\frac12}(\|h_1\|_{**}+\|h_2\|_{**}).
\end{equation}
Concerning $\partial_x^4\phi(x,y)$, using Green's representation formula and integration by parts we get that
\begin{equation}
\label{3.2-14}
\begin{aligned}
\partial_x^4\phi(x,y)=~&-6(\sqrt{2}-\e^2)\int_{\R^2}\partial_x^3G(x-s,y-t)\partial_x^2(\partial_xq\partial_x\phi)(s,t)dsdt\\
&-\int_{\R^2}\partial_x^3G(x-s,y-t)\partial_{x}^2h_1(s,t)dsdt\\
&-\int_{\R^2}\partial_x^3G(x-s,y-t)\partial_{xy}h_2(s,t)dsdt.
\end{aligned}	
\end{equation}
For the first term on the right hand side of \eqref{3.2-14}, we have
\begin{equation}
\label{3.2-15}
\begin{aligned}
&\left|\int_{\R^2}\partial_x^3G(x-s,y-t)\partial_x^2(\partial_xq\partial_x\phi)(s,t)dsdt\right|\\
&\leq C\e^{-\frac12}r^{-\frac72}(\|h_1\|_{**}+\|h_2\|_{***})\int_{B_{r/2}(x,y)}|\partial^3_xG(x-s,y-t)|dsdt
\\
&\quad+\frac{C\e^{-\frac12}}{r^{3/2}}\left(\int_{\R^2\setminus B_{r/2}(x,y)}|\partial_xq|^2dsdt\right)^{\frac12}
\left(\int_{\R^2\setminus B_{r/2}(x,y)}|\partial_x^3\phi|^2dsdt\right)^{\frac12}\\
&\quad+\frac{C\e^{-\frac12}}{r^{3/2}}\left(\int_{\R^2\setminus B_{r/2}(x,y)}|\partial_x^2q|^2dsdt\right)^{\frac12}
\left(\int_{\R^2\setminus B_{r/2}(x,y)}|\partial_x^2\phi|^2dsdt\right)^{\frac12}\\
&\quad+\frac{C\e^{-\frac12}}{r^{3/2}}\left(\int_{\R^2\setminus B_{r/2}(x,y)}|\partial_x^3q|^2dsdt\right)^{\frac12}
\left(\int_{\R^2\setminus B_{r/2}(x,y)}|\partial_x\phi|^2dsdt\right)^{\frac12}\\
&\leq C\e^{-\frac12}r^{-\frac32}\left(\|h_1\|_{**}+\|h_2\|_{***}\right)+ C\e^{-\frac12}r^{-\frac32}\|\phi\|_a\\
&\leq C\e^{-\frac12}r^{-\frac32}\left(\|h_1\|_{**}+\|h_2\|_{***}\right),
\end{aligned}
\end{equation}
where we have used \eqref{3.2-7}, \eqref{3.2-11}, \eqref{3.2-13d} and Theorem \ref{tha.1}. For the other two terms on the right hand side of \eqref{3.2-14} we have
\begin{equation}
\label{3.2-15a}
\begin{aligned}
&\left|\int_{\R^2}\partial_x^3G(x-s,y-t)\partial_x^2h_1(s,t)dsdt\right|\\
&\leq \frac{C}{(1+r)^{\frac52-\delta}}\|h_1\|_{**}\int_{B_{r/2}(x,y)}|\partial_x^3G(x-s,y-t)|dsdt\\
&\quad+\frac{C\e^{-\frac12}}{r^{\frac32}}\|h_1\|_{**}\int_{\R^2\setminus B_{r/2}(x,y)}\frac{1}{\left(1+\sqrt{s^2+t^2}\right)^{5/2-\delta}}dsdt\\
&\leq C\e^{-\frac12}r^{-\frac32}\|h_1\|_{**}.
\end{aligned}
\end{equation}
Similarly
\begin{equation}
\label{3.2-15b}
\begin{aligned}
\left|\int_{\R^2}\partial_x^3G(x-s,y-t)\partial_x\partial_yh_2(s,t)dsdt\right|\leq C\e^{-\frac12}r^{-\frac32}\|h_2\|_{***}.
\end{aligned}
\end{equation}
Hence, we have
\begin{equation}
\label{3.2-16}
\|(1+r)^{\frac32}\partial_x^4\phi\|_{L^\infty(R^2)}\leq 	C\e^{-\frac12}(\|h_1\|_{**}+\|h_2\|_{**}).
\end{equation}

Step 2. In this step, we study the terms $\partial_y^{n}\phi(x,y)$ for $n=1,2,3.$ For $\partial_y\phi(x,y)$ we notice that
\begin{equation}
\label{3.3-1}
\begin{aligned}
\partial_y\phi(x,y)=~&-6(\sqrt{2}-\e^2)\int_{\R^2}\partial_x\partial_yG(x-s,y-t)(\partial_xq\partial_x\phi)(s,t)dsdt\\
&-\int_{\R^2}\partial_x\partial_yG(x-s,y-t)h_1(s,t)dsdt\\
&-\int_{\R^2}\partial_y^2G(x-s,y-t)h_2(s,t)dsdt.
\end{aligned}
\end{equation}
For the first term on the right hand side of \eqref{3.3-1} we use \eqref{3.2-3b} to derive that
\begin{equation}
\label{3.3-1a}
\begin{aligned}
&\left|\int_{\R^2}\partial_x\partial_yG(x-s,y-t)(\partial_xq\partial_x\phi)(s,t)dsdt\right|\\
&\leq \frac{C}{(1+r)^{7/2-\delta}}\left(\|h_1\|_{**}+\|h_2\|_{***}\right)\int_{B_{r/2}(x,y)}|\partial_x\partial_yG(x-s,y-t)|dsdt\\
&\quad+\frac{C\e^{-\frac12}}{r^{\frac32}}\left(\int_{\R^2\setminus B_{r/2}(x,y)}|\partial_xq|^2dsdt\right)^\frac12\left(\int_{\R^2\setminus B_{r/2}(x,y)}|\partial_x\phi|^2dsdt\right)^\frac12\\
&\leq C\e^{-\frac12}r^{-\frac32}\left(\|h_1\|_{**}+\|h_2\|_{***}\right).
\end{aligned}
\end{equation}
For the other two terms on the right hand side of \eqref{3.3-1} we have
\begin{equation}
\label{3.3-1b}
\begin{aligned}
&\left|\int_{\R^2}\partial_x\partial_yG(x-s,y-t)h_1(s,t)dsdt\right|\\
&\leq \frac{C}{(1+r)^{\frac52-\delta}}\|h_1\|_{**}\int_{B_{r/2}(x,y)}|\partial_x\partial_yG(x-s,y-t)|dsdt\\
&\quad+\frac{C\e^{-\frac12}}{r^{\frac32}}\|h_1\|_{**}\int_{\R^2\setminus B_{r/2}(x,y)}\frac{1}{\left(1+\sqrt{s^2+t^2}\right)^{5/2-\delta}}dsdt \\
&\leq C\e^{-\frac12}r^{-\frac32}\|h_1\|_{**}.
\end{aligned}	
\end{equation}
Similarly
\begin{equation}
\label{3.3-1c}
\begin{aligned}
\left|\int_{\R^2}\partial^2_yG(x-s,y-t)h_2(s,t)dsdt\right|\leq C\e^{-\frac32}r^{-\frac32}\|h_2\|_{***}.
\end{aligned}
\end{equation}
Then we derive that
\begin{equation}
\label{3.3-1d}
\|(1+r)^{\frac32}\partial_y\phi(x,y)\|_{L^\infty(\R^2)}\leq C(\e^{-\frac12}\|h_1\|_{**}+\e^{-\frac32}\|h_2\|_{***}).
\end{equation}
As \eqref{3.2-7a}, we could replace \eqref{3.3-1c} by the following one
\begin{equation}
\label{3.3-1e}
\begin{aligned}
&\left|\int_{\R^2}\partial_yG(x-s,y-t)\partial_yh_2(s,t)dsdt\right|\\
&=\left|\int_{\R^2}\partial_x\partial_yG(x-s,y-t)\partial_{x}^{-1}\partial_yh_2(s,t)dsdt\right|\\
&\leq \frac{C}{(1+r)^{2-\delta}}\|h_2\|_{**}\int_{B_{r/2}(x,y)}|\partial_x\partial_yG(x-s,y-t)|dsdt\\
&\quad+\frac{C}{\e^{\frac12}r^{3/2}}\|h_2\|_{**}\int_{B_{r/2}(0)}\frac{1}{(1+\sqrt{s^2+t^2})^{2-\delta}}dsdt\\
&\quad+C\e^{-\frac12}\|h_2\|_{**}\int_{\R^2\setminus \left(B_{r/2}(x,y)\cup B_{r/2}(0)\right)}\frac{\left(\sqrt{(x-s)^2+(y-s)^2}\right)^{-\frac32}}{\left(1+\sqrt{s^2+t^2}\right)^{2-\delta}}dsdt \\
&\leq C\e^{-\frac12}r^{-\frac32+\delta}\|h_2\|_{**}.	
\end{aligned}	
\end{equation}
Then combined with \eqref{3.3-1b} and \eqref{3.3-1c} we could get a estimation on $\partial_y\phi(x,y)$ with a weaker decay and better coefficient
\begin{equation}
\label{3.3-1f}
\|(1+r)^{\frac32-\delta}\partial_y\phi(x,y)\|_{L^\infty(\R^2)}\leq C\e^{-\frac12}(\|h_1\|_{**}+\|h_2\|_{***}).
\end{equation}
For $\partial_y^2\phi(x,y)$, by Green representation formula we have
\begin{equation}
\label{3.3-2}
\begin{aligned}
\partial_y^2\phi(x,y)=~&6(\sqrt{2}-\e^2)\int_{\R^2}\partial_y^2G(x-s,y-t)\partial_x(\partial_xq\partial_x\phi)(s,t)dsdt\\
&+\int_{\R^2}\partial^2_yG(x-s,y-t)\partial_xh_1(s,t)dsdt\\
&+\int_{\R^2}\partial_y^2G(x-s,y-t)\partial_yh_2(s,t)dsdt.
\end{aligned}
\end{equation}
Using Theorem \ref{tha.1}, \eqref{3.2-7} and \eqref{3.2-11}, as \eqref{3.3-1a}-\eqref{3.3-1c} we get
\begin{equation}
\label{3.3-2a}
\left|\int_{\R^2}\partial_y^2G(x-s,y-t)\partial_x(\partial_xq\partial_x\phi)(s,t)dsdt\right|\leq C\e^{-\frac32}r^{-\frac32}(\|h_1\|_{**}+\|h_2\|_{***}),	
\end{equation}
\begin{equation}
\label{3.3-2b}
\left|\int_{\R^2}\partial_y^2G(x-s,y-t)\partial_xh_1(s,t)dsdt\right|\leq C\e^{-\frac32}r^{-\frac32}\|h_1\|_{**},
\end{equation}
and
\begin{equation}
\label{3.3-2c}
\left|\int_{\R^2}\partial_y^2G(x-s,y-t)\partial_yh_2(s,t)dsdt\right|\leq C\e^{-\frac32}r^{-\frac32}\|h_2\|_{***}.
\end{equation}
Combining \eqref{3.3-2a}-\eqref{3.3-2c} we get
\begin{equation}
\label{3.3-2d}
\|(1+r)^{\frac32}\partial_y^2\phi(x,y)\|_{L^\infty(\R^2)}\leq C\e^{-\frac32}(\|h_1\|_{**}+\|h_2\|_{***}).
\end{equation}
Concerning $\partial_y^3\phi(x,y)$, we have
\begin{equation}
\label{3.3-3}
\begin{aligned}
\partial_y^3\phi(x,y)=~&6(\sqrt{2}-\e^2)\int_{\R^2}\partial_y^3G(x-s,y-t)\partial_x(\partial_xq\partial_x\phi)(s,t)dsdt\\
&+\int_{\R^2}\partial^3_yG(x-s,y-t)\partial_xh_1(s,t)dsdt\\
&+\int_{\R^2}\partial_y^2G(x-s,y-t)\partial_yh_2(s,t)dsdt.
\end{aligned}
\end{equation}
We could use Theorem \ref{tha.1}, and follow the same argument as \eqref{3.3-2a}-\eqref{3.3-2d} to obtain that
\begin{equation}
\label{3.3-3d}
\|(1+r)^{\frac32}\partial_y^3\phi(x,y)\|_{L^\infty(\R^2)}\leq C\e^{-\frac72}(\|h_1\|_{**}+\|h_2\|_{***}).	
\end{equation}
\medskip

Step 3. Now we consider $\partial_x\partial_y\phi$, $\partial_x^2\partial_y\phi,$ $\partial_x\partial_y^2\phi$ and $\partial_x^2\partial_y^2\phi$. By representation formula we have
\begin{equation}
\label{3.3-4}
\begin{aligned}
\partial_x\partial_y\phi(x,y)=&~6(\sqrt{2}-\e^2)\int_{\R^2}\partial_x\partial_yG(x-s,y-t)\partial_x(\partial_xq\partial_x\phi)(s,t)dsdt\\
&+\int_{\R^2}\partial_x\partial_yG(x-s,y-t)\partial_xh_1(s,t)dsdt\\
&+\int_{\R^2}\partial_x\partial_yG(x-s,y-t)\partial_yh_2(s,t)dsdt.
\end{aligned}
\end{equation}
For the first term on the right hand side of \eqref{3.3-4}, by Theorem \ref{tha.1}, \eqref{3.2-7b} and \eqref{3.2-11}, we get
\begin{equation}
\label{3.3-4a}
\begin{aligned}
\left|\int_{\R^2}\partial_x\partial_yG(x-s,y-t)\partial_x(\partial_xq\partial_x\phi)(s,t)dsdt\right|\leq \frac{C\e^{-\frac12}}{r^{\frac32}}\left(\|h_1\|_{**}+\|h_2\|_{***}\right).
\end{aligned}	
\end{equation}
For the other two terms on the right hand side of \eqref{3.3-4}, we have
\begin{equation}
\label{3.3-4b}
\begin{aligned}
\left|\int_{\R^2}\partial_x\partial_yG(x-s,y-t)\partial_xh_1(s,t)dsdt\right|\leq C\e^{-\frac12}r^{-\frac32}\|h_1\|_{**},
\end{aligned}	
\end{equation}
and
\begin{equation}
\label{3.3-4c}
\begin{aligned}
\left|\int_{\R^2}\partial_x\partial_yG(x-s,y-t)\partial_yh_2(s,t)dsdt\right|\leq C\e^{-\frac12}r^{-\frac32}\|h_2\|_{***}.
\end{aligned}	
\end{equation}
The above three inequalities \eqref{3.3-4a1}, \eqref{3.3-4b1} and \eqref{3.3-4c1} give that
\begin{equation}
\label{3.3-4d}
\|(1+r)^{\frac32}\partial_x\partial_y\phi(x,y)\|_{L^\infty(\R^2)}\leq  C\e^{-\frac12}\left(\|h_1\|_{**}+\|h_2\|_{***}\right).
\end{equation}
For the term $\partial_x^2\partial_y\phi$, we have
\begin{equation}
\label{3.3-4-1}
\begin{aligned}
\partial_x^2\partial_y\phi(x,y)=&~-6(\sqrt{2}-\e^2)\int_{\R^2}\partial_x\partial_yG(x-s,y-t)\partial_x^2(\partial_xq\partial_x\phi)(s,t)dsdt\\
&-\int_{\R^2}\partial_x\partial_yG(x-s,y-t)\partial^2_xh_1(s,t)dsdt\\
&-\int_{\R^2}\partial_x\partial_yG(x-s,y-t)\partial_x\partial_yh_2(s,t)dsdt.
\end{aligned}
\end{equation}
For the first term on the right hand side of \eqref{3.3-4-1}, by Theorem \ref{tha.1}, \eqref{3.2-7}, \eqref{3.2-11} and  \eqref{3.2-13d}, following almost the same arguments in \eqref{3.2-15} we get
\begin{equation}
\label{3.3-4a1}
\begin{aligned}
\left|\int_{\R^2}\partial_x\partial_yG(x-s,y-t)\partial_x^2(\partial_xq\partial_x\phi)(s,t)dsdt\right|\leq \frac{C\e^{-\frac12}}{r^{\frac32}}\left(\|h_1\|_{**}+\|h_2\|_{***}\right),
\end{aligned}	
\end{equation}
For the other two terms on the right hand side of \eqref{3.3-4-1}, as \eqref{3.3-1b}-\eqref{3.3-1c} we have
\begin{equation}
\label{3.3-4b1}
\begin{aligned}
\left|\int_{\R^2}\partial_x\partial_yG(x-s,y-t)\partial_x^2h_1(s,t)dsdt\right|\leq C\e^{-\frac12}r^{-\frac32}\|h_1\|_{**},
\end{aligned}	
\end{equation}
and
\begin{equation}
\label{3.3-4c1}
\begin{aligned}
\left|\int_{\R^2}\partial_x\partial_yG(x-s,y-t)\partial_x\partial_yh_2(s,t)dsdt\right|\leq C\e^{-\frac12}r^{-\frac32}\|h_2\|_{***}.
\end{aligned}	
\end{equation}
The above three inequalities give that
\begin{equation}
\label{3.3-4d1}
\|(1+r)^{\frac32}\partial_x^2\partial_y\phi(x,y)\|_{L^\infty(\R^2)}\leq C\e^{-\frac12}\left(\|h_1\|_{**}+\|h_2\|_{***}\right).
\end{equation}
Concerning $\partial_x\partial_y^2\phi$, using integration by parts we have
\begin{equation}
\label{3.3-6}
\begin{aligned}
\partial_x\partial_y^2\phi=&-
6(\sqrt{2}-\e^2)\int_{\R^2}\partial_y^2G(x-s,y-t)\partial_x^2(\partial_xq\partial_x\phi)(s,t)dsdt\\
&-\int_{\R^2}\partial_y^2G(x-s,y-t)\partial_x^2h_1(s,t)dsdt\\
&-\int_{\R^2}\partial_y^2G(x-s,y-t)\partial_x\partial_yh_2(s,t)dsdt.
\end{aligned}	
\end{equation}
Using \eqref{3.2-7}, \eqref{3.2-11}, \eqref{3.2-13d} and Theorem \ref{tha.1}, we get
\begin{equation}
\label{3.3-6a}
\begin{aligned}
&\left|\int_{\R^2}\partial_y^2G(x-s,y-t)\partial_x^2(\partial_xq\partial_x\phi)(s,t)dsdt\right|\\
&\leq~\frac{C}{r^{7/2}}\e^{-\frac12}\left(\|h_1\|_{**}+\|h_2\|_{***}\right)\int_{B_{r/2}(x,y)}|\partial_y^2G(x-s,y-t)|dsdt\\
&\quad+\frac{C}{r^{3/2}}\e^{-\frac32}\left(\int_{\R^2\setminus B_{r/2}(x,y)}|\partial_xq|^2dsdt\right)^{\frac12}\left(\int_{\R^2\setminus B_{r/2}(x,y)}|\partial_x^3\phi|^2 dsdt\right)^{\frac12}\\
&\quad+\frac{C}{r^{3/2}}\e^{-\frac32}\left(\int_{\R^2\setminus B_{r/2}(x,y)}|\partial_x^2q|^2dsdt\right)^{\frac12}\left(\int_{\R^2\setminus B_{r/2}(x,y)}|\partial_x^2\phi|^2 dsdt\right)^{\frac12}\\
&\quad+\frac{C}{r^{3/2}}\e^{-\frac32}\left(\int_{\R^2\setminus B_{r/2}(x,y)}|\partial_x^3q|^2dsdt\right)^{\frac12}\left(\int_{\R^2\setminus B_{r/2}(x,y)}|\partial_x \phi|^2 dsdt\right)^{\frac12}\\
&\leq C\e^{-\frac32}r^{-\frac32}\left(\|h_1\|_{**}+\|h_2\|_{***}\right),	
\end{aligned}	
\end{equation}
and as \eqref{3.3-1b}-\eqref{3.3-1c} we get that
\begin{equation}
\label{3.3-6b}
\begin{aligned}
\left|\int_{\R^2}\partial_y^2G(x-s,y-t)\partial_x^2h_1(s,t)dsdt\right|\leq C\e^{-\frac32}r^{-\frac32}\|h_1\|_{**},
\end{aligned}	
\end{equation}
\begin{equation}
\label{3.3-6c}
\begin{aligned}
\left|\int_{\R^2}\partial_y^2G(x-s,y-t)\partial_x\partial_yh_2(s,t)dsdt\right|\leq C\e^{-\frac32}r^{-\frac32}\|h_2\|_{***}.
\end{aligned}	
\end{equation}
Combining \eqref{3.3-6a}, \eqref{3.3-6b} and \eqref{3.3-6c} we have
\begin{equation}
\label{3.3-6d0}
\|(1+r)^{\frac32}\partial_x\partial_y^2\phi(x,y)\|_{L^\infty(\R^2)}\leq C\e^{-\frac32}\left(\|h_1\|_{**}+\|h_2\|_{***}\right).	
\end{equation}
Following almost the same arguments from \eqref{3.3-6} to \eqref{3.3-6c} we have
\begin{equation}
\label{3.3-6d}
\|(1+r)^{\frac32}\partial_x^2\partial_y^2\phi(x,y)\|_{L^\infty(\R^2)}\leq C\e^{-\frac52}\left(\|h_1\|_{**}+\|h_2\|_{***}\right).	
\end{equation}
\medskip

Step 4. In this step we estimate $\partial_x^{-1}\partial_y^2\phi$ and $\partial_x^{-1}\partial_y^3\phi$, for the former one, using Green's representation formula, we have
\begin{equation}
\label{3.3-7}
\begin{aligned}
\partial_x^{-1}\partial_y^2\phi=~&-
6(\sqrt{2}-\e^2)\int_{\R^2}\partial_y^2G(x-s,y-t)(\partial_xq\partial_x\phi)(s,t)dsdt\\
&-\int_{\R^2}\partial_y^2G(x-s,y-t)h_1(s,t)dsdt\\
&-\int_{\R^2}\partial^2_yG(x-s,y-t)\partial_x^{-1}\partial_yh_2(s,t)dsdt.	
\end{aligned}
\end{equation}
For the first one on the right hand side of \eqref{3.3-7}, we have
\begin{equation}
\label{3.3-7a}
\begin{aligned}
&\left|\int_{\R^2}\partial_y^2G(x-s,y-t)(\partial_xq\partial_x\phi)(s,t)dsdt\right|\\
&\leq Cr^{-7/2+\delta}\left(\|h_1\|_{**}+\|h_2\|_{***}\right)\int_{B_{r/2}(x,y)}|\partial_y^2G(x-s,y-t)|dsdt\\
&\quad +C\e^{-\frac32}r^{-\frac32}\left(\int_{\R^2\setminus B_{r/2}(x,y)}
|\partial_xq|^2dsdt\right)^\frac12\left(\int_{\R^2\setminus B_{r/2}(x,y)}
|\partial_x\phi|^2dsdt\right)^\frac12\\
&\leq C\e^{-\frac32}r^{-\frac32}\left(\|h_1\|_{**}+\|h_2\|_{***}\right).
\end{aligned}	
\end{equation}
For the second term on the right hand side of \eqref{3.3-7} we get that
\begin{equation}
\label{3.3-7b}
\begin{aligned}
&\left|\int_{\R^2}\partial_y^2G(x-s,y-t)h_1(s,t)dsdt\right|\\
&\leq Cr^{-\frac52+\delta}\|h_1\|_{**}\int_{B_{r/2}(x,y)}|\partial_y^2G(x-s,y-t)|dsdt\\
&\quad+C\e^{-\frac32}r^{-\frac32}\int_{\R^2\setminus B_{r/2}(x,y)}|h_1(s,t)|dsdt\\
&\leq C\e^{-\frac32}r^{-\frac32}\|h_1\|_{**},	
\end{aligned}
\end{equation}
while for the third term on the right hand side of \eqref{3.3-7}, by \eqref{3.x-1} we have
\begin{equation}
\label{3.3-7c}
\begin{aligned}
&\left|\int_{\R^2}\partial_y^2G(x-s,y-t)\partial_x^{-1}\partial_yh_2(s,t)dsdt\right|\\
&\leq Cr^{-2+\delta}\|h_2\|_{***}\int_{B_{r/2}(x,y)}|\partial_y^2G(x-s,y-t)|dsdt\\
&\quad+C\e^{-\frac32}r^{-\frac32}\|h_2\|_{***}\int_{B_{r/2}(0)}\frac{1}{(1+\sqrt{s^2+t^2})^{2-\delta}}dsdt\\
&\quad+C\e^{-\frac32}\|h_2\|_{***}\int_{\R^2\setminus(B_{r/2}(x,y)\cup B_{r/2}(0))}\frac{C(1+\sqrt{s^2+t^2})^{-2+\delta}}{(\sqrt{(x-s)^2+(y-t)^2})^{3/2}}dsdt\\
&\leq C\e^{-\frac32}r^{-\frac32+\delta}\|h_2\|_{**}.
\end{aligned}
\end{equation}
Then from \eqref{3.3-7a}-\eqref{3.3-7c} we get that
\begin{equation}
\label{3.3-7d}
\|(1+r)^{\frac32-\delta}\partial_x^{-1}\partial_y^2\phi\|_{L^\infty(\R^2)}\leq C\e^{-\frac32}\left(\|h_1\|_{**}+\|h_2\|_{***}\right).	
\end{equation}
Similarly, we could derive that
\begin{equation}
\label{3.3-8}
\|(1+r)^{\frac32-\delta}\partial_x^{-1}\partial_y^3\phi\|_{L^\infty(\R^2)}\leq C\e^{-\frac72}\left(\|h_1\|_{**}+\|h_2\|_{***}\right).	
\end{equation}
\medskip

Step 5. In the final step we estimate the terms $\partial_y^4\phi$ and $\partial^{-1}_x\partial_y^4\phi$. For the former one, using \eqref{3.2-11}, \eqref{3.2-16}, \eqref{3.3-2d}, \eqref{3.3-6d} and equation \eqref{3.2-eq}, we get
\begin{equation}
\label{3.3-9}
\e^4\|(1+r)^{\frac32}\partial_y^4\phi\|_{L^\infty(\R^2)}	
\leq C\e^{-\frac32}\left(\|h_1\|_{**}+\|h_2\|_{***}\right).
\end{equation}
Concerning the last term $\partial_x^{-1}\partial_y^4\phi$, we take $\partial_x^{-1}$ on both sides of \eqref{3.2-eq}, and we get that
\begin{equation}
\label{3.3-10}
\e^4\|(1+r)^{\frac32-\delta}\partial_x^{-1}\partial_y^4\phi\|_{L^\infty(\R^2)}	
\leq C\e^{-\frac32}\left(\|h_1\|_{**}+\|h_2\|_{***}\right).
\end{equation}
Then we finish the whole proof.
\end{proof}

\section{ The nonlinear system and proof of Theorem \ref{app}}

\indent In section 2, we have seen the leading terms of both equations \eqref{equa1} and \eqref{equa2} can be reduced to the following single equation
\begin{equation}
	\label{1st-order}
	f_1=\frac{\sqrt{2}}{2}\partial_x g_1-\frac{g_1^2}{2}.
\end{equation}
Once $g_1$ is given, then we are able to solve $O(1)$ terms in both \eqref{equa1} and \eqref{equa2} by finding $f_1$ through \eqref{1st-order}. To derive a solution to the original GP equation \eqref{GPe}, we need to solve the equations \eqref{2.equation} and \eqref{Per2}:
\begin{equation}
\label{4-2}
\begin{aligned}
&\partial_x^4g_1-(2\sqrt{2}-\e^2)\partial_x^2g_1-3(\sqrt{2}-\e^2)\partial_x((\partial_xg_1)^2)-2\partial_y^2g_1+2\e^2\partial_x^2\partial_y^2g_1+\e^4\partial_y^4g_1\\
&=P_1+P_2+P_3,	
\end{aligned}	
\end{equation}
and
\begin{equation}
\label{eq:f2}
(\sqrt{2}-\e^2)\partial_xf_2+2g_1f_2=\partial_y^2g_1+\partial_xf_1-(f_1+\e^2f_2)^2g_1,
\end{equation}
where $P_1,~P_2$ and $P_3$ are given in \eqref{2.dec-p1}-\eqref{2.dec-p3}.
\medskip

We look for solutions of \eqref{4-2} such that $g_1=q+\phi$, with $q$ defined in \eqref{2.lump}, then equation \eqref{4-2} could be written as
\begin{equation}
\label{4-2a}
\begin{aligned}
&\partial_x^4\phi-(2\sqrt{2}-\e^2)\partial_x^2\phi-6(\sqrt{2}-\e^2)\partial_x(q\phi)-2\partial_y^2\phi+2\e^2\partial_x^2\partial_y^2\phi+\e^4\partial_y^4\phi\\
&=-\partial_x^4q+(2\sqrt{2}-\e^2)\partial_x^2q+3(\sqrt{2}-\e^2)\partial_x((\partial_xq)^2)+2\partial_y^2q-2\e^2\partial_x^2\partial_y^2q-\e^4\partial_y^4q\\
&\quad+3(\sqrt{2}-\e^2)\partial_x((\partial_x\phi)^2)+P_1+P_2+P_3.	
\end{aligned}	
\end{equation}
We set
\begin{equation}
\label{4.q}
\Gamma_q=-\partial_x^4q+(2\sqrt{2}-\e^2)\partial_x^2q+3(\sqrt{2}-\e^2)\partial_x((\partial_xq)^2)+2\partial_y^2q-2\e^2\partial_x^2\partial_y^2q-\e^4\partial_y^4q.	
\end{equation}
It is not difficult to see that $\Gamma_q\in\mathcal{H}_{ox}$ and
\begin{equation}
\label{4.q-est}
|\Gamma_q|+|\partial_x\Gamma_q|\leq C\e^2(1+r)^{-5}.
\end{equation}
As a consequence,
\begin{equation}
\label{4.q-est1}
\partial_x^{-1}\Gamma_q=-\int_x^\infty\Gamma_q(s,y)ds\quad \mbox{is well defined and we have}~\|\partial_x^{-1}\Gamma_q\|_{**}\leq C\e^2.	
\end{equation}
The estimation \eqref{4.q-est1} suggests that we look for solutions of \eqref{4-2a} in the following space
\begin{equation}
\label{4.fuc}
\mathcal{F}_\phi=\{\phi\in\mathcal{F}_1\mid \|\phi\|_{*}\leq C\e^2\},	
\end{equation}
with $C$ sufficiently large. On the other hand, it is not difficult to check that $\|(\partial_x\phi)^2\|_{**}\leq C\e^2$. Therefore we can absorb $\Gamma_q$ and $3(\sqrt{2}-\e^2)\partial_x((\partial_x\phi)^2)$ into $P_1$.
\medskip


In the following, we shall analyze \eqref{eq:f2}. For a given $g_1=q+\phi$ with $\phi\in\mathcal{F}_\phi$, we solve $f_2$ of \eqref{eq:f2} in the following space
\begin{equation}
\label{def-f3}
\mathcal{F}_{f_2}:=\{f\mid f\in\mathcal{H}_e,~\|f\|_{****}<C\},
\end{equation}
where
\begin{equation}
\begin{aligned}
\|f\|_{****}=~&\|(1+r)^{\frac32-\delta}f\|_{L^\infty(\R^2)}+\|(1+r)^{\frac32-\delta}\partial_xf\|_{L^\infty(\R^2)}\\
&+\|(1+r)^{\frac32-\delta}\partial_x^2f\|_{L^\infty(\R^2)}+\e\|(1+r)^{\frac32-\delta}\partial_x^3f\|_{L^\infty(\R^2)}\\
&+\e^2\|(1+r)^{\frac32-\delta}\partial_yf\|_{L^\infty(\R^2)}+\e^2\|(1+r)^{\frac32-\delta}\partial_x\partial_yf\|_{L^\infty(\R^2)}\\
&+\e^4\|(1+r)^{\frac32-\delta}\partial_y^2f\|_{L^\infty(\R^2)}+\e^4\|(1+r)^{\frac32-\delta}\partial_x\partial_y^2f\|_{L^\infty(\R^2)}.
\end{aligned}
\end{equation}

\begin{lemma}
\label{Le2.1}
Let $\phi\in\mathcal{F}_\phi$. Then for sufficiently small $\varepsilon$ there exists a solution $f_{2,\phi}\in\mathcal{F}_{f_2}$ of \eqref{eq:f2}.
In addition, for any two functions $\phi_1,~\phi_2\in\mathcal{F}_\phi$, we have the corresponding solutions $f_{2,\phi_1}$ and $f_{2,\phi_2}$ satisfying
\begin{equation}
\label{4f-phi}
\begin{aligned}
\|f_{2,\phi_1}-f_{2,\phi_2}\|_{****}\leq C\e^{-\frac32}\|\phi_1-\phi_2\|_{*}.
\end{aligned}
\end{equation}
\end{lemma}

\begin{proof}
We recall that $f_{1}$ is determined by $g_{1}$ by \eqref{1st-order}. For each given $g_{1},$ we shall solve  \eqref{eq:f2}  by a perturbation argument. To this aim, for each fixed $y\in\mathbb{R},$ let us consider the non-homogeneous equation
\begin{equation}
\label{4.non-h}
\left(  \sqrt{2}-\varepsilon^{2}\right)  \partial_{x}F_\mathfrak{h}+2qF_\mathfrak{h}=\mathfrak{h}.
\end{equation}
It is known the homogeneous equation
\[\left(  \sqrt{2}-\varepsilon^{2}\right)  \partial_{x}F_0+2qF_0=0\]
has a solution of the form
\[F_{0}\left(  x,y\right)  =\exp\left(  -\int_{-\infty}^{x}\frac
{2q\left( s,y\right)  }{\sqrt{2}-\varepsilon^{2}}ds\right).\]
Using the expression
$$q(x,y)=-\left(\frac{2\sqrt{2}}{2\sqrt{2}-\e^2}\right)^2\dfrac{\sqrt{8-2\sqrt{2}\e^2}x}{\frac{2\sqrt{2}-\e^2}{2\sqrt{2}}x^2+\frac{(2\sqrt{2}-\e^2)^2}{4\sqrt{2}}y^2+\frac{3}{2\sqrt{2}}},$$
we derive that
\begin{equation*}
\int_{-\infty}^{x}q\left( s,y\right)  ds\quad\mbox{is close to}\quad-
D_\e\log\left(\frac{2\sqrt{2}-\e^2}{2\sqrt{2}}x^2+\frac{(2\sqrt{2}-\e^2)^2}{4\sqrt{2}}y^2+\frac{3}{2\sqrt{2}}\right),
\end{equation*}
where
$$D_\e=\left(\frac{2\sqrt{2}}{2\sqrt{2}-\e^2}\right)^3\sqrt{2-\frac{\sqrt{2}}{2}\e^2}.$$
As a consequence,
\begin{equation}
\label{eta-decay}
F_{0}\quad\mbox{is close to}\quad \left(x^{2}+\sqrt{2}y^{2}+\frac{3}{2\sqrt{2}}\right)^2.
\end{equation}
In particular, $F_{0}$ is not decaying at infinity and even in $x$. Applying the variation of parameter formula, we can write the solution $F_\mathfrak{h}$ of \eqref{4.non-h} as the following form
\[F_{\mathfrak{h}}=F_{0}\left(  x,y\right)  \int_{-\infty}^{x}\frac{\mathfrak{h}\left(
	s,y\right)  }{F_{0}\left(  s,y\right)  }ds,\]
where the integral makes sense due to the boundedness of $\mathfrak{h}$ and the decay property of $F_0^{-1}$, see \eqref{eta-decay}.

For a given $\phi\in\mathcal{F}_\phi$ and $f_2\in\mathcal{F}_{f_2}$, we write equation \eqref{eq:f2} as
\begin{equation}
\label{eq:f3}
\begin{aligned}
(\sqrt{2}-\e^2)\partial_xf_2+2qf_2=&-2\phi f_2+\partial_y^2(q+\phi)+\partial_x
\left(\frac{\sqrt{2}}{2}\partial_x(q+\phi)-\frac12(q+\phi)^2\right)\\
&-\left[\left(\frac{\sqrt{2}}{2}\partial_x(q+\phi)-\frac12(q+\phi)^2\right)+\e^2f_2\right]^2(q+\phi).
\end{aligned}
\end{equation}
Denoting the right hand side of \eqref{eq:f3} by $E_\phi$ and we define the map $f_2\to F_{E_\phi}$ by $\mathcal{M}(f_2)$.

We shall first derive a solution to \eqref{eq:f3} in the following space
\begin{equation}
\label{4f-sec-def}	
\begin{aligned}
\widetilde{\mathcal{F}}_{f_2}:=\left\{f\mid f\in\mathcal{H}_e,~\|f\|_{*****}<C\right\},	
\end{aligned}
\end{equation}
where
\begin{equation}
\begin{aligned}
\|f\|_{*****}=~&\|(1+r)^{\frac32-\delta}f\|_{L^\infty(\R^2)}+\|(1+r)^{\frac32-\delta}\partial_xf\|_{L^\infty(\R^2)}\\
&+\e^2\|(1+r)^{\frac32-\delta}\partial_yf\|_{L^\infty(\R^2)}+\e^4\|(1+r)^{\frac32-\delta}\partial_y^2f\|_{L^\infty(\R^2)}.
\end{aligned}
\end{equation}
We claim that
\begin{equation}
\label{4.con-1}
\mathcal{M}(f_2)\in\widetilde{\mathcal{F}}_{f_2}.
\end{equation}
Let us analyze the terms defined in $\|\cdot\|_{*****}$ term by term. At first, we write
\begin{equation}
\label{eq:f4}
\begin{aligned}
\mathcal{M}(f_2)=&-F_0(x,y)\int_{x}^\infty\frac{E_\phi(s,y)}{F_0(s,y)}ds\\
=&~\partial_x^{-1}E_\phi(x,y)-F_0(x,y)\int_x^\infty\frac{\partial_x^{-1}E_\phi(s,y)}{F_0^2(s,y)}\partial_xF_0(s,y)ds.
\end{aligned}
\end{equation}
For $f_2\in\mathcal{F}_{f_2}$. By the definition of $E_\phi$, $\|\phi\|_*\leq C\e^2$ and $\|f_2\|_{*****}\leq C,$  it is not difficult to check that
\begin{equation}
\label{4-2.1}
\begin{aligned}
|\partial_x^{-1}E_\phi(x,y)|\leq~&C(1+r)^{-\frac32+\delta}\left(1+\log\frac{1}{\e}\|f\|_{*****}\|\phi\|_*+\e^{-\frac32}\|\phi\|_*\right)\\
\leq~&C(1+r)^{-\frac32+\delta}(1+\e^{-\frac32}\|\phi\|_{*}).
\end{aligned}
\end{equation}
Together with \eqref{eq:f4} we get that
\begin{equation}
\label{4-2.2}
\|(1+r)^{\frac32-\delta}\mathcal{M}(f_2)\|_{L^\infty(\R^2)}\leq C(1+\e^{-\frac32}\|\phi\|_{*}).
\end{equation}
Concerning the term $\partial_xf_2$, by \eqref{eq:f4} we have
\begin{equation}
\label{4-2.2a}
\partial_x\mathcal{M}(f_2)=-\partial_xF_0(x,y)\int_x^{\infty}\frac{E_\phi(s,y)}{F_0(s,y)}ds+E_\phi(x,y).	
\end{equation}
For the right hand side of \eqref{eq:f3} we have
\begin{equation}
\label{4-2.3}
|E_\phi(x,y)|\leq C(1+r)^{-\frac32}\left(1+\|f_2\|_{*****}\|\phi\|_*+\e^{-\frac32}\|\phi\|_*\right).
\end{equation}
Then it is not difficult to verify that
\begin{equation}
\label{4-2.4}
\|(1+r)^{\frac32-\delta}\partial_x\mathcal{M}(f_2)\|_{L^\infty(\R^2)}\leq C(1+\e^{-\frac32}\|\phi\|_{*}).	
\end{equation}
Next, we study the terms $\partial_y\mathcal{M}(f)$ and $\partial_y^2\mathcal{M}(f)$, by direct computation
\begin{equation}
\label{4-2.6}
\begin{aligned}
\partial_y\mathcal{M}(f_2)=&-\partial_yF_0(x,y)\int_x^\infty\frac{E_\phi(s,y)}{F_0(s,y)}ds+F_0(x,y)\int_x^\infty\frac{E_\phi(s,y)\partial_yF_0(s,y)}{F_0^2(s,y)}ds\\
&-F_0(x,y)\int_x^\infty\frac{\partial_yE_\phi(s,y)}{F_0(s,y)}ds.
\end{aligned}
\end{equation}
Using \eqref{4-2.3} we can estimate the first two terms of \eqref{4-2.6} by
\begin{equation}
\label{4-2.7}	
\begin{aligned}
&\left|\partial_yF_0(x,y)\int_x^\infty\frac{E_\phi(s,y)}{F_0(s,y)}ds\right|+\left|F_0(x,y)\int_x^\infty\frac{E_\phi(s,y)\partial_yF_0(s,y)}{F_0^2(s,y)}ds\right|\\
&\leq C(1+r)^{-\frac32}\left(1+\|f_2\|_{*****}\|\phi\|_*+\e^{-\frac32}\|\phi\|_*\right).	
\end{aligned}	
\end{equation}
While for the last term on the right hand side of \eqref{4-2.6}, as \eqref{eq:f4} we write it as
\begin{equation}
\label{4-2.8}
\begin{aligned}
F_0(x,y)\int_x^\infty\frac{\partial_yE_\phi(s,y)}{F_0(s,y)}ds=~&F_0(x,y)\int_x^\infty\frac{\partial_x^{-1}\partial_yE_\phi(s,y)}{F_0^2(s,y)}\partial_xF_0(s,y)ds\\
&-\partial_x^{-1}\partial_yE_\phi(x,y).
\end{aligned}
\end{equation}
Using the condition $\|\phi\|_*\leq C\e^2$ and $\|f_2\|_{*****}\leq C$ we get that
\begin{equation}
\label{4-2.9}
|\partial_x^{-1}\partial_yE_\phi(x,y)|\leq C(1+r)^{-\frac32+\delta}\left(1+\e^{-2}\log\frac{1}{\e}\|f\|_{****}\|\phi\|_*+\e^{-\frac72}\|\phi\|_*\right).
\end{equation}
Together with \eqref{4-2.6}, \eqref{4-2.7} and \eqref{4-2.8} we have
\begin{equation}
\label{4-2.10}
\|(1+r)^{\frac32-\delta}\partial_y\mathcal{M}(f_2)\|_{L^\infty(\R^2)}\leq C(1+\e^{-\frac72}\|\phi\|_*).
\end{equation}
Concerning $\partial_y^2\mathcal{M}(f_2),$
\begin{equation}
\label{4-2.11}	
\begin{aligned}
\partial_y^2\mathcal{M}(f_2)=&-\partial_y^2F_0(x,y)\int_x^\infty\frac{E_\phi(s,y)}{F_0(s,y)}ds+2\partial_yF_0(x,y)\int_x^\infty\frac{E_\phi(s,y)\partial_yF_0(s,y)}{F_0^2(s,y)}ds\\
&-2\partial_yF_0(x,y)\int_x^\infty\frac{\partial_yE_\phi(s,y)}{F_0(s,y)}ds+2F_0(x,y)\int_x^\infty\frac{\partial_yE_\phi(x,y)\partial_yF_0(s,y)}{F_0^2(s,y)}ds\\
&+F_0(x,y)\int_x^\infty\frac{E_\phi(s,y)(\partial_y^2F_0(s,y)F_0(s,y)-2(\partial_yF_0(s,y))^2)}{F_0^3(s,y)}ds\\
&-F_0(x,y)\int_x^\infty\frac{\partial_y^2E_\phi(s,y)}{F_0(s,y)}ds.
\end{aligned}	
\end{equation}
For the previous five terms on the right hand side of \eqref{4-2.11}, we could use \eqref{4-2.3} and \eqref{4-2.9} to bound them by
\begin{equation}
\label{4-2.12}
C(1+r)^{-\frac32+\delta}
\left(1+\e^{-2}\log\frac{1}{\e}\|f\|_{*****}\|\phi\|_*+\e^{-\frac72}\|\phi\|_*\right).	
\end{equation}
While for the sixth term, we write it as
\begin{equation}
\label{4-2.13}
\begin{aligned}
F_0(x,y)\int_x^\infty\frac{\partial_y^2E_\phi(s,y)}{F_0(s,y)}ds=~&F_0(x,y)\int_x^\infty\frac{\partial_x^{-1}\partial_y^2E_\phi(s,y)}{F_0^2(s,y)}\partial_xF_0(s,y)ds\\
&-\partial_x^{-1}\partial_y^2E_\phi(x,y).
\end{aligned}
\end{equation}
Using the condition $\|\phi\|_*\leq C\e^2$ and $\|f_2\|_{*****}\leq C$,
\begin{equation}
\label{4-2.14}
|\partial_x^{-1}\partial_y^2E_\phi(x,y)|\leq C(1+r)^{-\frac32+\delta}\left(1+\e^{-4}\log\frac{1}{\e}\|f\|_{*****}\|\phi\|_*+\e^{-\frac{11}2}\|\phi\|_*\right).
\end{equation}
Combined with \eqref{4-2.11}-\eqref{4-2.14} we get that
\begin{equation}
\label{4-2.15}
\|(1+r)^{\frac32-\delta}\partial_y^2\mathcal{M}(f_2)\|_{L^\infty(\R^2)}\leq C(1+\e^{-\frac{11}2}\|\phi\|_*).		
\end{equation}
Using \eqref{4-2.2}, \eqref{4-2.4}, \eqref{4-2.10}, \eqref{4-2.15} we get
\begin{equation}
\label{4.est-1-f2}
\|\mathcal{M}(f_2)\|_{*****}\leq C(1+\e^{-\frac32}\|\phi\|_*).
\end{equation}
It proves \eqref{4.con-1}.
\medskip

Next we shall show that $\mathcal{M}$ is a contraction map in $\widetilde{\mathcal{F}}_{f_2}$. For any $f_{2,1},f_{2,2}\in\widetilde{\mathcal{F}}_{f_2}$, we have
\[\mathcal{M}(f_{2,1})  -\mathcal{M}(f_{2,2})
=-F_0(x,y)\int_x^\infty\frac{H(f_{2,1},f_{2,2})}{F_0(s,y)}ds,\]
where
\begin{align*}
H(f_{2,1},f_{2,2})=-2\phi(f_{2,1}-f_{2,2})-2\e^2f_1g_1(f_{2,1}-f_{2,2})-\e^4g_1(f_{2,1}-f_{2,2})(f_{2,1}+f_{2,2}).
\end{align*}
By the condition $\|\phi\|_*\leq C\e^2$, we could show that
\begin{equation}
\label{4-2.16}
\begin{aligned}
\|(1+r)^{\frac32-\delta}\partial_x^{-1}H(f_{2,1},f_{2,2})\|_{L^\infty(\R^2)}&\leq C\e^\frac32\|f_{2,1}-f_{2,2}\|_{*****},\\
\|(1+r)^{\frac32-\delta}H(f_{2,1},f_{2,2})\|_{L^\infty(\R^2)}&\leq C\e^\frac32\|f_{2,1}-f_{2,2}\|_{*****},\\
\|(1+r)^{\frac32-\delta}\partial_y H(f_{2,1},f_{2,2})\|_{L^\infty(\R^2)}&\leq C\|f_{2,1}-f_{2,2}\|_{*****},\\
\|(1+r)^{\frac32-\delta}\partial_y^2H(f_{2,1},f_{2,2})\|_{L^\infty(\R^2)}&\leq C\e^{-2}\|f_{2,1}-f_{2,2}\|_{*****}.
\end{aligned}
\end{equation}
From the above inequality \eqref{4-2.16} one can prove that
\begin{equation}
\label{4-2.17}
\|\mathcal{M}(f_{2,1})-\mathcal{M}(f_{2,2})\|_{*****}	\leq C\e^{\frac32}\|f_{2,1}-f_{2,2}\|_{*****}.
\end{equation}
Then using the contraction mapping principle we deduce the existence of solution $f_2\in\widetilde{\mathcal{F}}_{f_2}$.

In order to show that $\mathcal{M}(f_2)\in\mathcal{F}_{f_2}$, we have to obtain the corresponding estimation on the following terms $\partial_x^2\mathcal{M}(f_2),~\partial_x^3\mathcal{M}(f_2)$, $\partial_x\partial_y\mathcal{M}(f_2)$ and $\partial_x\partial_y^2\mathcal{M}(f_2)$. Since the argument is almost similar, we shall use $\partial_x\partial_y\mathcal{M}(f_2)$ as an example to explain how to derive the estimation. We have already proved the existence of a solution to \eqref{eq:f3} in $\widetilde{\mathcal{F}}_{f_2}$, we denote the solution by $f_{2}$ and it satisfies \eqref{eq:f3}. We differentiate \eqref{eq:f3} with respect to $y$, we get that
\begin{equation}
\label{4.e.1}
\begin{aligned}
(\sqrt{2}-\e^2)\partial_x\partial_yf_2=&-2q\partial_yf_2-2\partial_yqf_2-2\partial_y\phi f_2-2\phi\partial_yf_2+\partial_y^3(q+\phi)\\
&+\partial_x\partial_y
\left(\frac{\sqrt{2}}{2}\partial_x(q+\phi)-\frac12(q+\phi)^2\right)\\
&-\partial_y\left(\left[\left(\frac{\sqrt{2}}{2}\partial_x(q+\phi)-\frac12(q+\phi)^2\right)+\e^2f_2\right]^2(q+\phi)\right).
\end{aligned}
\end{equation}
Using $\phi\in\mathcal{F}_\phi$ and $f_2\in\widetilde{F}(f_2)$, we see the right hand side of \eqref{4.e.1} is bounded by
\begin{equation}
\label{4.e.2}
C(1+r)^{-\frac32+\delta}\left(1+\e^{-2}\|f_2\|_{*****}+\e^{-2}\|\phi\|_{*}\|f_2\|_{*****}+\e^{-\frac72}\|\phi\|_{*}\right).	
\end{equation}
As a consequence,
\begin{equation}
\label{4.e.3}
\e^2\|(1+r)^{\frac32-\delta}\partial_x\partial_yf_2\|_{L^\infty(\R^2)}\leq C(1+\|f_2\|_{*****}+\e^{-\frac32}\|\phi\|_*).
\end{equation}
Similarly, we can show that
\begin{equation}
\label{4.e.4}
\begin{aligned}
\|(1+r)^{\frac32-\delta}\partial_x^2f_2\|_{L^\infty(\R^2)}&\leq C(1+\|f_2\|_{*****}+\e^{-\frac32}\|\phi\|_*),\\
\e\|(1+r)^{\frac32-\delta}\partial_x^3f_2\|_{L^\infty(\R^2)}&\leq C(1+\|f_2\|_{*****}+\e^{-\frac32}\|\phi\|_*),\\
\e^4\|(1+r)^{\frac32-\delta}\partial_x\partial_y^2f_2\|_{L^\infty(\R^2)}&\leq C(1+\|f_2\|_{*****}+\e^{-\frac32}\|\phi\|_*).
\end{aligned}	
\end{equation}
From \eqref{4.est-1-f2}, \eqref{4.e.3} and \eqref{4.e.4} we have seen the solution $f_2\in\mathcal{F}_{f_2}$.

It remains to prove \eqref{4f-phi}, we shall first show that
\begin{equation}
\label{4.contraction-1}
\|f_{2,\phi_1}-f_{2,\phi_2}\|_{*****}\leq C\e^{-\frac32}\|\phi_1-\phi_2\|_{*}.	
\end{equation}
We notice that the function $f_{2,\phi_1}-f_{2,\phi_2}$ satisfies
\begin{equation}
\label{4.expansion}
\begin{aligned}
&(\sqrt{2}-\e^2)\partial_x(f_{2,\phi_1}-f_{2,\phi_2})
+2q(f_{2,\phi_1}-f_{2,\phi_2})\\
&=-2\phi_1(f_{2,\phi_1}-f_{2,\phi_2})-2(\phi_1-\phi_2)f_{2,\phi_2}+\partial_y^2(\phi_1-\phi_2)\\
&\quad+\frac{\sqrt{2}}{2}\partial_x^2(\phi_1-\phi_2)
-\partial_x(q\phi_1-q\phi_2)-\frac12\partial_x(\phi_1^2-\phi_2^2)\\
&\quad-\left(\frac{\sqrt{2}}{2}\partial_x(q+\phi_1)-\frac12(q+\phi_1)^2\right)^2(q+\phi_1)\\
&\quad+\left(\frac{\sqrt{2}}{2}\partial_x(q+\phi_2)-\frac12(q+\phi_2)^2\right)^2(q+\phi_2)\\
&\quad-\e^2(\sqrt{2}\partial_x(q+\phi_1)-(q+\phi_1)^2)(q+\phi_1)(f_{2,\phi_1}-f_{2,\phi_2})\\
&\quad+\e^2\left(\frac{\sqrt{2}}{2}\partial_x(q+\phi_1)^2-(q+\phi_1)^3-\frac{\sqrt{2}}{2}\partial_x(q+\phi_2)^2+(q+\phi_2)^3\right)f_{2,\phi_2}\\
&\quad-\e^4(q+\phi_1)(f_{2,\phi_1}^2-f_{2,\phi_2}^2)
-\e^4f_{2,\phi_2}^2(\phi_1-\phi_2)\\
&=\Gamma_1(f_{2,\phi_1},f_{2,\phi_2})+\Gamma_2(\phi_1,\phi_2),
\end{aligned}	
\end{equation}
where
\begin{equation*}
\begin{aligned}
\Gamma_1(f_{2,\phi_1},f_{2,\phi_2})=&-2\phi_1(f_{2,\phi_1}-f_{2,\phi_2})-\e^4(q+\phi_1)(f_{2,\phi_1}^2-f_{2,\phi_2}^2)\\
&-\e^2(\sqrt{2}\partial_x(q+\phi_1)-(q+\phi_1)^2)(q+\phi_1)(f_{2,\phi_1}-f_{2,\phi_2}),
\end{aligned}
\end{equation*}
and $\Gamma_2(\phi_1,\phi_2)$ collects the left terms in \eqref{4.expansion}. As \eqref{4-2.16} we get that
\begin{equation}
\label{4.com-1}
\begin{aligned}
\|(1+r)^{\frac32-\delta}\partial_x^{-1}\Gamma_1(f_{2,\phi_1},f_{2,\phi_2})\|_{L^\infty(\R^2)}&\leq C\e^\frac32 \|f_{2,\phi_1}-f_{2,\phi_2}\|_{*****},\\
\|(1+r)^{\frac32-\delta}\Gamma_1(f_{2,\phi_1},f_{2,\phi_2})\|_{L^\infty(\R^2)}&\leq C\e^\frac32 \|f_{2,\phi_1}-f_{2,\phi_2}\|_{*****},\\
\|(1+r)^{\frac32-\delta}\partial_y \Gamma_1(f_{2,\phi_1},f_{2,\phi_2})\|_{L^\infty(\R^2)}&\leq C\log\frac{1}{\e} \|f_{2,\phi_1}-f_{2,\phi_2}\|_{*****},\\
\|(1+r)^{\frac32-\delta}\partial_y^2\Gamma_1(f_{2,\phi_1},f_{2,\phi_2})\|_{L^\infty(\R^2)}&\leq C\e^{-2}\log\frac{1}{\e} \|f_{2,\phi_1}-f_{2,\phi_2}\|_{*****},
\end{aligned}	
\end{equation}
and
\begin{equation}
\label{4.com-2}
\begin{aligned}
\|(1+r)^{\frac32-\delta}\partial_x^{-1}\Gamma_2(\phi_1,\phi_2)\|_{L^\infty(\R^2)}&\leq C\e^{-\frac32}\|\phi_1-\phi_2\|_{*},\\
\|(1+r)^{\frac32-\delta}\Gamma_2(\phi_1,\phi_2)\|_{L^\infty(\R^2)}&\leq C\e^{-\frac32}\|\phi_1-\phi_2\|_{*},\\
\|(1+r)^{\frac32-\delta}\partial_y \Gamma_2(\phi_1,\phi_2)\|_{L^\infty(\R^2)}&\leq C\e^{-\frac72}\|\phi_1-\phi_2\|_{*},\\
\|(1+r)^{\frac32-\delta}\partial_y^2\Gamma_2(\phi_1,\phi_2)\|_{L^\infty(\R^2)}&\leq C\e^{-\frac{11}{2}}\|\phi_1-\phi_2\|_{*},
\end{aligned}	
\end{equation}
Using \eqref{4.com-1} and \eqref{4.com-2} we could argue as from \eqref{eq:f4} to \eqref{4.est-1-f2} to conclude that
\begin{equation}
\label{4.com-3}
\|f_{2,\phi_1}-f_{2,\phi_2}\|_{*****}\leq C\e^\frac32\|f_{2,\phi_1}-f_{2,\phi_2}\|_{*****}+C\e^{-\frac32}\|\phi_1-\phi_2\|_*,	
\end{equation}
and it implies that
\begin{equation}
\label{4.com-4}
\|f_{2,\phi_1}-f_{2,\phi_2}\|_{*****}\leq C\e^{-\frac32}\|\phi_1-\phi_2\|_*.
\end{equation}
Next we use equation \eqref{4.expansion} and argue as \eqref{4.e.1}-\eqref{4.e.4} to derive that
\begin{equation}
\label{4.com-5}
\begin{aligned}
\|(1+r)^{\frac32-\delta}\partial_x^2(f_{2,\phi_1}-f_{2,\phi_2})	\|_{L^\infty(\R^2)}\leq C(\|f_{2,\phi_1}-f_{2,\phi_2}\|_{*****}+\e^{-\frac32}\|\phi_1-\phi_2\|_*),\\
\e\|(1+r)^{\frac32-\delta}\partial_x^3(f_{2,\phi_1}-f_{2,\phi_2})	\|_{L^\infty(\R^2)}\leq C(\|f_{2,\phi_1}-f_{2,\phi_2}\|_{*****}+\e^{-\frac32}\|\phi_1-\phi_2\|_*),\\
\e^2\|(1+r)^{\frac32-\delta}\partial_y\partial_x(f_{2,\phi_1}-f_{2,\phi_2})	\|_{L^\infty(\R^2)}\leq C(\|f_{2,\phi_1}-f_{2,\phi_2}\|_{*****}+\e^{-\frac32}\|\phi_1-\phi_2\|_*),\\
\e^4\|(1+r)^{\frac32-\delta}\partial_y^2\partial_x(f_{2,\phi_1}-f_{2,\phi_2})	\|_{L^\infty(\R^2)}\leq C(\|f_{2,\phi_1}-f_{2,\phi_2}\|_{*****}+\e^{-\frac32}\|\phi_1-\phi_2\|_*).
\end{aligned}
\end{equation}
Together with \eqref{4.com-4} we get \eqref{4f-phi} holds.
\end{proof}

Now we are able to give the proof to Theorem \ref{app}.

\begin{proof}[{Proof of Theorem \ref{app}.}]
With Lemma \ref{Le2.1} we have already seen that for a given function $\phi$, we could find a solution to \eqref{eq:f2}. Therefore, the original problem is equivalent to finding a solution to \eqref{4-2a}. We write \eqref{4-2a} as
\begin{equation}
\label{6.1}
\begin{aligned}
&\partial_x^4\phi-(2\sqrt{2}-\e^2)\partial_x^2\phi-6(\sqrt{2}-\e^2)\partial_x(q\phi)-2\partial_y^2\phi+2\e^2\partial_x^2\partial_y^2\phi+\e^4\partial_y^4\phi\\
&=\hat{P}_1+P_2+P_3,	
\end{aligned}
\end{equation}
where
$$\hat{P}_1=P_1+\Gamma_q+3(\sqrt{2}-\e^2)\partial_x((\partial_x\phi)^2).$$
We shall find a solution of \eqref{6.1} in \eqref{4.fuc} via fixed point argument.

For any $\phi\in\mathcal{F}_\phi$, we first claim that
\begin{equation}
\label{6.2}
\begin{aligned}
\|\hat{P}_1\|_{**}\leq C\e^2.	
\end{aligned}
\end{equation}
To prove the above claim, we have to check all the terms in $\hat{P}_1$ are bounded by $C\e^2$ in the norm of $\|\cdot\|_{**}.$ For example, we use $\partial_x(\partial_x\phi)^2$ and $\e^2\partial_x(\partial_yg_1)^2$ as examples to explain, while the other terms can be handled similarly, we leave the details to the interested reader. For $\partial_x(\partial_x\phi)^2$,
\begin{equation}
\label{6.3}
\begin{aligned}
\|(1+r)^{\frac52-\delta}(\partial_x\phi)^2\|_{L^\infty(\R^2)}&\leq C\|(1+r)^{\frac32-\delta}\partial_x\phi\|^2_{L^\infty(\R^2)}\leq C\|\phi\|_*^2,\\
\|(1+r)^{\frac52-\delta}\partial_x(\partial_x\phi)^2\|_{L^\infty(\R^2)}&\leq C\|\phi\|_*\|(1+r)^{\frac32}\partial_x^2\phi\|_{L^\infty(\R^2)}
\leq C\|\phi\|_*^2,\\
\|(1+r)^{\frac52-\delta}\partial_x^2(\partial_x\phi)^2\|_{L^\infty(\R^2)}&\leq C\|\phi\|_*\|(1+r)^{\frac32}\partial_x^3\phi\|_{L^\infty(\R^2)}
+C\|\phi\|_*^2\\
&\leq C\e^{-\frac12}\|\phi\|_{*}^2.
\end{aligned}
\end{equation}
Thus, we have
\begin{align}
\label{6.4}
\|(\partial_x\phi)^2\|_{**}\leq C\e^2.	
\end{align}
While for the later one, we can check that
\begin{equation}
\label{6.5}
\begin{aligned}
\e^2\|(1+r)^{\frac52-\delta}(\partial_yg_1)^2)\|_{L^\infty(\R^2)}&\leq C\e^2+C\e^\frac32\|\phi\|_*+C\e\|\phi\|_*^2,\\
\e^2\|(1+r)^{\frac52-\delta}\partial_x(\partial_yg_1)^2)\|_{L^\infty(\R^2)}&\leq C\e^2+C\e^\frac32\|\phi\|_*+C\e\|\phi\|_*^2,\\
\e^2\|(1+r)^{\frac52-\delta}\partial_x^2(\partial_yg_1)^2\|_{L^\infty(\R^2)}&\leq C\e^2+C\e^{\frac32}\|\phi\|_*+C\e\|\phi\|_*^2.
\end{aligned}
\end{equation}
It implies that
\begin{equation}
\label{6.6}
\|(\partial_yg_1)^2\|_{**}\leq C\e^{2}.	
\end{equation}
As \eqref{6.2}, we can prove that
\begin{equation}
\label{6.7}
\|P_2\|_{***}\leq C\e^2.	
\end{equation}
Concerning $P_3$, we could show that each term decays faster than $r^{-3}$, then $\partial_x^{-1}P_3$ is well defined. Subsequently we could absorb $P_3$ into $\hat{P}_1$ and prove that $\|\partial_x^{-1}P_3\|_{**}\leq C\e^2$. For example, we shall use the term $g_1\partial_y^2(g_1^2)$ to illustrate this point. It is easy to see that
\begin{equation}
\label{6.8}
\begin{aligned}
\e^2\|(1+r)^{\frac72-\delta}g_1\partial_y^2(g_1^2)\|_{L^\infty(\R^2)}&\leq C\e^2+C\e^\frac12\|\phi\|_*+C\e^{-1}\|\phi\|_*^2+C\e^{-2}\|\phi\|_*^3,\\
\e^2\|(1+r)^{\frac72-\delta}\partial_x(g_1\partial_y^2(g_1^2))\|_{L^\infty(\R^2)}&\leq C\e^2+C\e^{\frac12}\|\phi\|_*+C\e^{-1}\|\phi\|_*^2+C\e^{-2}\|\phi\|_*^3.
\end{aligned}	
\end{equation}
Then one can easily verify that
\begin{equation}
\label{6.9}
\|\partial_x^{-1}(g_1\partial_y^2(g_1^2))\|_{**}\leq C\e^2.	
\end{equation}
Repeating the same process for the other terms in $P_3$ we derive that
\begin{equation}
\label{6.10}
\|\partial_x^{-1}(P_3)\|_{**}\leq C\e^2.
\end{equation}
With \eqref{6.2}, \eqref{6.7} and \eqref{6.10} and Proposition \ref{pr3.2}, we find out a solution (denoted by $\mathcal{N}(\phi)$) to \eqref{6.1} such that
\begin{equation}
\|\mathcal{N}(\phi)\|_{*}\leq C(\|\hat P_1\|_{**}+\|P_2\|_{***}+\|\partial_x^{-1}P_3\|_{**})	\leq C\e^2.
\end{equation}
Therefore we could define a map $\phi\to\mathcal{N}(\phi)$ which sends $\mathcal{F}_\phi$ to itself provided $C$ is large enough.

Next, we shall prove the map $\phi\to\mathcal{N}(\phi)$ is a contraction map. We claim that for any two functions $\phi_1,\phi_2\in\mathcal{F}_\phi$,
\begin{equation}
\label{6.11}
\begin{aligned}
\|\hat{P}_1(\phi_1)-\hat{P}_1(\phi_2)\|_{**}&\leq C\e^{\frac12}\|\phi_1-\phi_2\|_{*},\\
\|P_2(\phi_1)-P_2(\phi_2)\|_{***}&\leq C\e^{\frac12}\|\phi_1-\phi_2\|_{*},\\
\|\partial_x^{-1}P_3(\phi_1)-\partial_x^{-1}P_3(\phi_2)\|_{**}&\leq C\e^{\frac12}\|\phi_1-\phi_2\|_{*}.
\end{aligned}
\end{equation}
We shall pick out one term from $\hat P_1$, $P_2$ and $P_3$ (see \eqref{2.dec-p1}-\eqref{2.dec-p3}) respectively as examples to prove \eqref{6.11}. To stress the dependence of $f_1,f_2,g_1$ on $\phi$, we replace them by $f_{1,\phi},f_{2,\phi},g_{1,\phi}.$ For $\hat P_1$, we take $\e^4\partial_x(f_1f_2)$ into consideration, we have
\begin{equation}
\label{6.e.1}
\begin{aligned}
&\e^4(f_{1,\phi_1}f_{2,\phi_1}-f_{1,\phi_2}f_{2,\phi_2})\\
&=\e^4(f_{1,\phi_1}-f_{1,\phi_2})f_{2,\phi_1}+\e^{4}f_{1,\phi_2}(f_{2,\phi_1}-f_{2,\phi_2})\\
&=\e^4\left(\left(\frac{\sqrt{2}}{2}\partial_x(q+\phi_1)-\frac12(q+\phi_1)^2\right)-\left(\frac{\sqrt{2}}{2}\partial_x(q+\phi_2)-\frac12(q+\phi_2)^2\right)\right)f_{2,\phi_1}\\
&\quad+\e^{4}f_{1,\phi_2}(f_{2,\phi_1}-f_{2,\phi_2})\\
&\leq C\e^4(1+r)^{-\frac52+\delta}\|(1+r)^{\frac32-\delta}\partial_x(\phi_1-\phi_2)\|_{L^\infty(\R^2)}\|(1+r)^{\frac32-\delta}f_{2,\phi_1}\|_{L^\infty(\R^2)}\\
&\quad+C\e^4(1+r)^{-\frac52+\delta}\|(1+r)^{\frac32-\delta}(\phi_1-\phi_2)\|_{L^\infty(\R^2)}\|(1+r)^{\frac32-\delta}f_{2,\phi_1}\|_{L^\infty(\R^2)}\\
&\quad+C\e^4(1+r)^{-\frac52+\delta}\|(1+r)^{\frac32-\delta}(\phi_1^2-\phi_2^2)\|_{L^\infty(\R^2)}\|(1+r)^{\frac32-\delta}f_{2,\phi_1}\|_{L^\infty(\R^2)}\\
&\quad+C\e^4(1+r)^{-\frac52+\delta}(1+\|\phi\|_*)\|(1+r)^{\frac32-\delta}(f_{2,\phi_1}-f_{2,\phi_2})\|_{L^\infty(\R^2)}\\
&\leq C(1+r)^{-\frac52+\delta}\e^{\frac12}\|\phi_1-\phi_2\|_*,
\end{aligned}		
\end{equation}
where we used Lemma \ref{Le2.1}. Similarly, one can also show that
\begin{equation}
\label{6.e.2}
\begin{aligned}
\|(1+r)^{\frac52-\delta}\partial_x(f_{1,\phi_1}f_{2,\phi_1}-f_{1,\phi_2}f_{2,\phi_2})\|_{L^\infty(\R^2)}&\leq C\e^\frac12\|\phi_1-\phi_2\|_*,\\
\|(1+r)^{\frac52-\delta}\partial_x^2(f_{1,\phi_1}f_{2,\phi_1}-f_{1,\phi_2}f_{2,\phi_2})\|_{L^\infty(\R^2)}&\leq C\e^\frac12\|\phi_1-\phi_2\|_*.
\end{aligned}
\end{equation}
Using \eqref{6.e.1} and \eqref{6.e.2} we have
\begin{equation}
\label{6.e.3}
\e^4\|f_{1,\phi_1}f_{2,\phi_1}-f_{1,\phi_2}f_{2,\phi_2}\|_{**}\leq C\e^\frac12\|\phi_1-\phi_2\|_*.
\end{equation}
For $P_2$, we consider $\e^2\partial_y(\partial_xg_1\partial_yg_1)$, by direct computation, we have
\begin{equation}
\label{6.e.4}
\begin{aligned}
&|\partial_xg_{1,\phi_1}\partial_yg_{1,\phi_1}-\partial_xg_{1,\phi_2}\partial_yg_{1,\phi_2}|\\
&=|\partial_x(q+\phi_1)\partial_y(q+\phi_1)-\partial_x(q+\phi_2)\partial_y(q+\phi_2)|\\	
&\leq|\partial_xq\partial_y(\phi_1-\phi_2)|+|\partial_yq\partial_x(\phi_1-\phi_2)|+|\partial_x\phi_1\partial_y\phi_1-\partial_x\phi_2\partial_y\phi_2|\\
&\leq C\e^{-\frac12}(1+r)^{-3+\delta}\|\phi_1-\phi_2\|_*(1+\|\phi_1\|_*+\|\phi_2\|_*)\\
&\leq C\e^{-\frac12}(1+r)^{-3+\delta}\|\phi_1-\phi_2\|_*.
\end{aligned}	
\end{equation}
In a similar way as \eqref{6.e.4}, one can check that
\begin{equation}
\label{6.e.5}
\begin{aligned}
\e^2\|(1+r)^{3-\delta}\partial_y(\partial_xg_{1,\phi_1}\partial_yg_{1,\phi_1}-\partial_xg_{1,\phi_2}\partial_yg_{1,\phi_2})\|_{L^\infty(\R^2)}	
&\leq C\e^{\frac12}\|\phi_1-\phi_2\|_*,\\
\e^2\|(1+r)^{3-\delta}\partial_x\partial_y(\partial_xg_{1,\phi_1}\partial_yg_{1,\phi_1}-\partial_xg_{1,\phi_2}\partial_yg_{1,\phi_2})\|_{L^\infty(\R^2)}	
&\leq C\e^{\frac12}\|\phi_1-\phi_2\|_*.	
\end{aligned}	
\end{equation}
As a consequence of \eqref{6.e.4} and \eqref{6.e.5} we get
\begin{equation}
\label{6.e.6}
\e^2\|\partial_xg_{1,\phi_1}\partial_yg_{1,\phi_1}-\partial_xg_{1,\phi_2}\partial_yg_{1,\phi_2}\|_{***}\leq C\e^{\frac12}\|\phi_1-\phi_2\|_*.	
\end{equation}
While for $P_3,$ we study the term $\e^4\partial_yg_1\partial_y(g_1^2)$,
\begin{equation}
\label{6.e.7}
\begin{aligned}
&\|(1+r)^{\frac72-\delta}\left(\partial_yg_{1,\phi_1}\partial_y(g_{1,\phi_1}^2)-\partial_yg_{1,\phi_2}\partial_y(g_{1,\phi_2}^2)\right)\|_{L^\infty(\R^2)}\\		
&=\|(1+r)^{\frac72-\delta}(\partial_y(q+\phi_1)\partial_y((q+\phi_1)^2)-\partial_y(q+\phi_2)\partial_y((q+\phi_2)^2))\|_{L^\infty(\R^2)}\\
&\leq C\|(1+r)^{\frac72-\delta}(\partial_yq)^2(\phi_1-\phi_2)\|_{L^\infty(\R^2)}\\
&\quad+C\|(1+r)^{\frac72-\delta}q\partial_yq\partial_y(\phi_1-\phi_2)\|_{L^\infty(\R^2)}\\
&\quad+C\|(1+r)^{\frac72-\delta}\partial_yq(\partial_y\phi_1\phi_1-\partial_y\phi_2\phi_2)\|_{L^\infty(\R^2)}\\
&\quad+C\|(1+r)^{\frac72-\delta}q((\partial_y\phi_1)^2-(\partial_y\phi_2)^2)\|_{L^\infty(\R^2)}\\
&\quad+C\|(1+r)^{\frac72-\delta}(\phi_1(\partial_y\phi_1)^2-\phi_2(\partial_y\phi_2)^2)\|_{L^\infty(\R^2)}\\
&\leq C\e^{-\frac12}\|\phi_1-\phi_2\|_{*}+C\e^{-1}\|\phi_1-\phi_2\|^2_{*}
+C\e^{-1}\|\phi_1-\phi_2\|_{*}^3,
\end{aligned}
\end{equation}
and
\begin{equation}
\label{6.e.8}
\begin{aligned}
&\|(1+r)^{\frac52-\delta}\partial_x\left(\partial_yg_{1,\phi_1}\partial_y(g_{1,\phi_1}^2)-\partial_yg_{1,\phi_2}\partial_y(g_{1,\phi_2}^2)\right)\|_{L^\infty(\R^2)}\\		
&=\|(1+r)^{\frac52-\delta}\partial_x(\partial_y(q+\phi_1)\partial_y((q+\phi_1)^2)-\partial_y(q+\phi_2)\partial_y((q+\phi_2)^2))\|_{L^\infty(\R^2)}\\
&\leq C\|(1+r)^{\frac52-\delta}\partial_x((\partial_yq)^2)(\phi_1-\phi_2)\|_{L^\infty(\R^2)}\\
&\quad+C\|(1+r)^{\frac52-\delta}(\partial_yq)^2\partial_x(\phi_1-\phi_2)\|_{L^\infty(\R^2)}\\
&\quad+C\|(1+r)^{\frac52-\delta}\partial_x(q\partial_yq)\partial_y(\phi_1-\phi_2)\|_{L^\infty(\R^2)}\\
&\quad+C\|(1+r)^{\frac52-\delta}q\partial_yq\partial_x\partial_y(\phi_1-\phi_2)\|_{L^\infty(\R^2)}\\
&\quad+C\|(1+r)^{\frac52-\delta}\partial_x(\partial_yq)(\partial_y\phi_1\phi_1-\partial_y\phi_2\phi_2)\|_{L^\infty(\R^2)}\\
&\quad+C\|(1+r)^{\frac52-\delta}\partial_yq\partial_x(\partial_y\phi_1\phi_1-\partial_y\phi_2\phi_2)\|_{L^\infty(\R^2)}\\
&\quad+C\|(1+r)^{\frac52-\delta}\partial_xq((\partial_y\phi_1)^2-(\partial_y\phi_2)^2)\|_{L^\infty(\R^2)}\\
&\quad+C\|(1+r)^{\frac52-\delta}q\partial_x((\partial_y\phi_1)^2-(\partial_y\phi_2)^2)\|_{L^\infty(\R^2)}\\
&\quad+C\|(1+r)^{\frac52-\delta}\partial_x(\phi_1(\partial_y\phi_1)^2-\phi_2(\partial_y\phi_2)^2)\|_{L^\infty(\R^2)}\\
&\leq C\e^{-\frac12}\|\phi_1-\phi_2\|_{*}+C\e^{-1}\|\phi_1-\phi_2\|^2_{*}
+C\e^{-1}\|\phi_1-\phi_2\|_{*}^3.
\end{aligned}
\end{equation}
From \eqref{6.e.7} ad \eqref{6.e.8} we derive that
\begin{equation}
\label{6.e.9}
\e^4\|\partial_x^{-1}\left(\partial_yg_{1,\phi_1}\partial_y((g_{1,\phi_1}^2))-\partial_yg_{1,\phi_2}\partial_y((g_{1,\phi_2}^2))\right)\|_{**}
\leq C\e^{\frac12}\|\phi_1-\phi_2\|_*.
\end{equation}
The other terms in $\hat{P}_1,~P_2$ and $P_3$ can be treated in a similar way and we leave the computations to the interested reader.

Using the claim \eqref{6.11} and Proposition \ref{pr3.2}, we derive that
\begin{equation}
\label{6.12}
\|\mathcal{N}(\phi_1)-\mathcal{N}(\phi_2)\|_{*}\leq C\e^{\frac12}\|\phi_1-\phi_2\|_{*}.	
\end{equation}
Then we deduce the existence of a solution $\phi$ to \eqref{6.1} by contraction principle. Thereby the original existence result is established.
\end{proof}
\bigskip

\begin{center}
	\textbf{Acknowledgement }
\end{center}

Y. Liu is partially supported by NSFC No. 11971026
and The Fundamental Research Funds for the Central Universities WK3470000014. Z. Wang is partially supported by NSFC No.11871386 and NSFC No.11931012. J. Wei is partially supported by NSERC of Canada. W. Yang is partially supported by NSFC No.11801550, No.11871470 and No.12171456.

\end{document}